\documentclass[reqno,12pt,a4paper]{amsart}

\usepackage[utf8]{inputenc}

\usepackage{enumerate}
\usepackage{amstext,amsmath,amsthm,amsfonts,amssymb,amscd}
\usepackage{latexsym,mathrsfs,dsfont,euscript}
\usepackage{mathbbol}
\usepackage{mathtools}
\usepackage{esint}

\usepackage{xcolor}

\usepackage{fullpage}

\usepackage{hyperref} 
\hypersetup{
  colorlinks,%
  citecolor=blue,%
    filecolor=red,%
    linkcolor=red,%
    urlcolor=blue
}
\usepackage{graphicx}

\usepackage[font=footnotesize,labelfont=it]{caption}

\newtheorem{theorem}{Theorem}[section]
\newtheorem{proposition}[theorem]{Proposition}
\newtheorem{lemma}[theorem]{Lemma}
\newtheorem{corollary}[theorem]{Corollary}
\theoremstyle{definition}
\newtheorem{definition}{Definition}[section]

\theoremstyle{remark}

\numberwithin{equation}{section}

\newcommand{\abs}[1]{\left\vert#1\right\vert}

\newcommand{\proin}[2]{\left<#1,#2\right>}

\allowdisplaybreaks

\begin{document}

\title[Fisher-Riemann geometry for nonparametric probability densities]{Fisher-Riemann geometry for nonparametric probability densities}


\author[]{Hugo Aimar}
\author[]{An\'ibal Chicco Ruiz}
\author[]{Ivana G\'{o}mez}


%
%

\begin{abstract}
In this article we aim to obtain the Fisher Riemann geodesics for nonparametric families of probability densities as a weak limit of the parametric case with increasing number of parameters.
\end{abstract}

\maketitle

\section{Introduction}\label{sec:introduction}
The Fisher information of a random variable is a useful tool for the estimation of parameters in parametric statistics. When the number of parameters describing the distributions of the random variables we are dealing with is larger than one, the Fisher information induces a remarkable Riemannian structure on the set of parameters. A simple but illustrative case is provided by the one dimensional Gaussian family, see \cite{CoSaStra05}. Let $U=\mathbb{R}^2_{+}=\{\bar{\theta}=(\theta_1,\theta_2): \theta_1\in\mathbb{R}, \theta_2>0\}$ be the upper half plane of $\mathbb{R}^2$. With $\theta_1$ the mean and $\theta_2$ the variance, we have a one to one correspondence between $U$ and the two parametric family $\{X\sim N(\theta_1,\theta_2)\}$. In other words, the random variables that have normal distribution with mean $\theta_1$ and variance $\theta_2$. In this case the metric that the Fisher information induces in $U$ is that of the hyperbolic Poincar\'{e} geometry. Of course the situation extends to any number of parameters and the geodesics that these Riemannian metrics induce in the set of parameters can be thought as trajectories joining a density (image) to another. In other words, we have a Fisher geodesic mass transport. This point of view has been used in some applications to image processing (see \cite{PeRa06}).

In this paper we aim to extend Fisher-Riemann geometry to non parametric sets of densities and to explore the behaviour of the corresponding geodesics. The starting point is the parametric case. Then by approximation of a general density by parametric ones we are able to establish and solve the corresponding Gauss geodesic equations. 

The paper is organized as follows. In Section~\ref{sec:theFisherInformation} we introduce the well known facts regarding Fisher information. Section~\ref{sec:TheRiemannianGeometryInTheSimplex} is devoted to introduce our particular discrete setting and to obtain the precise form of the geodesic equations in this case. Section~\ref{sec:SolutionsOfTheODEthatPreserveTheUnitSpeedGeodesicCondition} contains the result that provides solutions for the geodesic equations both in the continuous and discrete cases. Section~\ref{sec:FisherGeodesicTransportOfnDimensionalDensitiesAndOfTheirDyadicPixelations} contains main result of the paper regarding the convergence of the parametric geodesics to the geodesics corresponding to a given nonparametric density. In Section~\ref{secExamplesAndGraphics} we show examples and illustrations of the dynamics in different situations.
 
\section{The Fisher information}\label{sec:theFisherInformation}
A basic approach to Fisher Information from the point of view of Information Theory can be found in \cite{CoverThomasBook} and also \cite{NielsenNoticeAMS} for a recent account on Information Geometry. We shall roughly follow the lines in \cite{CoverThomasBook} to describe and define the basic concepts. Let $U$ be an open set in $\mathbb{R}^n$. We shall consider $U$ as the set of parameters $\theta=(\theta_1,\ldots,\theta_n)$ for a family $\mathcal{D}_U=\{\varphi(x,\theta):\theta\in U\}$ of probability densities in $\mathbb{R}^k$ with respect to Lebesgue measure $dx$ in $\mathbb{R}^k$. When $k=2$ a typical $\mathcal{D}_U$ is given by 
\begin{equation*}
	\varphi(x,\theta)=\varphi(x,\theta_1,\theta_2,\theta_3)=\frac{1}{\theta_3^2}\varphi\left(\frac{(x_1-\theta_1,x_2-\theta_2)}{\theta_3}\right)
\end{equation*}
with $\int_{\mathbb{R}^2}\varphi dx=1$, $\varphi\geq 0$ and $U=\{\theta\in\mathbb{R}^3: (\theta_1,\theta_2)\in\mathbb{R}^2, \theta_3>0\}$.

Given $U$, $\mathcal{D}_U$ and a random variable $X$ (or random vector) in $\mathbb{R}^k$ we may consider the new random variable $Y(\omega,\theta)=\varphi(X(\omega),\theta)$. The score is defined as a new random vector in $\mathbb{R}^n$ given by the $\theta$-gradient of $\log Y(\omega,\theta)$. Precisely,
\begin{align*}
	\mathcal{S}(\omega,\theta)&= \nabla_\theta\log Y(\omega,\theta)\\
	&= \left(\frac{\partial}{\partial\theta_1}\log Y(\omega,\theta),\ldots,\frac{\partial}{\partial\theta_n}\log Y(\omega,\theta)\right)\\
	&= \frac{1}{Y(\omega,\theta)}\left(\frac{\partial Y}{\partial\theta_1}(\omega,\theta),\ldots,\frac{\partial Y}{\partial\theta_n}(\omega,\theta)\right)\\
	&= \frac{1}{\varphi(X(\omega),\theta)}\left(\frac{\partial}{\partial\theta_1}\varphi(X(\omega),\theta),\ldots,\frac{\partial}{\partial\theta_n}\varphi(X(\omega),\theta)\right)\\
	&= \frac{\nabla_\theta\varphi(X(\omega),\theta)}{\varphi(X(\omega),\theta)}.
	\end{align*}
Of course some simple analytic conditions on the family $\mathcal{D}_U$ of densities are required in order to have a well defined score for the random variables distributed by densities in $\mathcal{D}_U$. First order smoothness in $\theta$ and non vanishing of the densities are the basic ones.
\begin{lemma}
	The expected value of the score $\mathcal{S}(\cdot,\theta)$ with respect to $\varphi(\cdot,\theta)$ vanishes.
\end{lemma} 

\begin{proof}[Sketch of the proof]
The expected value of $\mathcal{S}(\cdot,\theta)$ with respect to the probability measure $\varphi(x,\theta)dx$ in $\mathbb{R}^k$ is given by the $n$-vector
\begin{align*}
	\int_{\mathbb{R}^k} \mathcal{S}(x,\theta)\varphi(x,\theta) dx &= \int_{\mathbb{R}^k} \frac{\nabla_\theta\varphi(x,\theta)}{\varphi(x,\theta)}\varphi(x,\theta) dx\\
	&= \nabla_\theta\left(\int_{\mathbb{R}^k}\varphi(x,\theta) dx\right)\\
	&= \nabla_\theta 1\\
	&= \bar{0} = (0,\ldots,0)\in\mathbb{R}^n.
\end{align*}	
\end{proof}
From de above Lemma, the covariance matrix of the score $\mathcal{S}(\cdot,\theta)$ with respect to the corresponding $\varphi(\cdot,\theta)$ density in $\mathcal{D}_U$ is given by the expected value of the $n\times n$ random matrix $\mathcal{S}(\omega,\theta)\mathcal{S}^T(\omega,\theta)$, where $T$ denotes transposition.

\begin{lemma}
	The $i, j$ entry of the $n\times n$ covariance matrix of the score $\mathcal{S}(\omega,\theta)$ with respect to $\varphi(\cdot,\theta)$ is given by 
	\begin{equation*}
		J_{ij}(\theta) = \int_{\mathbb{R}^k}\frac{\partial}{\partial\theta_i}\log\varphi(x,\theta) \frac{\partial}{\partial\theta_j}\log\varphi(x,\theta)\varphi(x,\theta) dx.
	\end{equation*}
\end{lemma}
\begin{proof}
	Since the $\varphi(\cdot,\theta)$-mean of  $\mathcal{S}(\omega,\theta)$ vanishes, we have that the $\varphi(\cdot,\theta)$ covariance matrix of $\varphi(\cdot,\theta)$ is given by the expected value of $\mathcal{S}(\omega,\theta)\cdot\mathcal{S}^T(\omega,\theta)$ with respect to $\varphi(\cdot,\theta)$, that is
	\begin{align*}
		\int_{\mathbb{R}^k}\mathcal{S}(x,\theta)\cdot\mathcal{S}^T(x,\theta)\varphi(x,\theta) dx
		&= \left(	\int_{\mathbb{R}^k}\mathcal{S}_i(x,\theta)\mathcal{S}_j(x,\theta)\varphi(x,\theta) dx\right)_{ij}\\
		&= \left(\int_{\mathbb{R}^k}\frac{\partial}{\partial\theta_i}\log\varphi(x,\theta) \frac{\partial}{\partial\theta_j}\log\varphi(x,\theta)\varphi(x,\theta) dx\right)_{ij}.
	\end{align*}
\end{proof}
Notice that all the above considerations can be done in any positive $\sigma$-finite measure space $(X,\mu)$ instead of $(\mathbb{R}^k,dx)$. 
\begin{definition}
	Let $U$ be an open set in $\mathbb{R}^n$. Let $(X,\mu)$ be a positive $\sigma$-finite measure space. Let $\mathcal{D}_U=\{\varphi(\cdot,\theta):\, \varphi(\cdot,\theta):X\to\mathbb{R}^+, \int_X \varphi(x,\theta)d\mu(x)=1, \theta\in U\}$ be a family of densities parametrized on $U$. The \textbf{Fisher Information Matrix} is the $n\times n$ matrix valued function $J$ defined in $U$ by 
	\begin{equation*}
		J_{ij}(\theta) = \int_{X}\frac{\partial}{\partial\theta_i}\log\varphi(x,\theta) \frac{\partial}{\partial\theta_j}\log\varphi(x,\theta)\varphi(x,\theta) d\mu(x),
	\end{equation*}
provided the smoothness and integrability required.
\end{definition}
In our further analysis we shall be concerned with special cases of dimensionaly increasing open sets $U$ and their eventual convergence as $n$ tends to infinity when the family $\mathcal{D}_U$ is drawn by projection of a general density.

Let us finish this section with the explicit computation of the matrix $J$ in the case of $X$ finite and $\mu$ the counting measure when every density is parametrized by its values. Let us precise the situation. Let $X=\{x_1,\ldots,x_{n},x_{n+1}\}$ be a finite set and $\mu=\sum_{i=1}^{n+1}\delta_{x_i}$, with the unit mass at the point $x_i$. Notice that if $g:X\to\mathbb{R}^+$ is a density in the sense that $\int_X g(x)d\mu(x)=1$, we have that $\sum_{i=1}^{n+1} g(x_i)=1$. Hence every such $g$ can be parametrized by its $n$-first values $g(x_1),\ldots,g(x_n)$, since $g(x_{n+1})$ can only be $1-\sum_{i=1}^{n} g(x_i)$. Let us write this remark in terms of the above introduced notation. Set $U=\{\theta=(\theta_1,\ldots,\theta_n)\in\mathbb{R}^n: \theta_i>0 \textrm{ for every }i \textrm{ and } \sum_{i=1}^n\theta_i<1\}$. Then the set $\mathcal{D}=\{g:X\to\mathbb{R}^+; g \textrm{ non vanishing density with respect to }\mu\}$ coincides with $\mathcal{D}_U$ the set of all parametric densities $\varphi(\cdot,\theta)$, $\theta\in U$ given by $\varphi(x_j,\theta)=\theta_j$; $j=1,2,\ldots,n$  and $\varphi(x_{n+1},\theta)=1-\sum_{j=1}^n\theta_j$.
\begin{proposition}\label{propo:matrixFisherInformation}
Let $X=\{x_1,\ldots,x_{n},x_{n+1}\}$ and $U$ be the $n$-dimensional open simplex $\{\theta=(\theta_1,\ldots,\theta_n)\in\mathbb{R}^n: \theta_i>0 \textrm{ for every }i=1,\ldots,n \textrm{ and } \sum_{i=1}^n\theta_i<1\}$. Let $\mathcal{D}_U=\{\varphi(\cdot,\theta): \theta\in U\}$, where $\varphi(\cdot,\theta):X\to\mathbb{R}$ is given by $\varphi(x_j,\theta)=\theta_j$ for $j=1,\ldots,n$ and $\varphi(x_{n+1},\theta)=1-\sum_{j=1}^n\theta_j$. Then, the Fisher information matrix associated to $\mathcal{D}_U$ is the $n\times n$ matrix function defined in $U$ by
	\begin{equation*}
		J(\theta)=J(\theta_1,\ldots,\theta_n)=\frac{1}{\theta_{n+1}}\bar{\bar{1}} + D\Bigl(\frac{1}{\theta_j}\Bigr),
	\end{equation*}
where $\theta_{n+1}=1-\sum_{i=1}^n\theta_i>0$, $\bar{\bar{1}}$ is the $n\times n$ matrix with its $n^2$ entries equal to one and $ D\bigl(\frac{1}{\theta_j}\bigr)$ is the $n\times n$  diagonal matrix
\begin{equation*}
\begin{pmatrix}
	\frac{1}{\theta_1} & 0 & \ldots & 0\\
	0 & \frac{1}{\theta_2} & \ldots & 0\\
	\vdots& \vdots & \ddots & \vdots \\
	0 & 0 & \ldots & \frac{1}{\theta_n}
\end{pmatrix}.	
\end{equation*}
\end{proposition}
\begin{proof}
In order to compute the entries $J_{ij}(\theta)$ of $J(\theta)$, let us start by the calculation  of the partial derivatives  $\frac{\partial}{\partial\theta_j}\log \varphi(x,\theta)$ for every $x\in X=\{x_1,\ldots,x_n,x_{n+1}\}$ and every $\theta=(\theta_1,\ldots,\theta_n)\in U$. Notice first that $\log \varphi(x,\theta)$ can be written as 
\begin{align*}
	\log \varphi(x,\theta)&= \sum_{i=1}^n\log\theta_i\mathds{1}_{\{x_i\}}(x) + \log\left(1-\sum_{k=1}^n\theta_k\right)\mathds{1}_{\{x_{n+1}\}}(x)\\
	&= \sum_{i=1}^{n+1}\log\theta_i\mathds{1}_{\{x_i\}}(x),
\end{align*}
with $\theta_{n+1}=\theta_{n+1}(\theta_1,\ldots,\theta_n)=1-\sum_{k=1}^n\theta_k$ and, as usual $\mathds{1}_E(x)$ the indicator function of the set $E\subset X$. Hence, for $j=1,2,\ldots,n$; $\frac{\partial}{\partial\theta_j}\log \varphi(x,\theta)=\frac{1}{\theta_j}\mathds{1}_{\{x_j\}}(x)-\frac{1}{\theta_{n+1}}\mathds{1}_{\{x_{n+1}\}}(x)$. So that 
\begin{align*}
	J_{jj}(\theta) &= \int_X \left(\frac{\partial}{\partial\theta_j}\log \varphi(x,\theta)\right)^2\varphi(x,\theta) d\mu(x)\\
	&= \int_X\left(\frac{1}{\theta_j^2}\mathds{1}_{\{x_j\}}(x)+\frac{1}{\theta_{n+1}^2}\mathds{1}_{\{x_{n+1}\}}(x)\right)\varphi(x,\theta) d\mu(x)\\
	&= \sum_{l=1}^{n+1}\left(\frac{1}{\theta_j^2}\mathds{1}_{\{x_j\}}(x_l)+\frac{1}{\theta_{n+1}^2}\mathds{1}_{\{x_{n+1}\}}(x_l)\right)\theta_l\\
	&= \frac{1}{\theta_j}+\frac{1}{\theta_{n+1}}.
\end{align*}
On the other hand, for $i\neq j$, $i, j=1,2,\ldots,n$, we have
\begin{align*}
	J_{ij}(\theta) &= \int_X \left(\frac{\partial}{\partial\theta_i}\log \varphi(x,\theta)\frac{\partial}{\partial\theta_j}\log \varphi(x,\theta)\right)\varphi(x,\theta) d\mu(x)\\
	&= \int_X\left(\frac{1}{\theta_i}\mathds{1}_{\{x_i\}}(x)-\frac{1}{\theta_{n+1}}\mathds{1}_{\{x_{n+1}\}}(x)\right) \left(\frac{1}{\theta_j}\mathds{1}_{\{x_j\}}(x)-\frac{1}{\theta_{n+1}}\mathds{1}_{\{x_{n+1}\}}(x)\right)\varphi(x,\theta) d\mu(x)\\
	&= \int_X \frac{1}{\theta_{n+1}^2}\mathds{1}_{\{x_{n+1}\}}(x)\varphi(x,\theta) d\mu(x)\\
	&= \frac{1}{\theta_{n+1}}.
\end{align*}
In other words, for $i, j=1,2,\ldots,n$ we have
\begin{equation*}
	J_{ij}(\theta) = \frac{\delta_{ij}}{\theta_j} + \frac{1}{\theta_{n+1}},
\end{equation*}
where  $\delta_{ij}$ the Kronecker notation for the identity matrix. In explicit $n\times n$ matrix form
\begin{equation*}
	J(\theta)=
	\begin{pmatrix}
		\frac{1}{\theta_1}+\frac{1}{\theta_{n+1}} & \frac{1}{\theta_{n+1}} & \cdots & \frac{1}{\theta_{n+1}}\\
		\frac{1}{\theta_{n+1}} & \frac{1}{\theta_2}+\frac{1}{\theta_{n+1}} & \cdots & \frac{1}{\theta_{n+1}}\\
		\vdots& \vdots & \ddots & \vdots \\
		\frac{1}{\theta_{n+1}} & \frac{1}{\theta_{n+1}} & \cdots & \frac{1}{\theta_n}+\frac{1}{\theta_{n+1}}
	\end{pmatrix}.
\end{equation*}
\end{proof}

\section{The Riemannian geometry in the simplex induced by $J(\theta)=\theta_{n+1}^{-1}\bar{\bar{{1}}}+D(\theta_j^{-1})$}\label{sec:TheRiemannianGeometryInTheSimplex}

Let $U$ be the open simplex $U=\{\theta\in\mathbb{R}^n:\theta_i>0, i=1,\ldots,\textcolor{red}{n+1}\}$ with $\theta_{n+1}=1-\sum_{i=1}^n \theta_i$, as before. Let $J(\theta)$ be the matrix function defined in $U$ by the Fisher Information obtained in Proposition~\ref{propo:matrixFisherInformation}. Explicitly,
\begin{equation*}
	J(\theta) = \frac{1}{\theta_{n+1}} \bar{\bar{1}}+D\Bigl(\frac{1}{\theta_j}\Bigr).
\end{equation*}
\begin{theorem}\label{thm:arcLengthFisherInformation}
	For each $\theta\in U$ the quadratic form in $\mathbb{R}^n\times\mathbb{R}^n$ given by $\proin{v}{w}_{J(\theta)}= \sum_{i=1}^n\sum_{j=1}^n v_iJ_{ij}(\theta)w_j$ defines a scalar product in $\mathbb{R}^n$. The couple $(U,\proin{v}{w}_{J(\theta)})$ is a $\mathscr{C}^\infty$ Riemannian manifold of dimension $n$. Moreover, if $\boldsymbol{\theta}:[0,1]\to U$ is a smooth curve in $U$, $\boldsymbol{\theta}(t)=(\theta_1(t),\ldots,\theta_n(t))$, then, with $\theta_{n+1}(t)=1-\sum_{i=1}^n \theta_i(t)$, the arc length is given by
	\begin{equation*}
		l_J(\boldsymbol{\theta}) = \int_0^1\sqrt{\proin{\dot{\boldsymbol{\theta}}(t)}{\dot{\boldsymbol{\theta}}(t)}_{J(\boldsymbol{\theta}(t))}}dt = \int_0^1\sqrt{\sum_{j=1}^{n+1}\frac{(\dot{\theta}_j(t))^2}{\theta_j(t)}} dt,
	\end{equation*}
where $\dot{\boldsymbol\theta}(t)=(\dot{\theta}_1(t),\ldots,\dot{\theta}_n(t))$ is the velocity of $\boldsymbol{\theta}$ at each time $t\in [0,1]$.
\end{theorem}
\begin{proof}
	Since $J(\boldsymbol{\theta})$ is symmetric, we only need to show that it is also positive definite and to provide an explicit formula for $\proin{v}{v}_{J(\boldsymbol{\theta})}$. Take $v\in\mathbb{R}^n$, $v\neq 0$, then
	\begin{align*}
		v^T J(\boldsymbol{\theta})v &= v^T\left(\frac{1}{\theta_{n+1}} \bar{\bar{1}}+D\Bigl(\frac{1}{\theta_j}\Bigr)\right)v\\
		&= v^T\frac{1}{\theta_{n+1}} \bar{\bar{1}}v + v^TD\Bigl(\frac{1}{\theta_j}\Bigr)v\\
		&= \frac{1}{\theta_{n+1}}\left(\sum_{i=1}^n v_i\right)^2 + \sum_{i=1}^n \frac{v_i^2}{\theta_i}>0.
	\end{align*}
In order to prove the formula for the length of the curve $\boldsymbol{\theta}(t)$ we only have to observe that
$\frac{1}{\theta_{n+1}(t)}\left(\sum_{i=1}^n \dot{\theta}_i(t)\right)^2 + \sum_{i=1}^n \frac{\dot{\theta}_i^2(t)}{\theta_i(t)}=\frac{1}{\theta_{n+1}(t)}\left(\dot{\theta}_{n+1}(t)\right)^2 + \sum_{i=1}^n \frac{\dot{\theta}_i^2(t)}{\theta_i(t)}=\sum_{i=1}^{n+1} \frac{\dot{\theta}_i^2(t)}{\theta_i(t)}$.
\end{proof}
The following result is a direct consequence of Theorem~\ref{thm:arcLengthFisherInformation}.
\begin{corollary}
	Set $l(\boldsymbol{\theta})$ to denote the standard length of the curve $\boldsymbol{\theta}(t)$. Then
	\begin{enumerate}[(i)]
		\item $l(\boldsymbol{\theta})\leq l_J(\boldsymbol{\theta})$ for every $\boldsymbol{\theta}$;
		\item for every compact subset $V$ of $U$, there exists $C>0$ depending only on $V$ such that $l_J(\boldsymbol{\theta})\leq C l(\boldsymbol{\theta})$ for every $\boldsymbol{\theta}$ contained in $V$.
	\end{enumerate}
\end{corollary}

We shall now find the geodesic ODE system for the Riemannian manifold $(U,\proin{}{}_J)$. The following lemmas will be of help at writing the Christoffel symbols in our setting.
\begin{lemma}\label{lemma:productsymmetricmatrix}
Let $A$ be the $n\times n$ real matrix given by $A= \bar{\bar{1}} + D(c_i)$, where $D(c_i)$ is the diagonal $n\times n$ matrix with $c_i>0$, $i=1,2,\ldots,n$. Then $A^{-1}=(\textrm{det }A)^{-1} K$, where $		K_{ii}=\prod_{j\neq i}c_j + \sum_{l\neq i}\Bigl(\prod_{\substack{m\neq l\\m\neq i}}c_m\Bigr)$
and $K_{ij}= -\prod_{\substack{m\neq i\\ m\neq j}} c_m$, for $i\neq j$ and $det\, A=\prod_{i=1}^n c_i + \sum_{i=1}^n \prod_{j\neq i}^n c_j$.
\end{lemma}
\begin{proof}
Let us first prove inductively that $det\, A=\prod_{i=1}^n c_i + \sum_{i=1}^n \prod_{j\neq i}^n c_j$ for $c_i\geq 0$. Case $n=2$ is simple: $det\,\begin{pmatrix}
	1+c_1 & 1\\ 1 & 1+c_2
\end{pmatrix}=(1+c_1)(1+c_2)-1=c_1 c_2+c_1+c_2$, as desired. Consider the $n\times n$ matrix $A=\bar{\bar{1}}+D(\bar{c})$ and suppose that the formula holds for the case of $A$ of dimension  $(n-1)\times (n-1)$. Hence
\begin{equation*}
	det\, A = \sum_{i=1}^n a_{1i}(-1)^{1+i}M_{1i} = (1+c_1)M_{11}+\sum_{i=2}^{n}(-1)^{1+i}M_{1i},
\end{equation*}
where $M_{1i}= det\,(A_{1i})$ and $A_{1i}$ is the $(n-1)$-submatrix of $A$ obtained by removing the first row and the $i$-th column of $A$. Observe that $A_{11}$ is a $(n-1)\times (n-1)$ matrix of type $\bar{\bar{1}}+D(\bar{c})$, so that $M_{11}=\prod_{i=2}^n c_i + \sum_{i=2}^n\prod_{\substack{j\neq i\\ j\neq 1}} c_j$. Also, $A_{12}$ is of type $\bar{\bar{1}}+D(\bar{c})$, but with diagonal $(0,c_3,\ldots,c_n)$, then again by inductive hypothesis, $M_{12}=\prod_{\substack{j\neq 1\\ j\neq 2}} c_j$. On the other hand, $A_{13}$ is of type $\bar{\bar{1}}+D(\bar{c})$ up to one row permutation, with diagonal $(0,c_2,c_4,\ldots,c_n)$, obtaining $M_{13}=-\prod_{\substack{j\neq 1\\ j\neq 3}} c_j$. In general, for $i=2,\ldots,n$, the submatrix $A_{1i}$ is of type $\bar{\bar{1}}+D(\bar{c})$ with diagonal $(0,c_2,\ldots,c_{i-1},c_{i+1},\ldots,c_n)$ up to $i-2$ row permutations, which gives $M_{1i}=(-1)^{i-2}\prod_{\substack{j\neq i\\ j\neq 1}} c_j$. Then
\begin{align*}
	det\,A &= (1+c_1)\Bigl(\prod_{i=2}^n c_i + \sum_{i=2}^n\prod_{\substack{j\neq i\\ j\neq 1}} c_j\Bigr) + \sum_{i=2}^n(-1)\prod_{\substack{j\neq i\\ j\neq 1}} c_j\\
	&= \prod_{i=2}^n c_i + \prod_{i=1}^n c_i + \sum_{i=2}^n\prod_{j\neq i} c_j\\
	&= \prod_{i=1}^n c_i + \sum_{i=1}^n\prod_{j\neq i}c_j.
\end{align*}

Let us now prove that $(det\, A)^{-1}K$ is the inverse of $A$. Since both, $A$ and $K$ are symmetric, we only need to compute $KA$. Take first $i\neq j$, then
\begin{align*}
(KA)_{ij} &= \sum_{l=1}^n K_{il}A_{lj}\\
&= K_{ii}A_{ij}+K_{ij}A_{jj}-\sum_{\substack{l\neq i\\ l\neq j}}\Bigl(\prod_{\substack{m\neq i\\ m\neq l}}c_m\Bigr)\\
&= \prod_{p\neq i}c_p + \sum_{l\neq i}\Bigl(\prod_{\substack{m\neq l\\ m\neq i}}c_m\Bigr)-\Bigl(\prod_{\substack{m\neq i\\ m\neq j}}c_m\Bigr)(1+c_j)-\sum_{\substack{l\neq i\\ l\neq j}}\Bigl(\prod_{\substack{m\neq i\\ m\neq l}} c_m\Bigr)\\
&= \prod_{p\neq i} c_p + \Bigl[\sum_{\substack{l\neq i\\ l\neq j}}\prod_{\substack{m\neq l\\ m\neq i}} c_m + \prod_{\substack{m\neq j\\ m\neq i}} c_m\Bigr] - \prod_{\substack{m\neq i\\ m\neq j}} c_m - \prod_{m\neq i} c_m - \sum_{\substack{l\neq i\\ l\neq j}}\Bigl(\prod_{\substack{m\neq i\\ m\neq l}}c_m\Bigr)\\
&= 0.
\end{align*}
On the other hand,
\begin{align*}
(KA)_{ii} &= \sum_{l=1}^n K_{il} A_{li}\\
&= K_{ii}A_{ii}+\sum_{l\neq i} K_{il}\\
&=\Bigl(\prod_{j\neq i} c_j + \sum_{l\neq i}\Bigl(\prod_{\substack{m\neq l\\ m\neq i}} c_m\Bigr)\Bigr) (1+c_i) - \sum_{l\neq i}\Bigl(\prod_{\substack{m\neq i\\ m\neq l}} c_m\Bigr)\\
&= \prod_{j\neq i} c_j + \sum_{l\neq i}\Bigl(\prod_{\substack{m\neq l\\ m\neq i}} c_m\Bigr) + \prod_{j} c_j + \sum_{l\neq i}\Bigl(\prod_{m\neq l} c_m\Bigr) - \sum_{l\neq i}\Bigl(\prod_{\substack{m\neq i\\ m\neq l}} c_m\Bigr)\\
&=  \prod_{j} c_j + \sum_{l} \Bigl(\prod_{m\neq l} c_m\Bigr)\\
&= \textrm{det} A,
\end{align*}
for every $i=1,\ldots,n$. Then $(\textrm{det}A)^{-1}K$ is the inverse of $A$.
\end{proof}

Let us now apply the above lemma in order to obtain the inverse of the Riemannian matrix $J(\theta)$ for $\theta\in U$, which as usual we shall denote by $J=(g_{ij})$.
\begin{lemma}\label{lemma:inverseRiemannianMatrixJ}
For $\theta\in U$, $\theta=(\theta_1,\ldots,\theta_n)$, set $\theta_{n+1}=1-\sum_{i=1}^n\theta_i$. Then, the inverse of the metric matrix $J(\theta)$, $J^{-1}(\theta)=(g^{ij}(\theta))_{i,j=1,\ldots,n}$ is given by $g^{ii}=\theta_i(1-\theta_i)$, $i=1,\ldots,n$ and $g^{ij}=-\theta_i\theta_j$ for $i\neq j$. Or, in terms of the Kronecker delta, $g^{ij}=\theta_i(\delta_{ij}-\theta_j)$.
\end{lemma}
\begin{proof}
Since $J(\theta)=\frac{1}{\theta_{n+1}}\bar{\bar{1}} +D(\tfrac{1}{\theta_j})$, we have that $\theta_{n+1} J(\theta) = \bar{\bar{1}} + D(\tfrac{\theta_{n+1}}{\theta_j})$ and we can apply Lemma~\ref{lemma:productsymmetricmatrix} with $c_j=\tfrac{\theta_{n+1}}{\theta_j}$. Set $M(\theta)=\prod_{j=1}^{n+1}\frac{\theta_{n+1}}{\theta_j}$. Then, with $A=\theta_{n+1} J(\theta)$, we have
\begin{align*}
\textrm{det }A &= M(\theta) + \sum_{l=1}^{n}\frac{\theta_l}{\theta_{n+1}} M(\theta)\\
&= M(\theta)\biggl(1+\frac{1}{\theta_{n+1}}\sum_{l=1}^{n}\theta_l\biggr) \\
&=M(\theta)\biggl(1+\frac{1-\theta_{n+1}}{\theta_{n+1}}\biggr)\\
&= \frac{M(\theta)}{\theta_{n+1}}.
\end{align*}
On the other hand,
	\begin{align*}
		K_{ii} &= \prod_{j\neq i}\frac{\theta_{n+1}}{\theta_j} + \sum_{l\neq i}\biggl(\prod_{\substack{m\neq l\\ m\neq i}}\frac{\theta_{n+1}}{\theta_m}\biggr)\\
		&= \frac{\theta_i}{\theta_{n+1}} M(\theta) + \sum_{l\neq i}\frac{\theta_l}{\theta_{n+1}}\frac{\theta_i}{\theta_{n+1}} M(\theta)\\
		&= \frac{M(\theta)}{\theta_{n+1}}\biggl(\theta_i +\frac{\theta_i}{\theta_{n+1}}\sum_{l\neq i}\theta_l\biggr)\\
		&= \frac{M(\theta)}{\theta_{n+1}}\theta_i\biggl(1+\frac{1-\theta_{n+1}-\theta_i}{\theta_{n+1}}\biggr)\\
		&= \frac{M(\theta)}{\theta_{n+1}}\theta_i \biggl(\frac{\theta_{n+1} + 1 -\theta_{n+1} - \theta_i}{\theta_{n+1}}\biggr) \\
		&=\frac{M(\theta)}{\theta_{n+1}}\frac{\theta_i(1-\theta_i)}{\theta_{n+1}}.
	\end{align*}
So that, if $A^{ii}$ denote the diagonal entries of $A^{-1}$, we have that
\begin{equation*}
	A^{ii} = \frac{\theta_i(1-\theta_i)}{\theta_{n+1}}.
\end{equation*}
For $i\neq j$ we have from Lemma~\ref{lemma:productsymmetricmatrix},
\begin{equation*}
K_{ij}=-\prod_{\substack{m\neq i\\ m\neq j}}\frac{\theta_{n+1}}{\theta_m} = -\frac{\theta_i}{\theta_{n+1}}\frac{\theta_j}{\theta_{n+1}} M(\theta).
\end{equation*}
Hence, $A^{ij}=-\frac{\theta_i\theta_j}{\theta_{n+1}}$.	So that $g^{ij}=\theta_{n+1} A^{ij}$, then $g^{ii}(\theta)=\theta_i(1-\theta_i)$ and $g^{ij}=-\theta_i\theta_j$ for $i\neq j$. Which can be written as $g^{ij}=\theta_i(\delta_{ij}-\theta_j)$ in terms of the Kronecker symbols $\delta_{ij}$.
\end{proof}

The above lemmas allows us to find the explicit $n\times n$ ODE system for the geodesics in $U$ induced by $J$. Recall (see do Carmo \cite{DoCarmobook}) that in general the geodesic system generated by $J=(g_{ij})$, $i,j =1,\ldots,n$ is given by
\begin{equation*}
	\frac{d^2\theta_k}{dt^2} + \sum_{i=1}^{n}\sum_{j=1}^{n} \Gamma^k_{ij} \frac{d\theta_i}{dt}\frac{d\theta_j}{dt}=0;\quad k=1,\ldots,n,
\end{equation*}
where the Christoffel symbols $\Gamma^k_{ij}$ are given by
\begin{equation*}
	\Gamma^k_{ij}=\sum_{l=1}^{n}\Gamma_{ijl} g^{lk},
\end{equation*}
with $(g^{lk})_{l,k=1,\ldots,n}=J^{-1}$, $\Gamma_{ijl}=\frac{1}{2}(g_{jl,i}+g_{li,j}-g_{ij,l})$ and $g_{jl,i}=\frac{\partial}{\partial\theta_i} g_{jl}$. In the next result we obtain the Christoffel symbols and the geodesic system in our particular case.
\begin{proposition}\label{propo:geodesicsystem}
Let $U=\{\theta=(\theta_1,\ldots,\theta_n): 0<\theta_i<1 \textrm{ for every } i \textrm{ and } 0<\sum_{j=1}^{n}\theta_j<1\}$, $\theta_{n+1}=1-\sum_{i=1}^{n}\theta_i$, $J(\theta)=\frac{1}{\theta_{n+1}}\bar{\bar{1}}+D(\frac{1}{\theta_j})$ and $\proin{u}{v}_{J(\theta)}=u^T J(\theta) v$. Then, for $k=1,\ldots,n$, we have
\begin{enumerate}
\item $\Gamma^k_{ij}=\frac{1}{2}\left[\frac{\theta_k}{\theta_{n+1}}-\frac{1}{\theta_i}\delta_{ij}\delta_{jk}+\frac{\theta_k}{\theta_i}\delta_{ij}\right]$; and
\item the geodesic system is given by
	\begin{equation*}
			2\frac{d^2\theta_k}{dt^2} + \frac{\theta_k}{\theta_{n+1}}\left(\sum_{i=1}^{n}\frac{d\theta_i}{dt}\right)^2-\frac{1}{\theta_k}\left(\frac{d\theta_k}{dt}\right)^2 + \theta_k\sum_{j=1}^{n}\frac{1}{\theta_j}\left(\frac{d\theta_j}{dt}\right)^2=0.
	\end{equation*}
\end{enumerate}
\end{proposition}
\begin{proof}
Let us start by computing the Christoffel symbols of the first kind
$\Gamma_{ijl}=\frac{1}{2}(g_{jl,i}+g_{li,j}-g_{ij,l})$.	
Notice first that
\begin{align*}
	g_{jl,i} &= \frac{\partial}{\partial\theta_i} g_{jl}\\
	&= \frac{\partial}{\partial\theta_i} \left(\frac{1}{\theta_{n+1}} + \frac{\delta_{jl}}{\theta_j}\right)\\
	&=\frac{1}{\theta_{n+1}^2} - \delta_{jl}\delta_{ji}\frac{1}{\theta_i^2}.
\end{align*}
Hence
\begin{align*}
	\Gamma_{ijl} &= \frac{1}{2}\left(g_{jl,i}+g_{li,j}-g_{ij,l}\right)\\
	&= \frac{1}{2}\left(\frac{1}{\theta_{n+1}^2} - \delta_{jl}\delta_{ji}\frac{1}{\theta_i^2}- \delta_{li}\delta_{lj}\frac{1}{\theta_j^2} + \delta_{ij}\delta_{il}\frac{1}{\theta_l^2}\right).
\end{align*}
So that $\Gamma_{ijl}=\frac{1}{2}\frac{1}{\theta_{n+1}^2}$, for $(i,j,l)\notin\Delta$ the diagonal of $\{1,2,\ldots,n\}^3$ and $\Gamma_{iii}=\frac{1}{2}\Bigl(\frac{1}{\theta_{n+1}^2}-\frac{1}{\theta_i^2}\Bigr)$ for every $i\in \{1,\ldots,n\}$. Let us now compute the Christoffel symbols of the second kind. Recall that for $i,j,k=1,\ldots,n$, the Christoffel symbol of the second kind is given by
$\Gamma_{ij}^k = \sum_{l=1}^{n} \Gamma_{ijl} g^{lk}$, where $\Gamma_{ijl}$ are the Christoffel symbols provided by $J$ and $(g^{lk})_{l,k=1,\ldots,n}$ is the inverse of $J$. From Lemma~\ref{lemma:inverseRiemannianMatrixJ} we have $g^{lk} = \theta_l (\delta_{lk} - \theta_k)$. Hence
\begin{align*}
	\Gamma_{ij}^k &= \sum_{l=1}^{n} \Gamma_{ijl} g^{lk}\\
	&= \frac{1}{2} \sum_{l=1}^{n} \Bigl(\frac{1}{\theta_{n+1}^2} - \frac{1}{\theta_i^2}\delta_{jl}\delta_{ji}- \frac{1}{\theta_j^2}\delta_{li}\delta_{lj} + \frac{1}{\theta_l^2}\delta_{ij}\delta_{il}\Bigr)\theta_l(\delta_{lk} - \theta_k)\\
	&= \frac{1}{2} \sum_{l=1}^{n}\theta_l \left[
	\frac{1}{\theta_{n+1}^2}\delta_{lk} - \frac{1}{\theta_i^2}\delta_{jl}\delta_{ji}\delta_{lk}- \frac{1}{\theta_j^2}\delta_{li}\delta_{lj}\delta_{lk}\right.\\
	&\phantom{\frac{1}{2} \sum_{l=1}^{n}}
	+ \frac{1}{\theta_l^2}\delta_{ij}\delta_{il}\delta_{lk} - \frac{\theta_k}{\theta_{n+1}^2} + \frac{\theta_k}{\theta_i^2}\delta_{jl}\delta_{ji}\\
	&\phantom{\frac{1}{2} \sum_{l=1}^{n}} \left.+ \frac{\theta_k}{\theta_j^2}\delta_{li}\delta_{lj} - \frac{\theta_k}{\theta_l^2}\delta_{ij}\delta_{il}
	\right]\\
	&= \frac{1}{2}\left[\frac{\theta_k}{\theta_{n+1}^2} -\frac{1}{\theta_i^2}\delta_{ji}\sum_{l=1}^{n}\theta_l \delta_{jl}\delta_{lk}- \frac{1}{\theta_j^2}\sum_{l=1}^{n}\theta_l\delta_{li}\delta_{lj}\delta_{lk}\right.\\
	&\phantom{\frac{1}{2} \sum_{l=1}^{n}} + \delta_{ij}\sum_{l=1}^{n} \frac{\theta_l}{\theta_l^2}\delta_{il}\delta_{lk} - \frac{\theta_k(1-\theta_{n+1})}{\theta_{n+1}^2}\\
	&\phantom{\frac{1}{2} \sum_{l=1}^{n}} +\frac{\theta_k}{\theta_i^2}\delta_{ji}\theta_j + \frac{\theta_k}{\theta_j^2}\sum_{l=1}^{n}\theta_l\delta_{li}\delta_{lj}\\
	&\phantom{\frac{1}{2} \sum_{l=1}^{n}}\left. - \theta_k\delta_{ij}\frac{1}{\theta_i}\right],
\end{align*}
then
\begin{align*}
	2\Gamma_{ij}^{k} &= \frac{\theta_k}{\theta_{n+1}} - \frac{1}{\theta_i}\delta_{ij}\delta_{jk}\\
	&\phantom{\frac{1}{2}} -\frac{1}{\theta_i}\delta_{ij}\delta_{jk} + \frac{1}{\theta_i}\delta_{ij}\delta_{ik}\\
	&\phantom{\frac{1}{2}} +\frac{\theta_k}{\theta_i}\delta_{ji} + \frac{\theta_k}{\theta_i}\delta_{ij}-\frac{\theta_k}{\theta_i}\delta_{ij}\\
	&= \frac{\theta_k}{\theta_{n+1}} - \frac{1}{\theta_i}\delta_{ij}\delta_{jk} + \frac{\theta_k}{\theta_i}\delta_{ij}.
\end{align*}
The geodesic ODE system is given by
\begin{equation*}
	2 \frac{d^2\theta_k}{dt^2} + \sum_{i=1}^{n}\sum_{j=1}^{n} 2\Gamma_{ij}^k \frac{d\theta_i}{dt} \frac{d\theta_j}{dt} = 0;\quad k=1,\ldots,n.
\end{equation*}
For fixed $k$, let us compute the quadratic form induced by the matrix $(2\Gamma_{ij}^k)_{i,j=1,\ldots,n}$ on a  vector $u=(u_i)_{i=1,\ldots,n}$ for fixed $k\in\{1,\ldots,n\}$,
\begin{align*}
	Q(u) &= 2 \sum_{i=1}^{n}\sum_{j=1}^{n} \Gamma_{ij}^k u_i u_j\\
	&= \sum_{i=1}^{n}\sum_{j=1}^{n} \Bigg(\frac{\theta_k}{\theta_{n+1}}-\frac{1}{\theta_i}\delta_{ij}\delta_{jk}+\frac{\theta_k}{\theta_i}\delta_{ij}\Bigg)u_i u_j\\
	&= \frac{\theta_k}{\theta_{n+1}}\Bigl(\sum_{i=1}^{n} u_i\Bigr)^2 - \sum_{i=1}^{n}\frac{u_i}{\theta_i}\sum_{j=1}^{n} u_j\delta_{ij}\delta_{jk} + \\
	&\phantom{\frac{\theta_k}{\theta_{n+1}}} + \theta_k\sum_{j=1}^{n} u_j \left(\sum_{i=1}^{n}\frac{u_i}{\theta_i}\delta_{ij}\right)\\
	&= \frac{\theta_k}{\theta_{n+1}}\Bigl(\sum_{i=1}^{n} u_i\Bigr)^2 - \frac{u_k^2}{\theta_k} + \theta_k \sum_{j=1}^{n}\frac{u_j^2}{\theta_j}.
\end{align*}
So that, taking $u=(\frac{d\theta_i}{dt})_{i=1,\ldots,n}$, we get
\begin{equation*}
	2\frac{d\theta_k^2}{dt^2} + \frac{\theta_k}{\theta_{n+1}}\Bigl(\sum_{i=1}^{n}\frac{d\theta_i}{dt} \Bigr)^2 
	- \frac{1}{\theta_k}\Bigl(\frac{d\theta_k}{dt}\Bigr)^2 + \theta_k\sum_{j=1}^{n}\frac{1}{\theta_j}\Bigl(\frac{d\theta_j}{dt}\Bigr)^2 = 0,
\end{equation*}
for $k=1,\ldots,n$.	
\end{proof}
Taking into account that a geodesic is a unit speed curve with respect to $J(\theta)$ that satisfies the ODE system in \textit{(2)} of Proposition~\ref{propo:geodesicsystem}, we are in position to stablish the main result of this section.
\begin{theorem}\label{thm:SecondOrderDecoupledSystem}
	Let $U$ and $J$ as before. Then a $J$-geodesic in $U$ is a solution of the following second order decoupled system
	\begin{equation*}
		2\theta_k\frac{d^2\theta_k}{dt^2}+\theta_k^2-\left(\frac{d\theta_k}{dt}\right)^2 = 0;\quad k=1,2,\ldots,n.
	\end{equation*}
\end{theorem}
\begin{proof}
From Theorem~\ref{thm:arcLengthFisherInformation} we see that the unit speed condition $\proin{\dot{\boldsymbol{\theta}}(t)}{\dot{\boldsymbol{\theta}}(t)}_{J}=1$ is equivalent to
\begin{align*}
	1 &= \sum_{j=1}^{n+1}\frac{1}{\theta_j}\Bigl(\frac{d\theta_j}{dt}\Bigr)^2\\
	&= \frac{1}{\theta_{n+1}}\Bigl(\frac{d\theta_{n+1}}{dt}\Bigr)^2+ \sum_{j=1}^{n}\frac{1}{\theta_j}\Bigl(\frac{d\theta_j}{dt}\Bigr)^2\\
	&= \frac{1}{\theta_{n+1}}\Bigl(\frac{d}{dt}\bigl(1-\sum_{j=1}^n\theta_j\bigr)\Bigr)^2+ \sum_{j=1}^{n}\frac{1}{\theta_j}\Bigl(\frac{d\theta_j}{dt}\Bigr)^2\\
	&= \frac{1}{\theta_{n+1}}\left(\sum_{j=1}^{n}\frac{d\theta_j}{dt}\right)^2 + \sum_{j=1}^{n}\frac{1}{\theta_j}\Bigl(\frac{d\theta_j}{dt}\Bigr)^2.
\end{align*}
So that, by simple inspection of the ODE in \textit{(2)} of Proposition~\ref{propo:geodesicsystem} we get 
\begin{equation*}
2\frac{d^2\theta_k}{dt^2}+\theta_k-\frac{1}{\theta_k}\left(\frac{d\theta_k}{dt}\right)^2 = 0 .
\end{equation*}
for every $k=1,\ldots,n$. In $U$ this system is equivalent to $2\theta_k\frac{d^2\theta_k}{dt^2}+\theta_k^2-\left(\frac{d\theta_k}{dt}\right)^2 = 0$. In simplified dot notation for derivatives, the geodesic system can be written as $2\theta_k\ddot{\theta}_k+\theta_k^2-(\dot{\theta}_k)^2=0$.
\end{proof}

\section{Solutions of the ODE that preserve the unit speed geodesic condition} \label{sec:SolutionsOfTheODEthatPreserveTheUnitSpeedGeodesicCondition}
In this section we aim to give sufficient conditions in the initial position and velocity $\theta_k(0)=\theta_k^0$ and $\dot{\theta}_k(0)=v_k^0$; $k=1,\ldots,n$, in order to obtain explicit solutions for the problem
\begin{equation*}
	\begin{cases}
		2\theta_k\ddot{\theta}_k+\theta_k^2-(\dot{\theta}_k)^2=0, \quad k=1,\ldots,n;\\
		\theta_k(0)=\theta_k^0;\\
		\dot{\theta}_k(0)=v_k^0;
	\end{cases}
\end{equation*}
with $\theta_k(t)$ satisfying the conservation formula
\begin{equation*}
	\sum_{k=1}^{n+1}\frac{(\dot{\theta}(t))^2}{\theta_k(t)} \equiv 1,\quad t>0,
\end{equation*}
for the unit speed of the solution with respect to $\proin{}{}_J$. We shall actually work in a much more general setting including discrete and continuous cases at once.
In the following result we obtain explicit formulas for the solutions of the basic ODE.
\begin{lemma}\label{lemma:IVP}
	Let $y_0>0$ and $z_0\in\mathbb{R}$ be given. Set  $\alpha=\frac{y_0^2+z_0^2}{y_0}=y_0+\frac{z_0^2}{y_0}$ and $\beta=\tan^{-1}\frac{z_0}{y_0}$. Then, the function of $t\in\mathbb{R}$ given by
	\begin{align*}
		y(t) &= \alpha \cos^2\Bigl(\frac{t}{2}-\beta\Bigr)\\
		&= y_0\cos^2\Bigl(\frac{t}{2}\Bigr) + \frac{z_0^2}{y_0}\sin^2\Bigl(\frac{t}{2}\Bigr) + z_0\sin t,
	\end{align*}
is the unique solution of the initial value problem
\begin{equation*} 
	(P)
	\begin{cases}
		2 y \ddot{y} + y^2 - (\dot{y})^2 = 0;\\
		y(0) = y_0;\\
		\dot{y}(0) = z_0.
	\end{cases}
\end{equation*}
\end{lemma}
\begin{proof}
Notice first that since $y_0$ is positive, the first order associated system satisfies the classical uniqueness results in some neighbourhood of $(y_0,z_0)$. Let us start by proving the identity
\begin{equation*}
	\alpha\cos^2\left(\frac{t}{2}-\beta\right) = y_0\cos^2\Bigl(\frac{t}{2}\Bigr)+\frac{z_0^2}{y_0}\sin^2\Bigl(\frac{t}{2}\Bigr) + z_0\sin t. 
\end{equation*}
Since $\beta=\tan^{-1}\frac{z_0}{y_0}$ it is clear that $\cos^2\beta=\frac{y_0^2}{y_0^2+z_0^2}$ and $\sin^2\beta=\frac{z_0^2}{y_0^2+z_0^2}$. Then
\begin{align*}
	\alpha\cos^2\left(\frac{t}{2}-\beta\right) &= \alpha\left(\cos\frac{t}{2}\cos\beta +\sin\frac{t}{2}\sin\beta\right)^2\\
	&= \alpha\cos^2\beta\left(\cos\frac{t}{2}+\frac{\sin\beta}{\cos\beta}\sin\frac{t}{2}\right)^2\\
	&= \frac{y_0^2+z_0^2}{y_0}\frac{y_0^2}{y_0^2+z_0^2}\left(\cos^2\frac{t}{2}+\Bigl(\frac{z_0}{y_0}\Bigr)^2\sin^2\frac{t}{2} + \frac{z_0}{y_0}\sin t\right)\\
	&= y_0\cos^2\frac{t}{2} + \frac{z_0^2}{y_0}\sin^2\frac{t}{2} + z_0\sin t.
\end{align*}
Let us show that $y(t)$ solves \textit{(P)}. Notice first that $y(0)=\alpha \cos^2\beta=\frac{y_0^2+z_0^2}{y_0}\frac{y_0^2}{y_0^2+z_0^2}=y_0$. From the second formula for $y(t)$ we see that $\dot{y}(t)=-y_0\sin \frac{t}{2}\cos\frac{t}{2}+\frac{z_0^2}{y_0}\cos\frac{t}{2}\sin \frac{t}{2}+z_0\cos t$. Hence $\dot{y}(0)=z_0$. Let us finally check that $y(t)$ satisfies the ODE. From the first expression for $y(t)$ we see that $\dot{y}(t)=-\alpha\cos\bigl(\tfrac{t}{2}-\beta\bigr)\sin\bigl(\tfrac{t}{2}-\beta\bigr)$. Let us compute $\ddot{y}$,
\begin{align*}
	\ddot{y}(t) &= -\frac{\alpha}{2}\left[-\sin^2\Bigl(\frac{t}{2}-\beta\Bigr)+\cos^2\Bigl(\frac{t}{2}-\beta\Bigr)\right]\\
	&= \frac{\alpha}{2}\left[2\sin^2\Bigl(\frac{t}{2}-\beta\Bigr)-1\right]\\
	&= \alpha\left[\sin^2\Bigl(\frac{t}{2}-\beta\Bigr)-\frac{1}{2}\right]. 
\end{align*}
Then
\begin{align*}
		2 y \ddot{y} + y^2 - (\dot{y})^2 &= 2\alpha^2\cos^2\Bigl(\frac{t}{2}-\beta\Bigr)\left[\sin^2\Bigl(\frac{t}{2}-\beta\Bigr)-\frac{1}{2}\right] \\
		&\phantom{=\,\,}+ \alpha^2\cos^4\Bigl(\frac{t}{2}-\beta\Bigr)-\alpha^2\cos^2\Bigl(\frac{t}{2}-\beta\Bigr)\sin^2\Bigl(\frac{t}{2}-\beta\Bigr)\\
		&= \alpha^2\cos^2\Bigl(\frac{t}{2}-\beta\Bigr)\left[2\sin^2\Bigl(\frac{t}{2}-\beta\Bigr) - 1 + \cos^2\Bigl(\frac{t}{2}-\beta\Bigr) - \sin^2\Bigl(\frac{t}{2}-\beta\Bigr)\right]\\
		&= 0,
\end{align*}
for every $t$.
\end{proof}

The next statement, which contains the main result of this section, is proved in a general measure space that contains both the discrete and continuous cases.
\begin{theorem}\label{thm:MainResultUnitSpeedCondition}
	Let $(X,\mu)$ be a positive $\sigma$-finite measure space. Let $f_0:X\to\mathbb{R}^+$ be a positive probability density with respect to $\mu$, i.e., $\int f_0(x)d\mu(x)=1$ and $f_0(x)>0$ for every $x\in X$.  Let $g_0:X\to\mathbb{R}$ be an integrable function, i.e., $\int_X\abs{g_0}d\mu<\infty$ such that
	\begin{enumerate}[(a)]
		\item $\int_X g_0(x)d\mu(x)=0$;
		\item $\int_X\frac{g_0^2(x)}{f_0(x)}d\mu(x)=1$, i.e. $\frac{g_0^2}{f_0}$ is a probability density with respect to $\mu$.
	\end{enumerate}
Then, the function $f:X\times\mathbb{R}\to\mathbb{R}$ given by
\begin{equation*}
	f(x,t) = \alpha(x)\cos^2\Bigl(\frac{t}{2}-\beta(x)\Bigr),
\end{equation*}
where $\alpha(x)=\frac{f_0^2(x)+g_0^2(x)}{f_0(x)}$ and $\beta(x)=\tan^{-1}\frac{g_0(x)}{f_0(x)}$ satisfies the following properties;
\begin{enumerate}[(i)]
	\item for every $x\in X$ the function of $t$ given by $f_x(t)=f(x,t)$ solves \textit{(P)} with $y_0=f_0(x)$ and $z_0=g_0(x)$ for each $x\in X$;
	\item $\int_X f(x,t)d\mu(x)=1$ for every $t$, i.e., $f(\cdot,t)$ is a probability density for every $t\in\mathbb{R}$;
	\item $\int_X\frac{(\dot{f}(x,t))^2}{f(x,t)}d\mu(x)=1$ for every $t$, i.e. $\frac{(\dot{f}(\cdot,t))^2}{f(\cdot,t)}$ is a probability density for every $t\in\mathbb{R}$.
\end{enumerate}
\end{theorem}
\begin{proof}
Property \textit{(i)} follows from Lemma~\ref{lemma:IVP}. Let us check \textit{(ii)}. From the formula $f(x,t)=f_0(x)\cos^2\tfrac{t}{2}+\tfrac{(g_0(x))^2}{f_0(x)}\sin^2\tfrac{t}{2}+g_0(x)\sin t$, \textit{(a)}, \textit{(b)} and the fact the $f_0$ is a probability density, we obtain
\begin{align*}
\int_X f(x,t)d\mu(x) &= \left(\int_Xf_0d\mu\right)\cos^2\frac{t}{2}+\left(\int_X\frac{g_0^2}{f_0}d\mu\right)\sin^2\frac{t}{2}+\left(\int_Xg_0d\mu\right)\sin t\\
&= \cos^2\left(\frac{t}{2}\right)+\sin^2\left(\frac{t}{2}\right)\\
&= 1,
\end{align*}
for every $t$. Let us finally check \textit{(iii)}. Notice that
\begin{equation*}
	\frac{(\dot{f}(x,t))^2}{f(x,t)} = \frac{\left(\alpha(x)\cos\bigl(\tfrac{t}{2}-\beta(x)\bigr)\sin\bigl(\tfrac{t}{2}-\beta(x)\bigr)\right)^2}{\alpha(x)\cos^2\bigl(\tfrac{t}{2}-\beta(x)\bigr)} = \alpha(x)\sin^2\bigl(\tfrac{t}{2}-\beta(x)\bigr).
\end{equation*}
Now. from \textit{(ii)}
\begin{align*}
	1+ \int_X \frac{(\dot{f}(x,t))^2}{f(x,t)} d\mu(x) &= \int_X \left(f(x,t)+\frac{(\dot{f}(x,t))^2}{f(x,t)} \right)d\mu(x) \\
	&= \int_X\left(\alpha(x)\cos^2\bigl(\tfrac{t}{2}-\beta(x)\bigr)+\alpha(x)\sin^2\bigl(\tfrac{t}{2}-\beta(x)\bigr)\right)d\mu(x)\\
	&= \int_X\alpha(x)d\mu(x)\\
	&= \int_X\left(f_0(x) + \frac{(g_0(x))^2}{f_0(x)}\right)d\mu(x)\\
	&= 2;
\end{align*}
so that $\int_X \frac{(\dot{f}(x,t))^2}{f(x,t)} d\mu(x)=1$, for every $t\in\mathbb{R}$.
\end{proof}

Let us observe at this point that in the setting of Section~\ref{sec:TheRiemannianGeometryInTheSimplex} we are dealing with a geometric structure in $U=\{\theta\in\mathbb{R}^n:\theta_i>0, i=1,\ldots,n+1\}$ which can be identified with the set of all probability densities $f$ in the space $(X,\mu)$ with $X=\{1,2,\ldots,n+1\}$ and $\mu$ the counting measure. In fact, $\theta\to f_\theta$, with $f_\theta(i)=\theta_i$ is a density, since $\int_X f_\theta d\mu=\sum_{i=1}^{n+1}\theta_i=1$. In this sense, the Fisher-Riemann geometry in $U$ translates into a geometry in the set of positive densities in $(X,\mu)$. A geodesic curve of densities will be of the form $f(i,t)=\theta_i(t)$, $i\in X$, with $2\theta_i\ddot{\theta}_i+\theta_i^2-(\dot{\theta}_i)^2=0$. More explicitly Theorem~\ref{thm:MainResultUnitSpeedCondition} gives the analytical form of $f(i,t)$ for $i=1,\ldots,n+1$,
\begin{equation*}
	f(i,t) = \frac{f_0^2(i)+g_0^2(i)}{f_0(i)}\cos^2\Bigl(\frac{t}{2}-\tan^{-1}\frac{g_0(i)}{f_0(i)}\Bigr),
\end{equation*}
when $\sum_{i=1}^{n+1}f_0(i)=1$;  $\sum_{i=1}^{n+1}g_0(i)=0$ and $\sum_{i=1}^{n}\frac{g_0^2(i)}{f_0(i)}=1$.

On the other hand, Theorem~\ref{thm:MainResultUnitSpeedCondition} contains also continuous cases of non vanishing densities in some subset $\Omega$ of $\mathbb{R}^n$. In fact,
\begin{equation*}
		f(x,t) = \frac{f_0^2(x)+g_0^2(x)}{f_0(x)}\cos^2\Bigl(\frac{t}{2}-\tan^{-1}\frac{g_0(x)}{f_0(x)}\Bigr)
\end{equation*}
is a Fisher-Riemann geodesic trajectory in the set of probability densities in $\Omega$ provided that $f_0$ and $\tfrac{g_0^2}{f_0}$ are densities in $\Omega$ and $\int_\Omega g_0 dx=0$.

\section{Fisher geodesic transport of $m$-dimensional densities and the convergence of their dyadic pixelations}\label{sec:FisherGeodesicTransportOfnDimensionalDensitiesAndOfTheirDyadicPixelations}

The generality of the basic measure space $(X,\mu)$ in Theorem~\ref{thm:MainResultUnitSpeedCondition} allows its application to the approximation of the geodesic curves corresponding to continuous (parametric) densities by geodesic curves corresponding to discrete (finite dimensional) settings.

Let $Q=[0,1)^m$ be the unit cube in $\mathbb{R}^m$. Let us consider the nested dyadic partitions of $Q$ that we proceed to describe.
For an integer $j\geq 0$  and $\mathbf{k}\in \mathcal{K}(j)\doteq\{(k_1, \ldots, k_m)\in 
\mathbb{Z}^m:k_i=0,1, \ldots, 2^j-1; i=1, \ldots, m\}$, set 
\begin{equation*}
Q_\mathbf{k}^j=\prod_{i=1}^m [k_i2^{-j}, (k_i+1)2^{-j}),\ \ 
\mathcal{D}^j=\{Q_\mathbf{k}^j:\mathbf{k}\in \mathcal{K}(j)\}\ \   \text{ and } \ \
\mathcal{D}=\bigcup_{j\geq0}\mathcal{D}^j.	
\end{equation*}
Notice that $Q_0^0=Q$ is the unique element of $\mathcal{D}^0$. Observe also that $Q_\mathbf{k}^j\cap Q_\mathbf{k'}^j=\emptyset$ for $\mathbf{k}\neq \mathbf{k'}$, $Q=\bigcup_{\mathbf{k}\in \mathcal{K}(j)}Q_\mathbf{k}^j$,  and each $Q_\mathbf{k}^j$ is the disjoint union of $2^m$ cubes $Q_\mathbf{k'}^{j+1}\in \mathcal{D}^{j+1}$.

Let us consider now a sequence of discrete (finite) probability spaces $(X_j, \mu_j)$, $j\geq 0$, that converges weakly to $(Q, dx)$. For $j\geq 0$, set $X_j=\{\mathbf{k}2^{-j}:\mathbf{k}\in\mathcal{K}(j)\}$ and $\mu_j$ is $2^{ -mj}$ times  the counting measure in $X_j$. Observe that each $x_\mathbf{k}=\mathbf{k}2^{-j}$ is the lower left corner of $Q_{\mathbf{k}}^j$. Given a positive density $f_0:Q\to \mathbb{R}^{+}$, $\int_Q f_0(x) dx=1$, for each $j\geq 0$ define $f^j_0:X_j\to \mathbb{R}^{+}$ by 
\begin{equation}\label{def:fj}
	f_0^j(x_{\mathbf{k}})\doteq\frac{1}{|Q_{\mathbf{k}}^j|}\int_{Q_{\mathbf{k}}^j} f_0(y) dy=\fint_{Q_{\mathbf{k}}^j} f_0 dy,
\end{equation}
where we use the notation $\fint_E \psi$ for the mean value of $\psi$ on $E$.
Note that $|Q_{\mathbf{k}}^j|=2^{-mj}$ for all $\mathbf{k}\in \mathcal{K}(j)$. Then 
\begin{align*}
	\int_{X_j}f_0^j d\mu_j &= 2^{-mj} \sum_{\mathbf{k}\in\mathcal{K}(j)} \frac{1}{|Q_{\mathbf{k}}^j|}\int_{Q_{\mathbf{k}}^j} f_0(y) dy\\
	&= \sum_{\mathbf{k}\in\mathcal{K}(j)}\int_{Q_{\mathbf{k}}^j} f_0(y) dy\\
	&= \int_Q f_0(y) dy \\
	&=1,
\end{align*}
and each $f_0^j$ is a probability density in the space $(X_j, \mu_j)$.

Notice also that if $g_0:Q \to \mathbb{R}$ is integrable, $\int_Q g_0(x) dx=0$ and we define $g^j_0:X_j\to \mathbb{R}$ as we did with $f^j_0$, we also have that $\int_{X_j} g_0^j\, d \mu_j=0$, for every $j\geq 0$, since
\[\int_{X_j} g_0^j\, d\mu_j= \int_Q g_0\, dy.\]
Of course, the ``unit speed'' initial condition in the $j$-th level of approximation is not guaranteed by the ``unit speed'' continuous initial condition $\int_Q \frac{g_0(x)^2}{f_0(x)} dx=1$. The next result contains a simple situation in which the finite dimensional geodesics converge to the corresponding geodesics in a non parametric family of continuous densities.

\begin{theorem}
	Let $f_0$ be a positive measurable function in the cube $Q$ such that $\int_Q f_0\, dx =1$ and $f_0 \geq \delta$ almost everywhere for some positive $\delta$. Let $g_0$ be an integrable real function defined in the cube $Q$ that satisfies
	\begin{enumerate}[a)]
		\item $\int_Q g_0\, dx=0$ and
		\item $\int_Q \frac{g_0^2}{f_0}\, dx=1.$
	\end{enumerate}
	For each $j\geq 0$ let us consider the real functions $f_0^j$ and $g_0^j$ with domain in $X_j$ defined as in \eqref{def:fj}. For $j$ large, set $\widetilde{g}_0^j:X_j\to\mathbb{R}$ given by $\widetilde{g}_0^j=\alpha_j^{-\frac{1}{2}}g_0^j$, where $\alpha_j=\int_{X_j} \frac{(g_0^j)^2}{f_0^j} d\mu_j>0$. Then
	\begin{enumerate}[i)]
		\item $\int_{X_j} f_0^j d\mu_j=1$;
		\item $\int_{X_j} \widetilde{g}_0^j d\mu_j=0$;
		\item $\int_{X_j} \frac{(\widetilde{g}_0^j)^2}{f_0^j} d\mu_j=1$;
		\item the sequence for Fisher-Riemann discrete density geodesics with initial conditions $f_0^j$ and $\widetilde{g}_0^j$, given by
		\begin{equation*}
			f^j(\cdot, t)= \frac{(f_0^j)^2+(\widetilde{g}_0^j)^2}{f_0^j} \cos^2 \left(\frac{t}{2}-\tan^{-1} \frac{\widetilde{g}_0^j}{f_0^j}\right)
		\end{equation*}	
		for each $t\in \mathbb{R}$, ``converges weakly'', as $j\rightarrow \infty$, to the continuous density geodesic with initial conditions $f_0$  and $g_0$ given by
		\begin{equation*}
			f(x, t)= \frac{f_0(x)^2 + g_0(x)^2}{f_0(x)} \cos^2 \left(\frac{t}{2}-\tan^{-1} \frac{g_0(x)}{f_0(x)}\right),
		\end{equation*}
		where  $x\in Q$. More precisely, for every $t$ and every $\varphi$ compactly supported and continuous in $Q$, we have 
		\begin{equation*}
			\left| \int_{X_j} f^j(y,t) \varphi(y) d\mu_j(y) - \int_Q f(y,t) \varphi(y) dy \right| \to 0 
		\end{equation*}
        for $j\to\infty$.
		\end{enumerate}
\end{theorem}

\begin{proof}
	First observe that, since $g_0$ satisfies $b)$, it can not be zero almost everywhere. Hence, for $j$ large enough, $(g_0^j)^2$ is positive in some set of positive measure and then $\alpha_j=\int_{X_j}\frac{(g_0^j)^2}{f_0^j}d\mu_ j>0$.

	Item $i)$ was proved before the statement of the theorem. To prove item $ii)$, notice that $\int_{X_j} \widetilde{g}_0^j\, d\mu_j=\alpha_j^{-1/2} \int_{X_j} g_0^j\, d\mu_j=0$. Item $iii)$ follows from the definition of $\alpha_j$,
	\begin{equation*}
		\int_{X_j} \frac{(\widetilde{g}_0^j)^2}{f_0^j}\, d\mu_j=\frac{1}{\alpha_j} \int_{X_j} \frac{(g_0^j)^2}{f_0^j}\, d\mu_j =1.
	\end{equation*}
	Let us finally prove $iv)$. Recall that, from Theorem~\ref{thm:MainResultUnitSpeedCondition} and Lemma~\ref{lemma:IVP}, we have
	\begin{equation*}
			f(x,t) =f_0(x)\cos^2\left(\frac{t}{2}\right)+\frac{g_0(x)^2}{f_0(x)}\sin^2\left(\frac{t}{2}\right)+g_0(x)\sin(t)
	\end{equation*}
	for each $x\in Q$ and $t\in \mathbb{R}$, 
	and
	\begin{equation*}
		f^j(x_{\mathbf{k}},t) =f^j_0(x_\mathbf{k})\cos^2\left(\frac{t}{2}\right)+\frac{\widetilde{g}^j_0(x_\mathbf{k})^2}{f^j_0(x_\mathbf{k})}\sin^2\left(\frac{t}{2}\right)+\widetilde{g}^j_0(x_\mathbf{k})\sin(t)
	\end{equation*}
	for each $x_{\mathbf{k}}\in X_j=\{\mathbf{k}2^{ -j}:\mathbf{k}\in \mathcal{K}(j)\}$.
	So that, we only have to prove that for every compactly supported and continuous $\varphi$ in $Q$, we have
	
	\begin{equation}\label{tesis_teo_1}
		\lim_{j\rightarrow\infty}   \left| \int_{X_j} f_0^j(y) \varphi(y) d\mu_j(y) - \int_Q f_0(y) \varphi(y) dy \right|=0,
	\end{equation}
	
	\begin{equation}\label{tesis_teo_2}
		\lim_{j\rightarrow\infty}   \left| \int_{X_j} \widetilde{g}_0^j(y) \varphi(y) d\mu_j(y) - \int_Q g_0(y) \varphi(y) dy \right|=0,
	\end{equation}
	
	\begin{equation}\label{tesis_teo_3}
		\lim_{j\rightarrow\infty}   \left| \int_{X_j} \frac{\widetilde{g}_0^j(y)^2}{f_0^j(y)} \varphi(y) d\mu_j(y) - \int_Q \frac{g_0(y)^2}{f_0(y)} \varphi(y) dy \right|=0.
	\end{equation}
	In order to prove \eqref{tesis_teo_1}, let us first compute the two integrals involved,
	\begin{align*}
		A &=\int_{X_j} f_0^j(y) \varphi(y) d\mu_j(y) \\
		&= \sum_{\mathbf{k}\in \mathcal{K}(j)} f_0^j(x_{\mathbf{k}}) \varphi(x_{\mathbf{k}})2^{ -mj}\\
		&= \sum_{\mathbf{k}\in \mathcal{K}(j)}\left(\frac{2^{-mj}}{|Q^j_{\mathbf{k}}|} \int_{Q^j_{\mathbf{k}}} f_0(y) dy\right) \varphi(x_{\mathbf{k}})\\
		&= \sum_{\mathbf{k}\in \mathcal{K}(j)}\int_{Q^j_{\mathbf{k}}} f_0(y)\varphi(x_{\mathbf{k}}) dy, 
	\end{align*}
	and
	\begin{align*}
		B =\int_Q f_0(y) \varphi(y) d(y) 
		= \sum_{\mathbf{k}\in \mathcal{K}(j)}\int_{Q^j_{\mathbf{k}}} f_0(y)\varphi(y) dy.   
	\end{align*}	
	Hence,
	\begin{equation*}
		|A-B|\leq \sum_{\mathbf{k}\in \mathcal{K}(j)}\int_{Q^j_{\mathbf{k}}} f_0(y)|\varphi(x_{\mathbf{k}})-\varphi(y)| dy
	\end{equation*}	
	and \eqref{tesis_teo_1} follows from the uniform continuity of $\varphi$ in $Q$ and the integrability of $f_0$ in $Q$.
	
	\noindent Proof of \eqref{tesis_teo_2}:
	\begin{align*}
		\left| \int_{X_j} \widetilde{g}_0^j \varphi d\mu_j - \int_Q g_0 \varphi dy \right|
		&\leq \left| \int_{X_j} (\widetilde{g}_0^j-g_0^j) \varphi d\mu_j \right|+ \left| \int_{X_j} g_0^j \varphi d\mu_j - \int_Q g_0 \varphi dy \right|\\
		&= I+II.
	\end{align*}
	For the second term $II$, we may argue as in \eqref{tesis_teo_1} since $|g_0|\in L^1(Q)$. Let us estimate $I$.
	\begin{align*}
		I &\leq \int_{X_j} |\alpha_j^{ -1/2}g_0^j-g_0^j| |\varphi| d\mu_j \\
		& \leq |\alpha_j^{ -1/2}-1|\ \, \|\varphi\|_{\infty} \int_{X_j} |g_0^j|  d\mu_j \\
		&\leq |\alpha_j^{ -1/2}-1|\ \,  \|\varphi\|_{\infty} \|g_0\|_1,
	\end{align*}
	which tends to zero as $j\to\infty$, since 
	\begin{equation}\label{eq:lim_aj}
		\alpha_j =\int_{X_j}\frac{(g_0^j)^2}{f_0^j} d\mu_j \underset{j\to\infty}{\longrightarrow} \int_Q \frac{g_0^2}{f_0} dy =1.
	\end{equation}
In fact, since $f_0^j(x_{\mathbf{k}})=\fint_{Q_{\mathbf{k}}^j} f_0 $ and $g_0^j(x_{\mathbf{k}})=\fint_{Q_{\mathbf{k}}^j} g_0 $, (the average value in $Q_{{\mathbf{k}}^j}$ of $f_0$ and $g_0$, respectively) the functions $\psi_j:Q\rightarrow \mathbb{R}^{+}$ given by $\psi_j\doteq\sum_{k\in \mathcal{K}(j)} \frac{g_0^j(x_{\mathbf{k}})^2}{f_0^j(x_{\mathbf{k}})}\mathds{1}_{Q_{\mathbf{k}}^j}$, converges pointwise to $\frac{g_0^2}{f_0}$ in $Q$, due to Lebesgue differentiation theorem through dyadic cubes. Also, $f_0 \geq \delta >0$ in $Q$ and $g_0\in L^2(Q)$ implies that 
\[\psi_j (x)\leq \frac{1}{\delta} \sum_{k\in \mathcal{K}(j)} \left(\frac{1}{|Q_{\mathbf{k}}^j|}\int_{Q_{\mathbf{k}}^j}g_0\, dy \right)^2\mathds{1}_{Q_{\mathbf{k}}^j}(x) \leq \frac{1}{\delta} \|g_0\|^2_{L^2(Q)} \sum_{k\in \mathcal{K}(j)} \frac{1}{|Q_{\mathbf{k}}^j|}\mathds{1}_{Q_{\mathbf{k}}^j}(x)\]
for all $x\in Q$, and the Lebesgue dominated convergence theorem finishes the proof of our claim \eqref{eq:lim_aj}, since $\alpha_j=\int_Q \psi_j\, dy \to \int_Q\frac{g_0^2}{f_0}\, dy=1$ as $j\rightarrow \infty$.
	
	\noindent Proof of \eqref{tesis_teo_3}:
	\begin{align*}
		\left| \int_{X_j} \frac{(\widetilde{g}_0^j)^2}{f_0^j} \varphi d\mu_j - \int_Q \frac{g_0^2}{f_0} \varphi dy \right| &\leq
		\left| \int_{X_j} \frac{(\widetilde{g}_0^j)^2-(g_0^j)^2}{f_0^j} \varphi d\mu_j \right| +
		\left| \int_{X_j} \frac{(g_0^j)^2}{f_0^j} \varphi d\mu_j - \int_Q \frac{g_0^2}{f_0} \varphi dy \right|\\
		&\leq I_j + {II}_j.
	\end{align*}
	Let us first estimate $I_j$,
	\begin{align*}
		I_j &\leq \left|\left(\frac{1}{\alpha_j}-1\right) \int_{X_j} \frac{(g_0^j)^2}{f_0^j} \varphi d\mu_j \right| \\
		&\leq \|\varphi\|_{\infty}\left|\frac{1}{\alpha_j}-1\right| \int_{X_j} \frac{(g_0^j)^2}{f_0^j} \varphi d\mu_j \\
		&=\leq \|\varphi\|_{\infty} |1-\alpha_j|,
	\end{align*}
	which from \eqref{eq:lim_aj} tends to zero when $j$ tends to infinity. Let us finally estimate ${II}_j$.
	\begin{align*}
		{II}_j &\leq \left| \sum_{k\in \mathcal{K}(j)} \frac{g_0^j(x_{\mathbf{k}})^2}{f_0^j(x_{\mathbf{k}})}\varphi(x_{\mathbf{k}})|Q_{\mathbf{k}}^j|-\int_Q\frac{g_0^2}{f_0}\varphi \, dy \right|\\
		&= \left|\int_Q\left( \sum_{k\in \mathcal{K}(j)} \frac{g_0^j(x_{\mathbf{k}})^2}{f_0^j(x_{\mathbf{k}})}\varphi(x_{\mathbf{k}})\mathds{1}_{Q_{\mathbf{k}}^j}-\frac{g_0^2}{f_0}\varphi \right) dy \right|.
	\end{align*}
	Now, from the continuity of $\varphi$ and, again, from the dyadic version of the differentiation Lebesgue theorem, we see that the function we are integrating on $Q$ tends to zero as $j$ tends to infinity almost everywhere. But, since $\varphi$ is bounded, $g\in L^2(Q)$ and $f_0\geq \delta >0$, we can apply the dominated convergence Lebesgue theorem to finish the proof of the theorem.
\end{proof}

\section{Examples and graphics}\label{secExamplesAndGraphics}
In this section, using \texttt{Octave}, we compute the explicit solutions of the geodesics of densities in some simple but illustrative situations.

Let us first start by the discrete case when $U=\{\theta_0=(\theta_1,\dots, \theta_n): 0<\theta_i<1, i=1, \dots n, \textrm{ and }\theta_{n+1}=1-\sum_{i=1}^n\theta_i>0\}$. In order to adapt our situation to the setting of Theorem~\ref{thm:MainResultUnitSpeedCondition} let us take $X=\{x_1,\dots,x_{n+1}\}$ and $\mu$ the counting measure on $X$. Let $f_0: X\to\mathbb{R}^+$ be given by $f_0(x_i)=\theta^0_i$, with $\theta^0_i >0$, $i=1,\dots,n+1$ and $\theta^0_1+\dots+\theta^0_{n+1}=1$. This $f_0$ identifies one and only one point $\theta^0=(\theta^0_1,\dots,\theta^0_n)$ in $U$. This will be the initial point of the geodesic in $U$. Following with the notation of Theorem~\ref{thm:MainResultUnitSpeedCondition}, we have to choose an initial velocity given by $g_0: X\to\mathbb{R}$, $g_0(x_i)=v^0_i$, satisfying two conditions: 
\begin{enumerate}[(a)]
	\item $v^0_1+\dots+v^0_{n+1}=0$, and
	\item $\frac{(v^0_1)^2}{\theta^0_1}+\dots+\frac{(v^0_{n+1})^2}{\theta^0_{n+1}}=1$.
\end{enumerate}
For a fixed $\theta^0$, this is equivalent to choose a $v^0=(v^0_1,\dots, v^0_n)$ in the $(n-1)$-dimensional ellipsoid given by
\begin{equation}\label{eq:ellipsoid}
    \sum_{i=1}^n \frac{(v^0_i)^2}{\theta^0_ i}+ \dfrac{( v^0_1+\dots+v^0_n)^2}{\theta^0_{n+1}}=1,
\end{equation}
which is the ``unit ball'' for the Fisher-Rao metric.
Hence, given $\theta^0$ and $v^0$, the geodesic in $U$ is the curve $\theta(t)=(\theta_1(t),\dots, \theta_n(t))$ given by
\begin{equation*} 
		\theta_i(t) = \theta^0_i\cos^2\left(\frac{t}{2}\right)+\frac{(v^0_i)^2}{\theta^0_i}\sin^2\left(\frac{t}{2}\right)+v^0_i\sin t	
\end{equation*}
for $i=1, \dots, n$.

Consider the simplest case $n=2$ with initial point $\theta^0=(\tfrac{1}{3},\tfrac{1}{3})$ which implies $\theta_3=\frac{1}{3}$. Equation \eqref{eq:ellipsoid} describe the ellipse in the plane of the of the variables $v^0_1$ and $v^0_2$, given by
\begin{equation*}
(v^0_1)^2+(v^0_2)^2+v^0_1v^0_2=\frac{1}{6}.
\end{equation*}
Figure~\ref{fig:figure61} shows some of these curves starting at the point $\theta^0=(\tfrac{1}{3},\tfrac{1}{3})$ for different initial velocities in the ellipse parametrized, with parameter $\tau\in\mathbb{R}$, by
\begin{align*}
	v^0_1(\tau) &= \frac{\sqrt{2}}{6}\left(-\sqrt{3}\cos\tau +\sin\tau\right),\\
	v^0_2(\tau) &= \frac{\sqrt{2}}{6}\left(\sqrt{3}\cos\tau + \sin\tau\right).
\end{align*}
\begin{figure}
	\includegraphics[width=10cm]{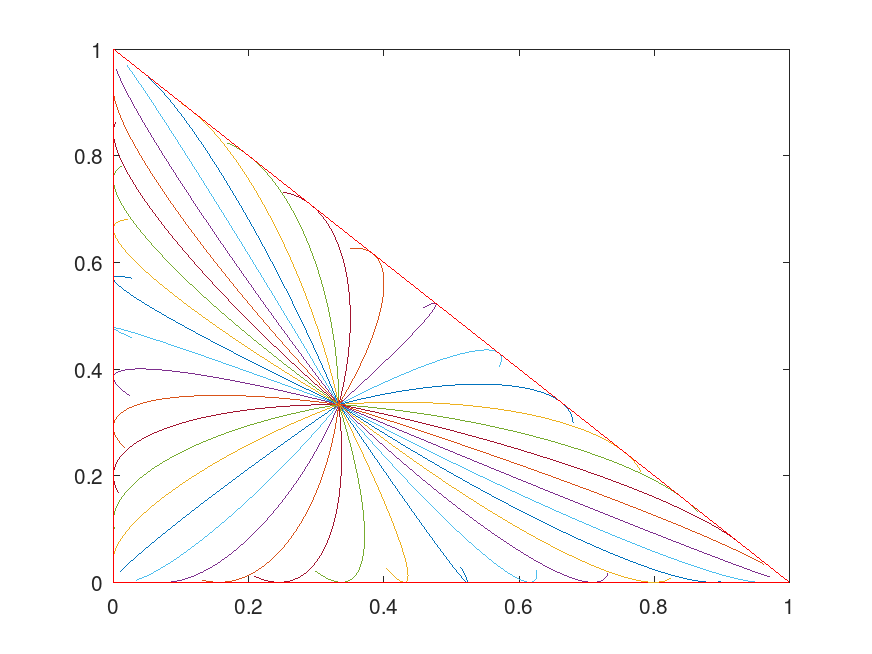}
	\caption{Several trajectories from $(\frac{1}{3},\frac{1}{3})$ for different initial velocities. Here we consider $t\in[0,\frac{\pi}{2}].$ }\label{fig:figure61}
\end{figure}
Note that all trajectories are ellipses, except when $v^0$ or $-v^0$  points to the vertices of the simplex, in which cases the trajectories are straight lines.

In figures~\ref{fig:figure612}, \ref{fig:figure613} and \ref{fig:figure614} we draw some complete trajectories for different values of $v^0$. 
\begin{figure}
    \centering
    \begin{minipage}{0.3\textwidth}
        \centering
        \includegraphics[width=\textwidth]{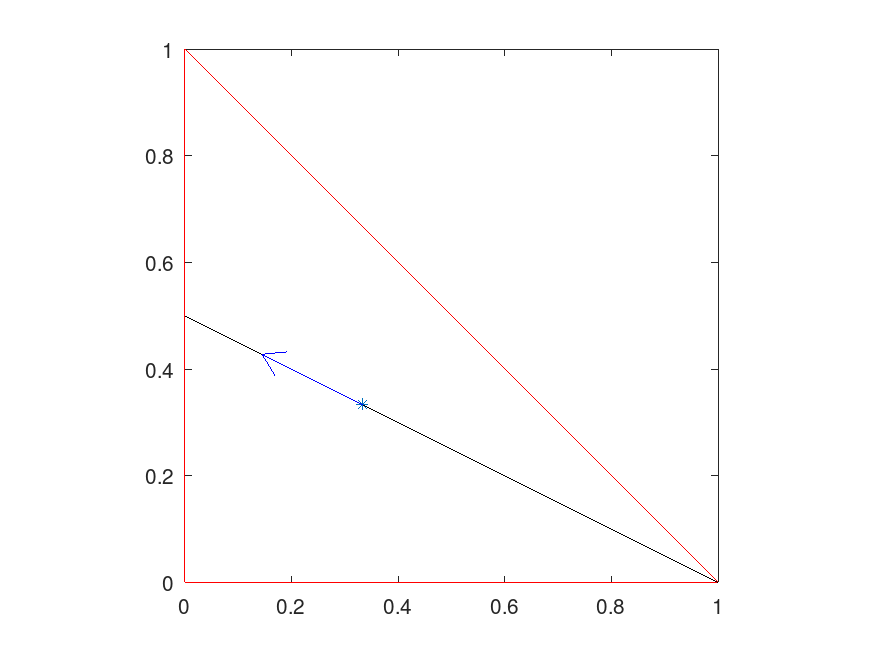} 
        \caption{$\tau=\frac{11\pi}{6}$}\label{fig:figure612}
    \end{minipage}\hfill
    \begin{minipage}{0.3\textwidth}
        \centering
        \includegraphics[width=\textwidth]{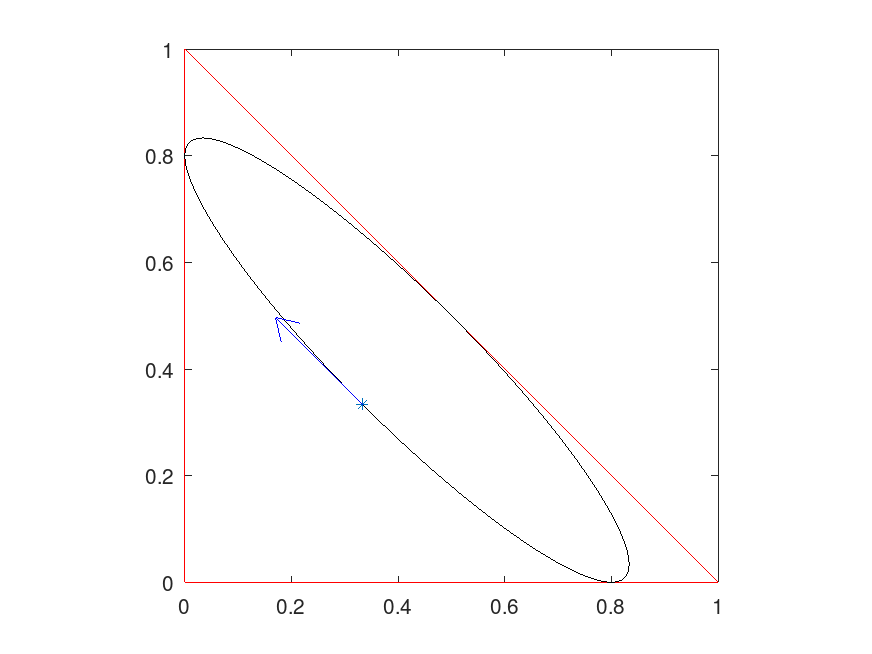} 
        \caption{$\tau \in(\frac{11\pi}{6},2\pi)$}\label{fig:figure613}
    \end{minipage}\hfill
    \begin{minipage}{0.3\textwidth}
        \centering
        \includegraphics[width=\textwidth]{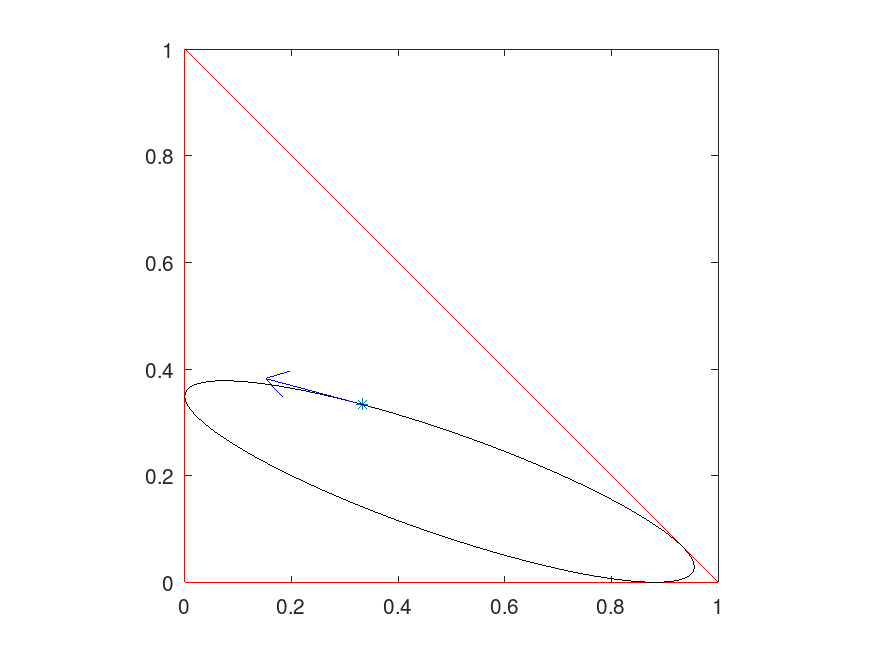} 
        \caption{$\tau \in(\frac{3\pi}{2},\frac{11\pi}{6})$}\label{fig:figure614}
    \end{minipage}
\end{figure}

For the case $n=3$, in Figure~\ref{fig:figure615} we draw several trajectories in $U=\{(\theta_1,\theta_2, \theta_3): 0<\theta_i<1, i=1,2,3; \theta_1+\theta_2+\theta_3>1\}$, all with initial point $\theta^0=[0.25, 0.25, 0.25]$ and different initial velocities.
\begin{figure}
	\includegraphics[width=13cm]{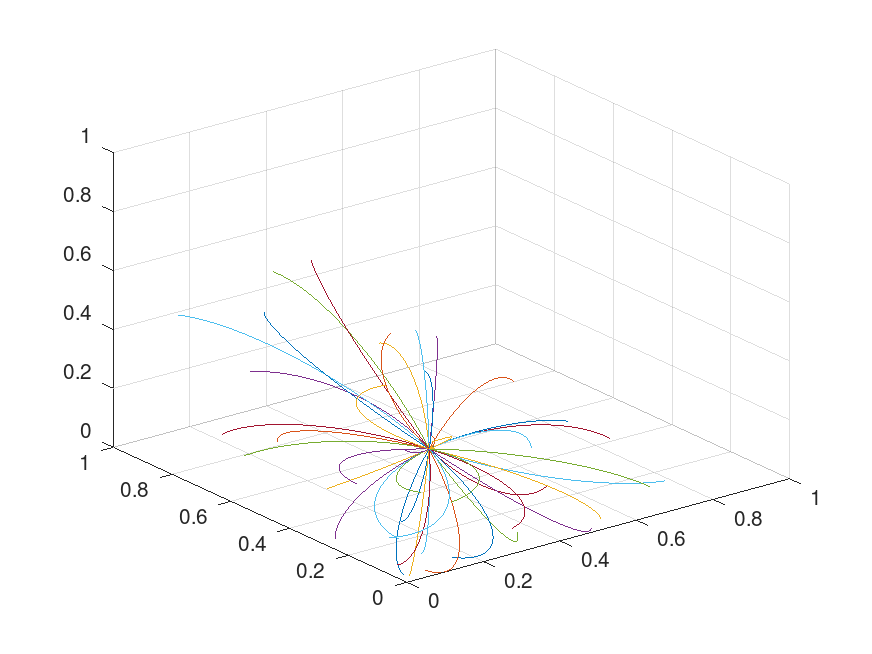}
	\caption{3-dimensional trajectories}\label{fig:figure615}
\end{figure}

Let us now show the dynamic of the Fisher transport of the uniform density in the interval $[0,1]$ to some more concentrated distribution. The space $(X,\mu)$ is now $([0,1],dx)$, where $dx$ denotes the Lebesgue measure. Take $f_0(x)\equiv 1$, the uniform density and $g_0: [0,1]\to \mathbb{R}$ be such that $\int_{[0,1]}g_0 dx=0$ and $\frac{g_0^2(x)}{f_0(x)}=g_0^2(x)=4\,\mathds{1}_{[0,1/4]}(x)$. It is clear that we can choose several functions $g_0$ satisfying above conditions. We illustrate the dynamics of $f(x,t)$ for $t\in [0,\pi]$ in two simple cases of $g_0$. Taking first $g_{0,1}(x)=2\left(\mathds{1}_{[0,1/8]}(x)-\mathds{1}_{(1/8,1/4]}(x)\right)$ and then $g_{0,2}(x)=2(\mathds{1}_{[0,1/16]}(x)-\mathds{1}_{(1/16,1/8]}(x)+\mathds{1}_{(1/8,3/16]}(x)-\mathds{1}_{(3/16,1/4]}(x))$ we obtain the frames depicted in Figure~\ref{fig:figure621} and \ref{fig:figure622}.

\begin{figure}
\begin{tabular}{|c|c|c|}
	\hline
	\includegraphics[width=5cm]{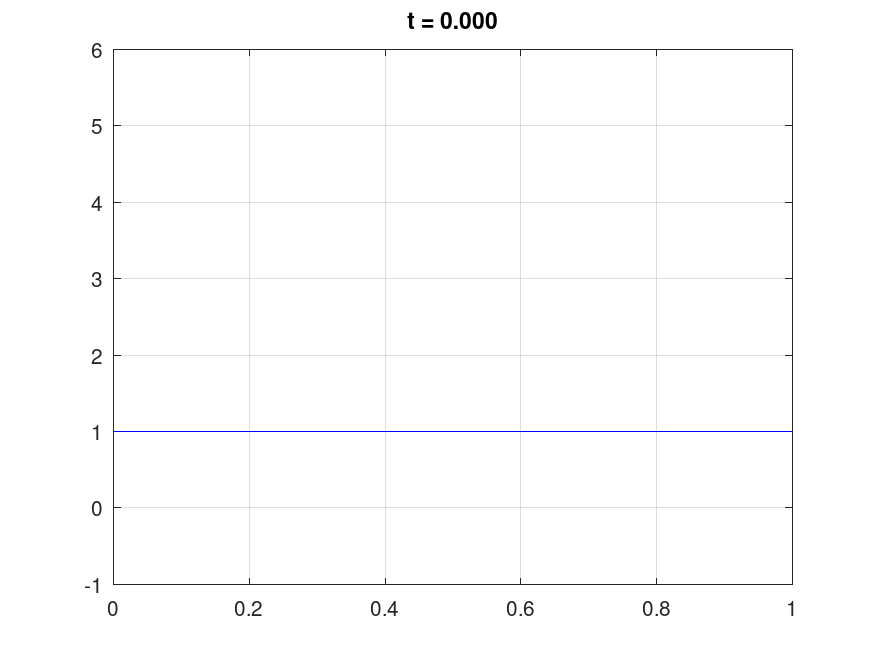} & \includegraphics[width=5cm]{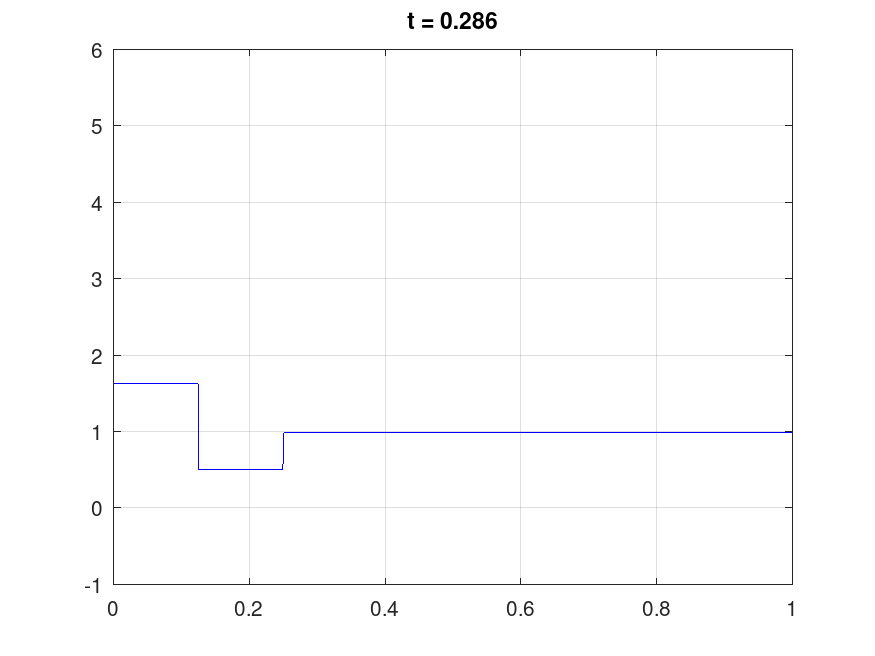} &  \includegraphics[width=5cm]{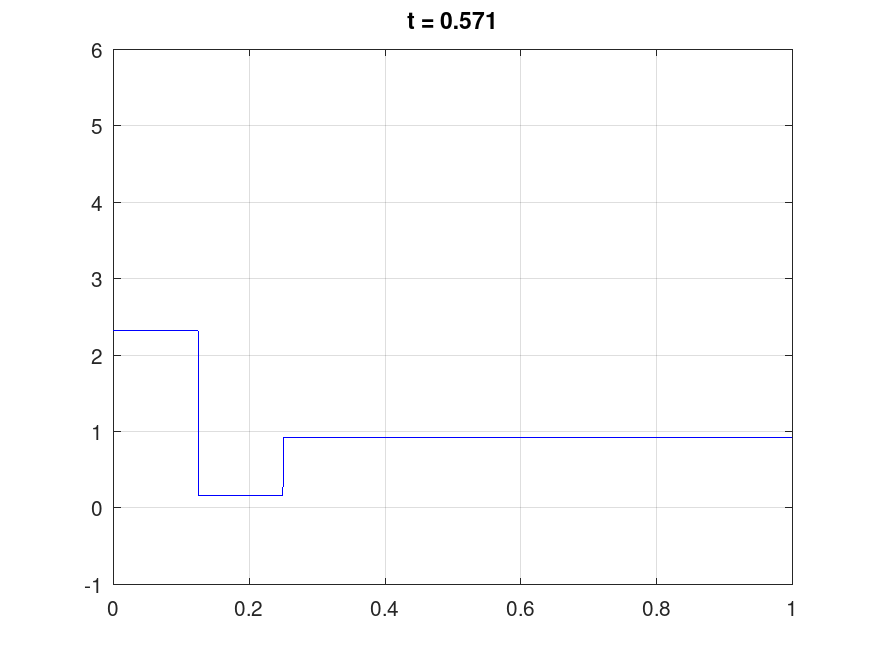}\\
	\hline
	\includegraphics[width=5cm]{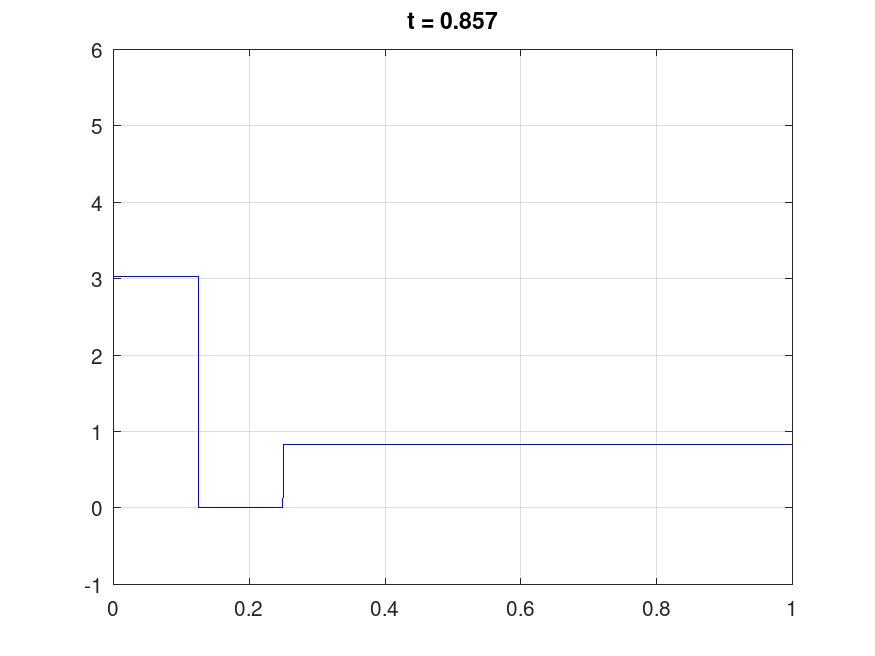} &
	\includegraphics[width=5cm]{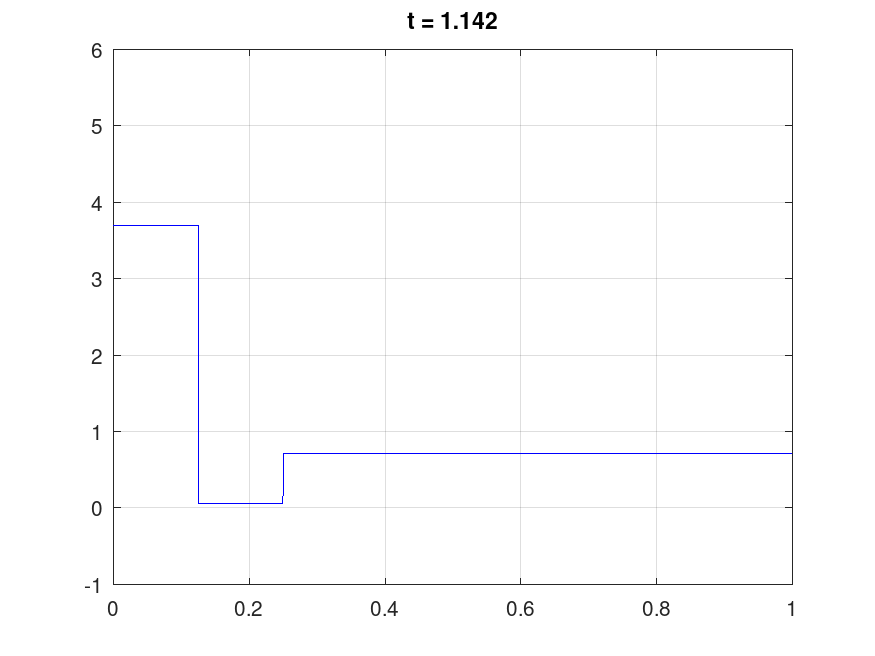} & \includegraphics[width=5cm]{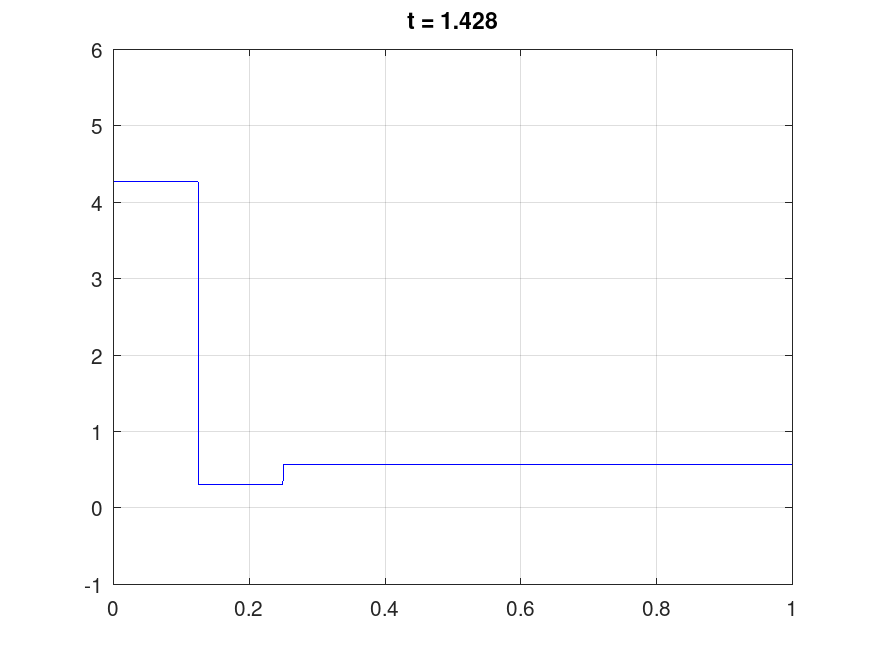}  \\
	\hline
	\includegraphics[width=5cm]{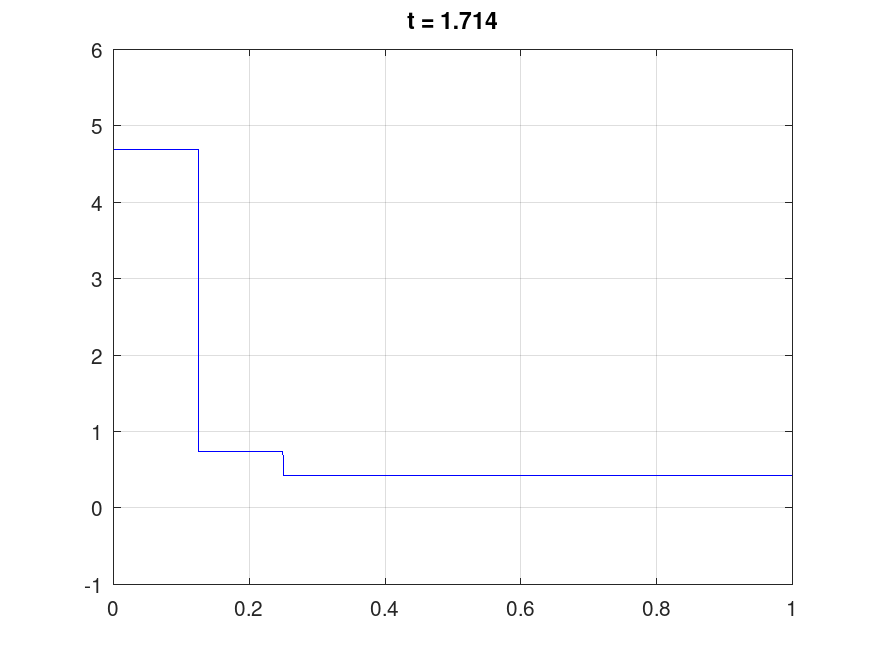} &
	\includegraphics[width=5cm]{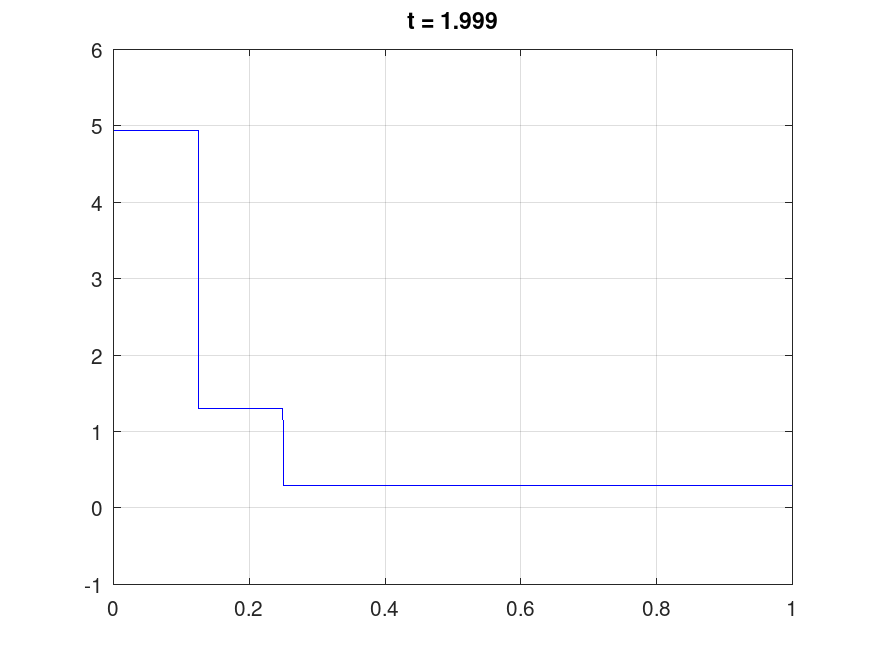} & \includegraphics[width=5cm]{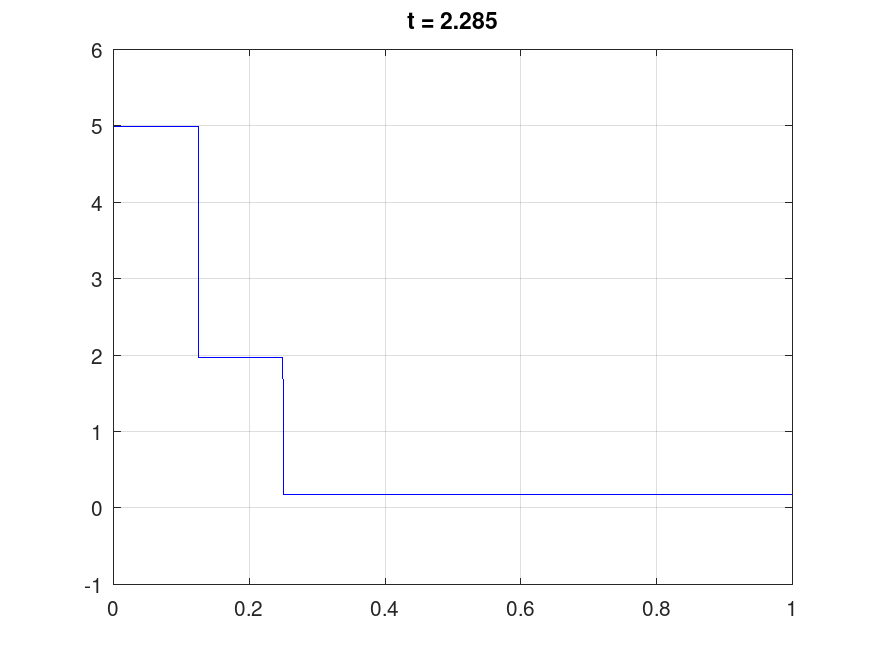} \\
	\hline
	\includegraphics[width=5cm]{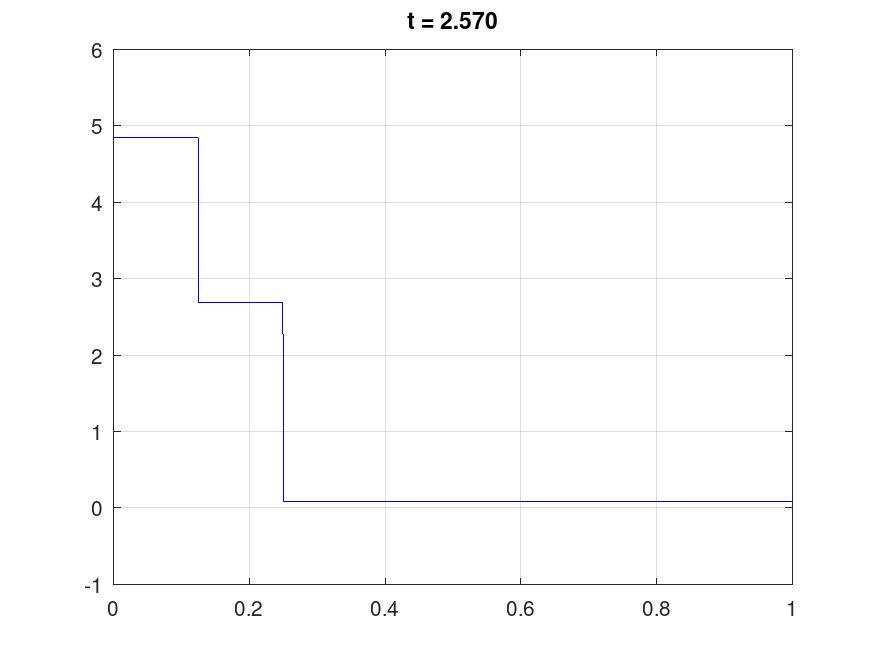} &
	\includegraphics[width=5cm]{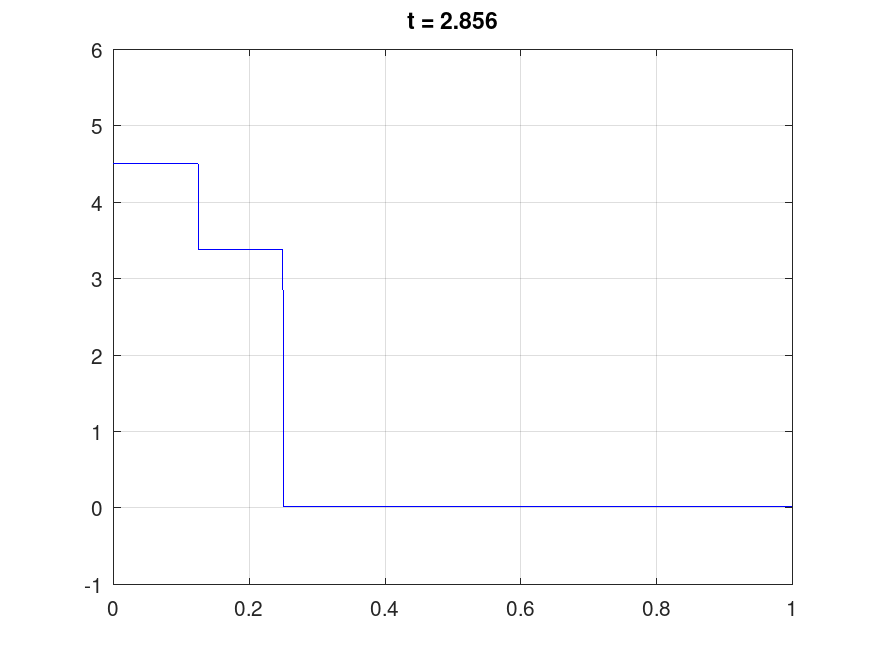} & \includegraphics[width=5cm]{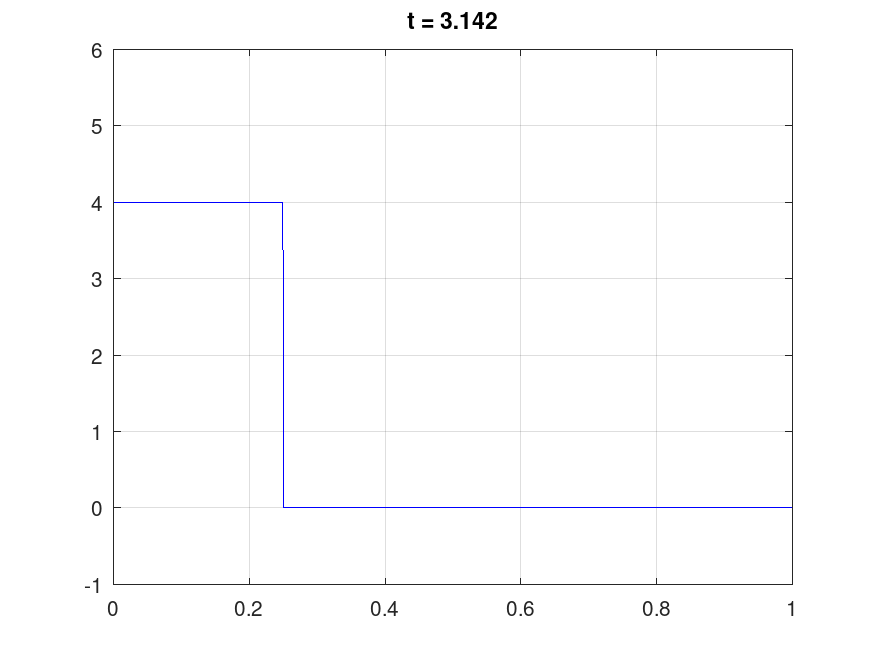} \\
	\hline
\end{tabular}
\caption{$g_{0, 1}$}\label{fig:figure621}
\end{figure}
\begin{figure}
\begin{tabular}{|c|c|c|}
	\hline
	\includegraphics[width=5cm]{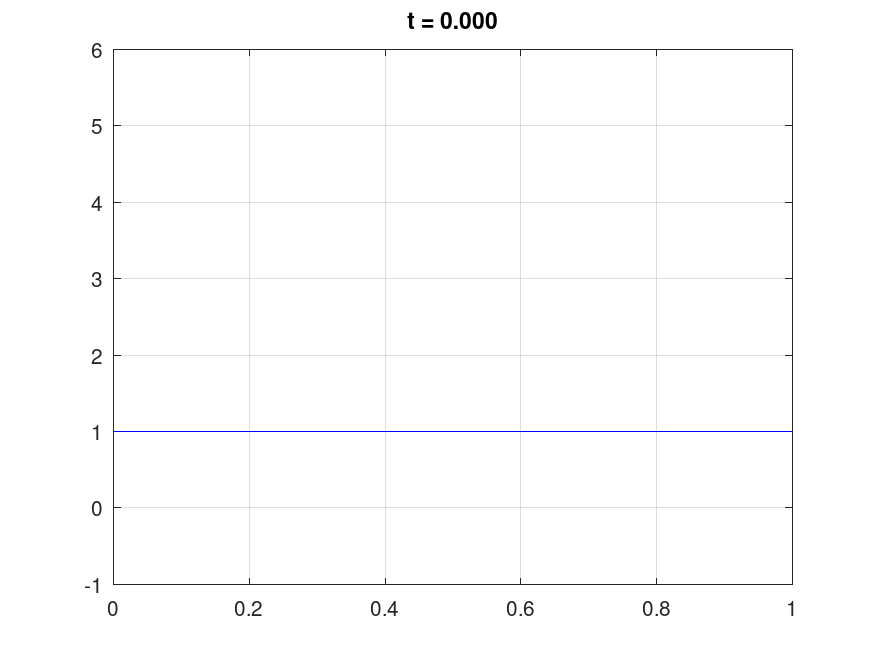} & \includegraphics[width=5cm]{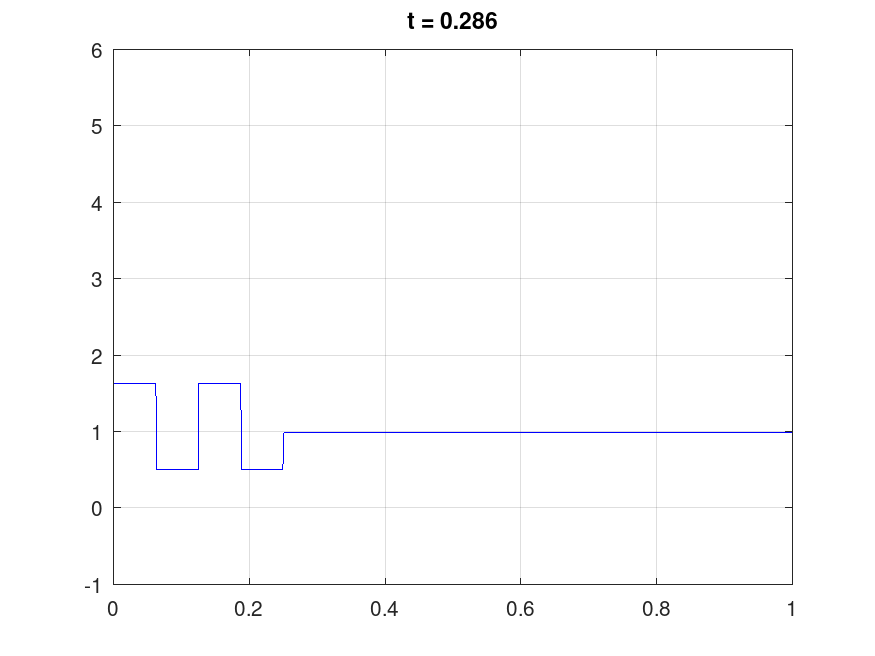} &  \includegraphics[width=5cm]{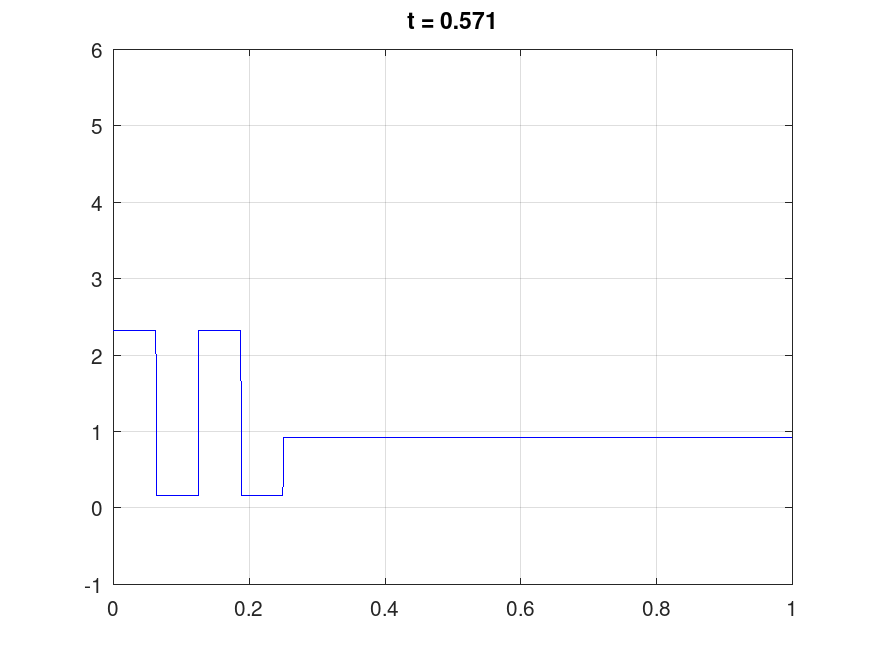}\\
	\hline
	\includegraphics[width=5cm]{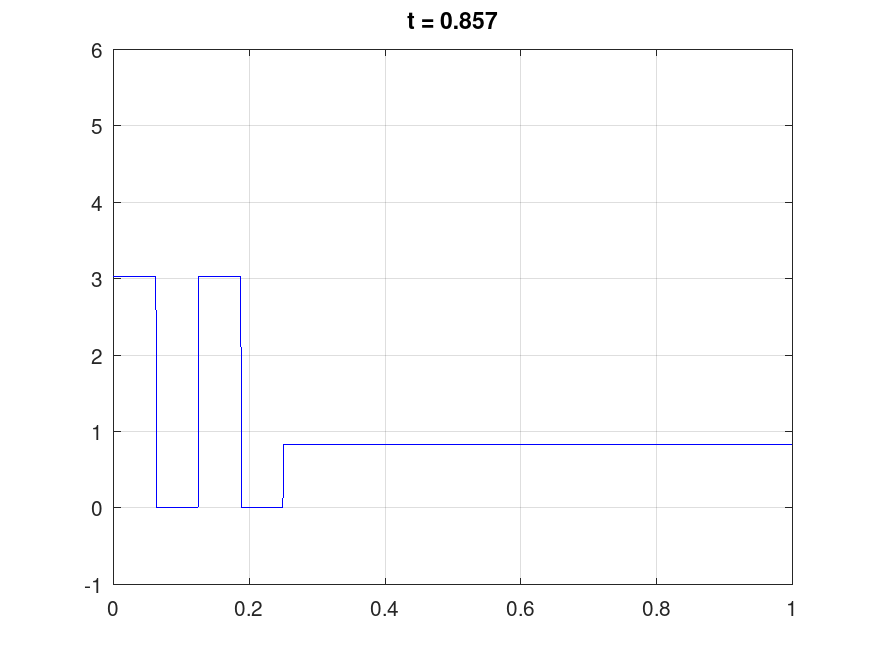} &
	\includegraphics[width=5cm]{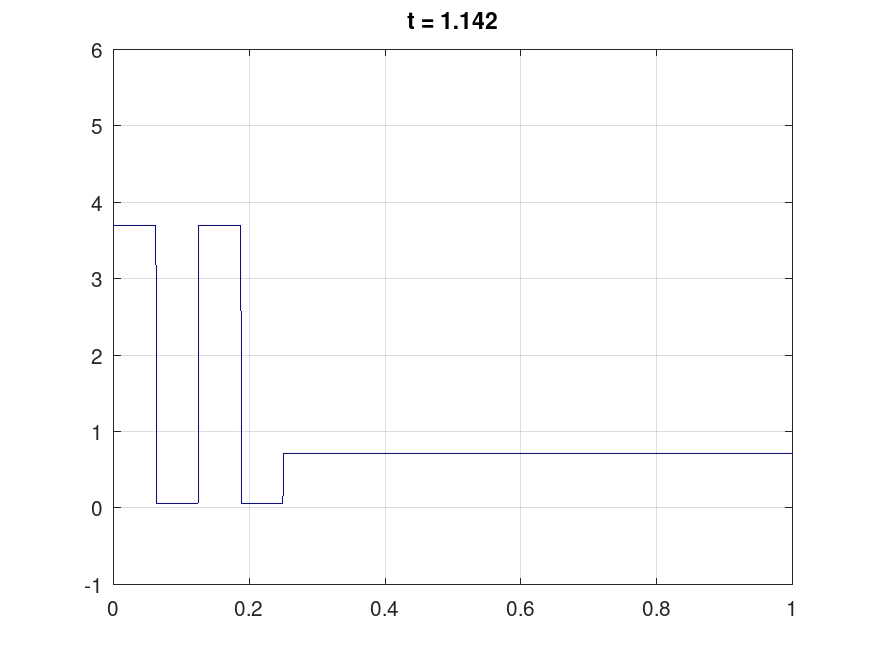} & \includegraphics[width=5cm]{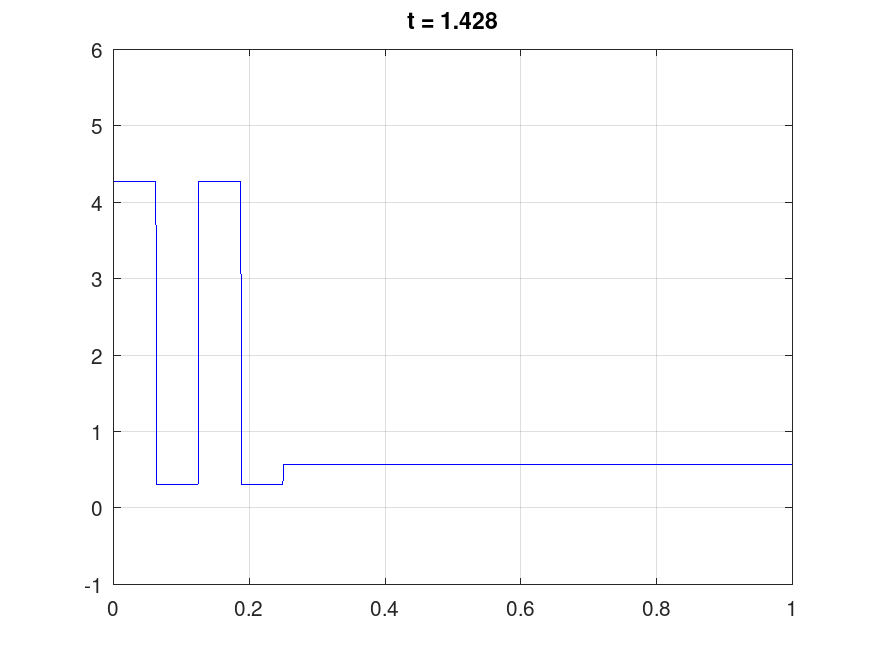}  \\
	\hline
	\includegraphics[width=5cm]{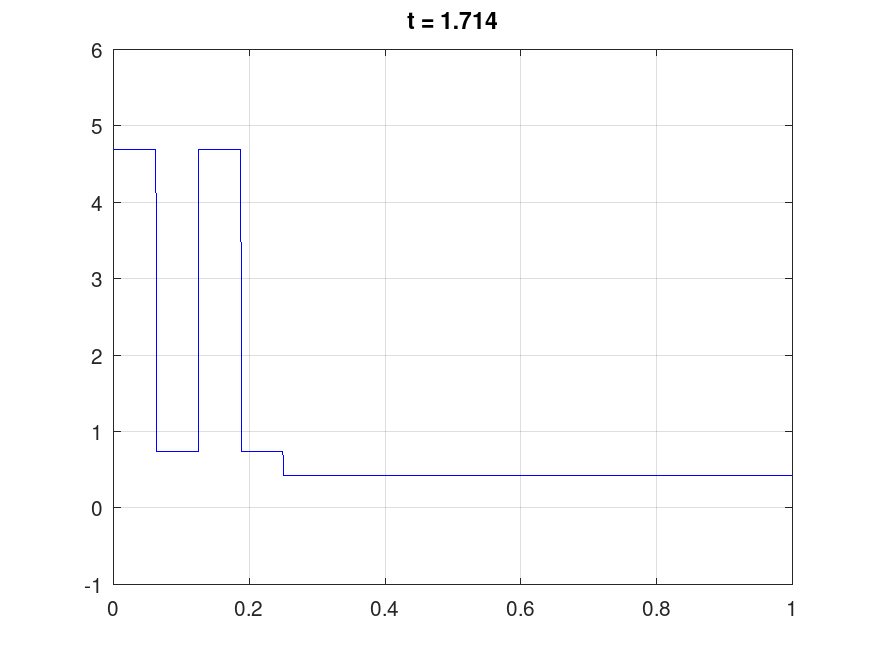} &
	\includegraphics[width=5cm]{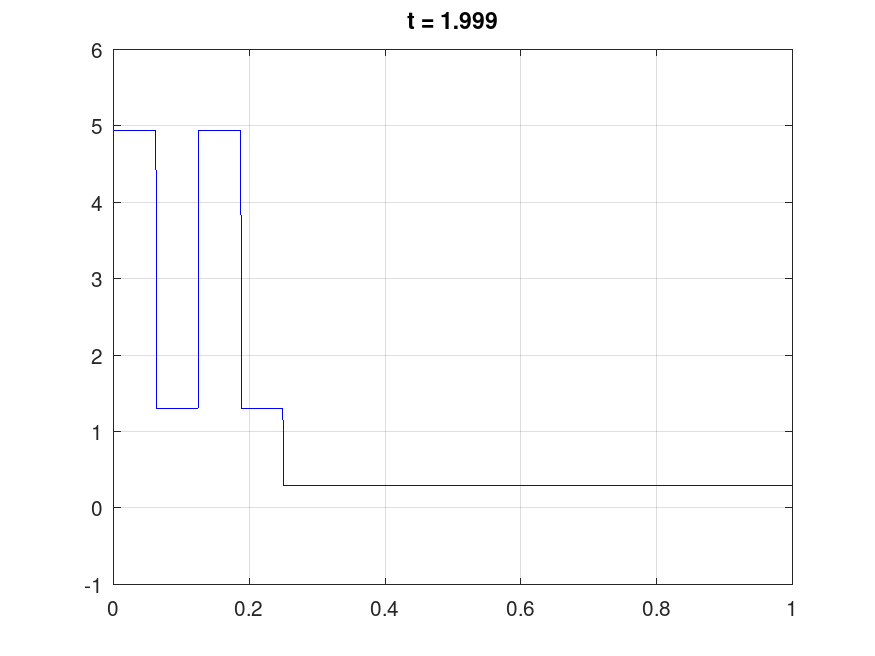} & \includegraphics[width=5cm]{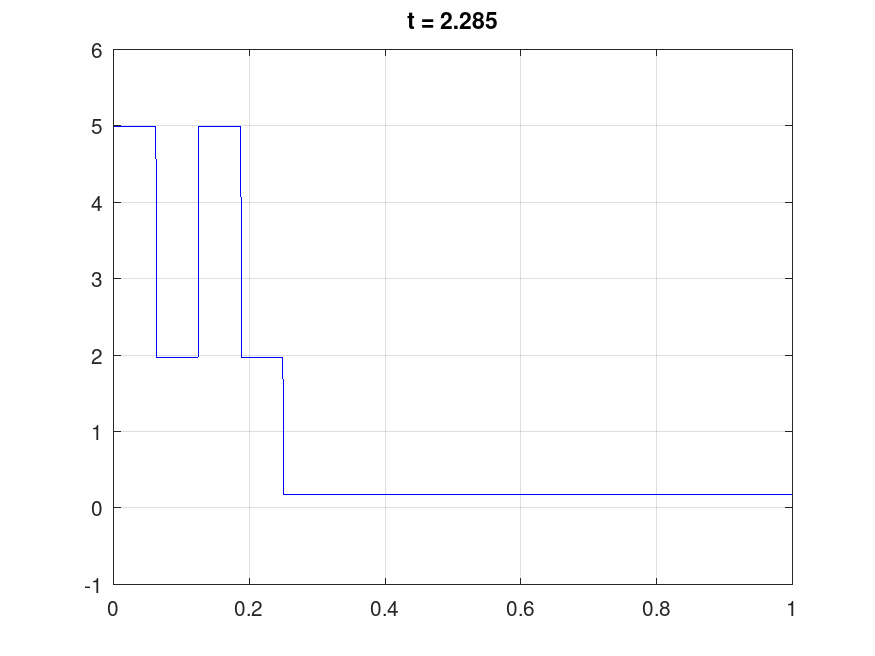} \\
	\hline
	\includegraphics[width=5cm]{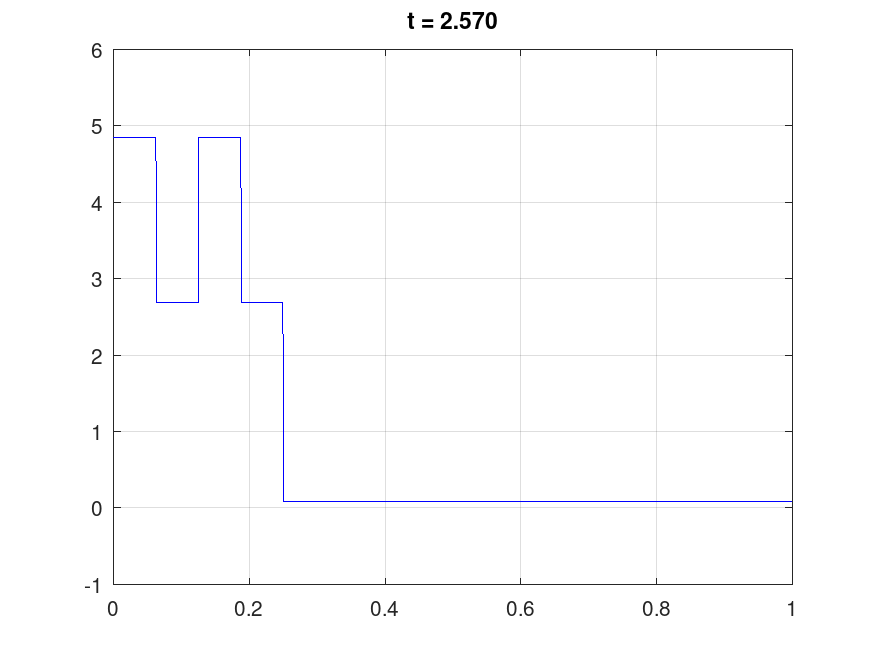} &
	\includegraphics[width=5cm]{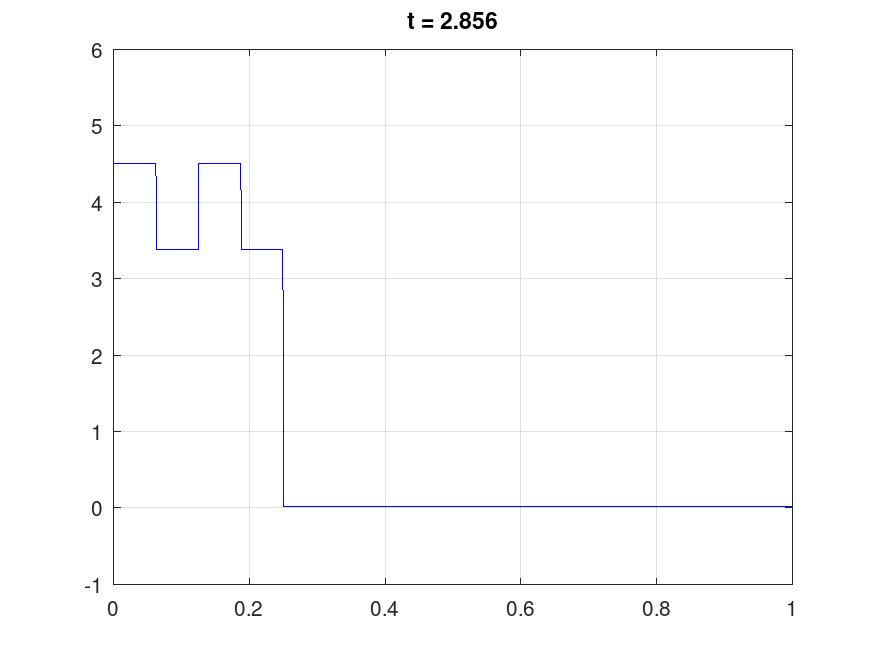} & \includegraphics[width=5cm]{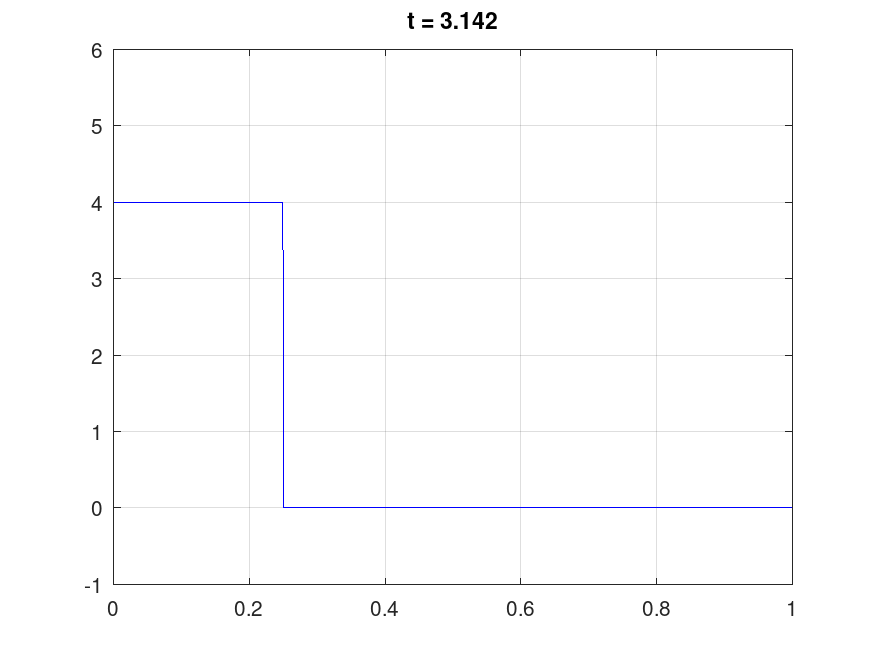} \\
	\hline
\end{tabular}
\caption{$g_{0, 2}$}\label{fig:figure622}
\end{figure}

It is clear that the density trajectory generated by the  Riemann Fisher geometry depends strongly on the function $g_0$ even when $f_0$ and $\frac{g_0^2}{f_0}$ are fixed.
This situation can be observed in even more clearly in the two dimensional case. In this situation $X=[0,1]^2$, $d\mu=dx dy$ is the area measure in the unit square. Take now $f_0(x,y)\equiv 1$ the uniform density and $g_0(x,y)$ a function in $[0,1]^2$ such that $\iint g_0 dx dy=0$ and $\frac{g_0^2}{f_0}=16\,\mathds{1}_{[0,1/4]^2}(x,y)$. In Figure~\ref{fig:figure631}, Figure~\ref{fig:figure632} and Figure~\ref{fig:figure633}, we depict the situation for three different choices of $g_0$:
\begin{align*}
	g_{0 1}(x,y) &=  4\left(\mathds{1}_{[0,1/4]\times [0,1/8]}(x,y)-\mathds{1}_{[0,1/4]\times [1/8,1/4]}(x,y)\right)\\
	g_{0 2}(x,y) &=  4\left(\mathds{1}_{[0,1/8]\times [0,1/8]\bigcup [1/8,1/4]\times [1/8,1/4]}(x,y)-\mathds{1}_{[1/8,1/4]\times [0,1/8]\bigcup [0,1/8]\times [1/8,1/4]}(x,y)\right)\\
	g_{0 3}(x,y) &=  4\left(\mathds{1}_E(x,y)-\mathds{1}_{[0,1]^2\setminus E}(x,y)\right)
\end{align*}
with 
\begin{equation*}
E=\bigcup_{\substack{k_1=0,1,2,3\\ k_2=0,1,2,3\\  k_1+k_2 \textrm{ even }}} Q_{k_1,k_2},	
\end{equation*}
and $Q_{k_1,k_2}=[0,\tfrac{1}{16}]^2 + (k_1,k_2)$.
\begin{figure}
\begin{tabular}{|c|c|c|}
	\hline
	\includegraphics[width=5cm]{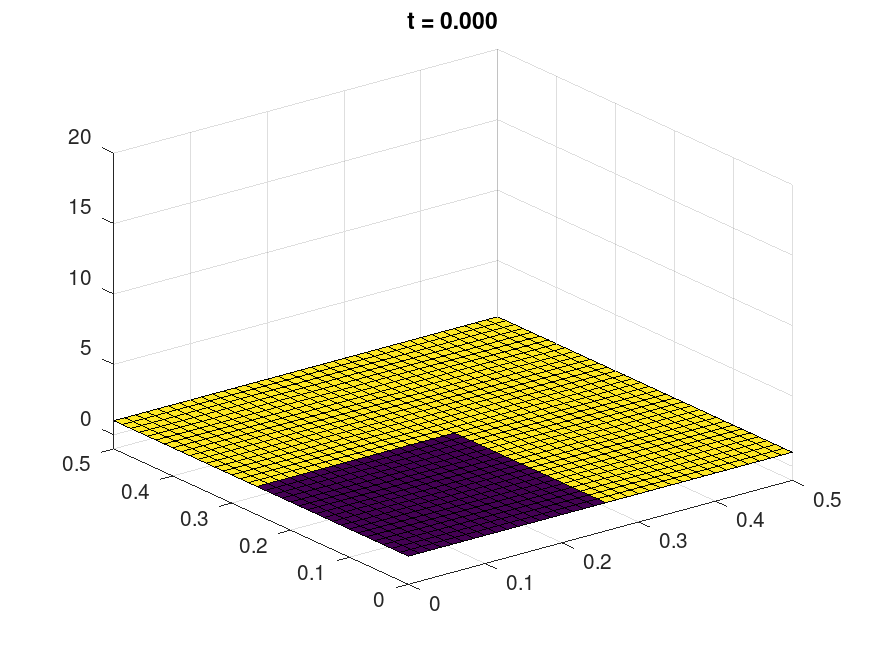} & \includegraphics[width=5cm]{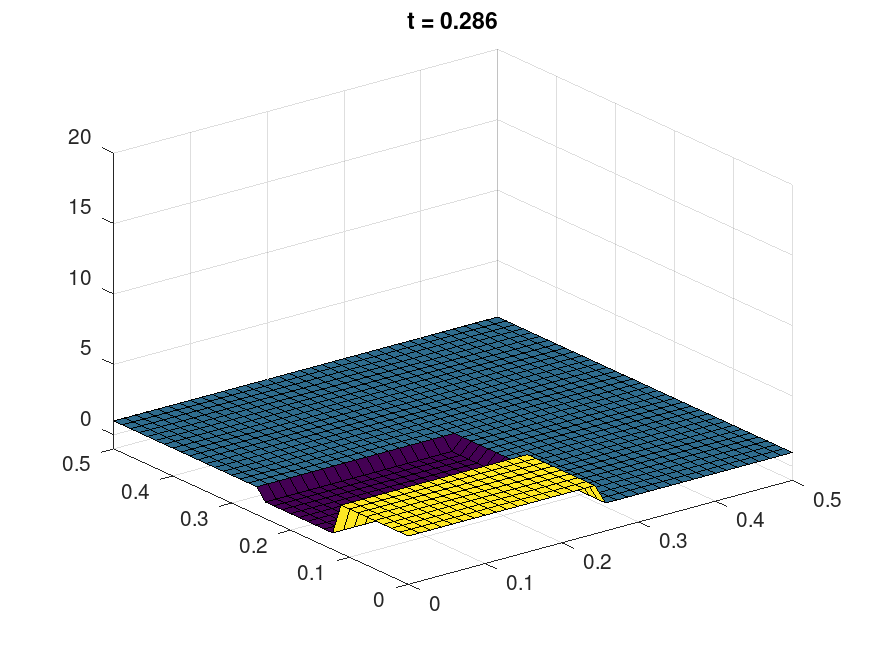} &  \includegraphics[width=5cm]{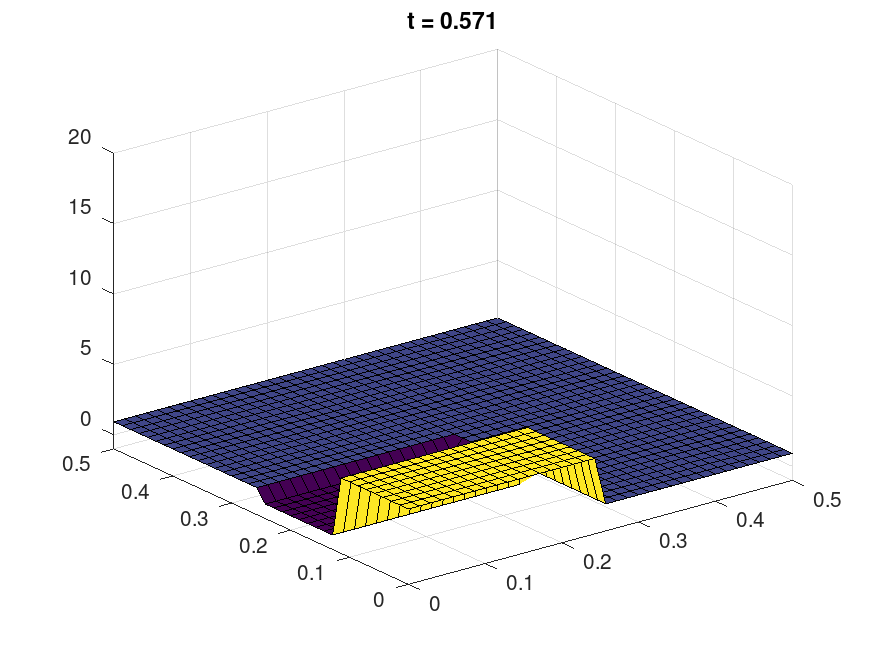}\\
	\hline
	\includegraphics[width=5cm]{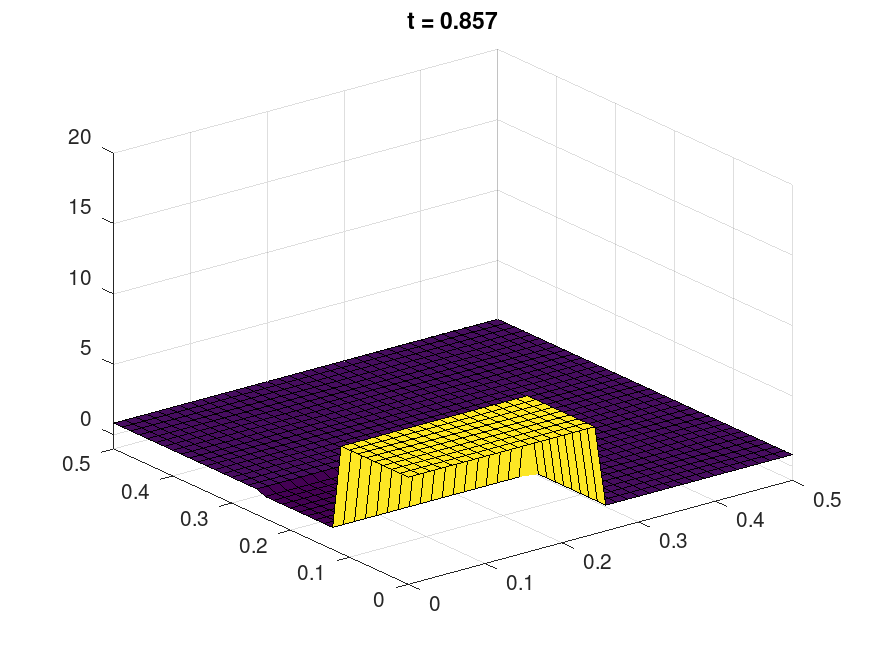} &
	\includegraphics[width=5cm]{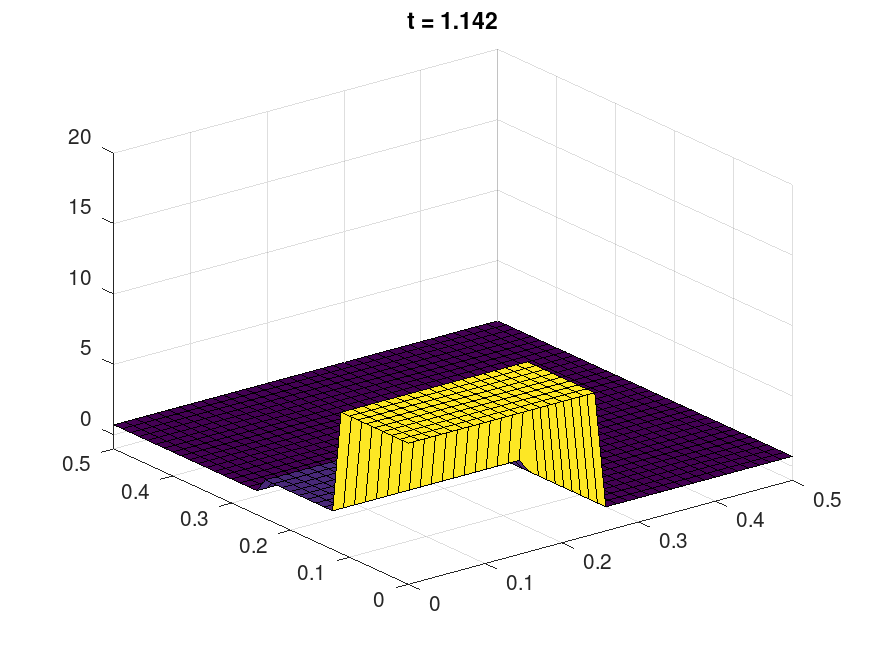} & \includegraphics[width=5cm]{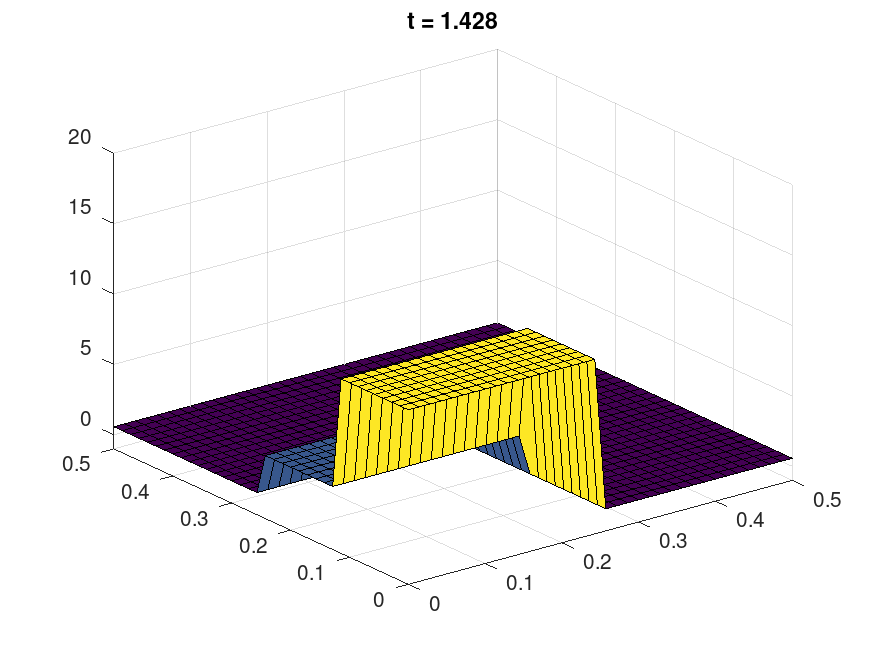}  \\
	\hline
	\includegraphics[width=5cm]{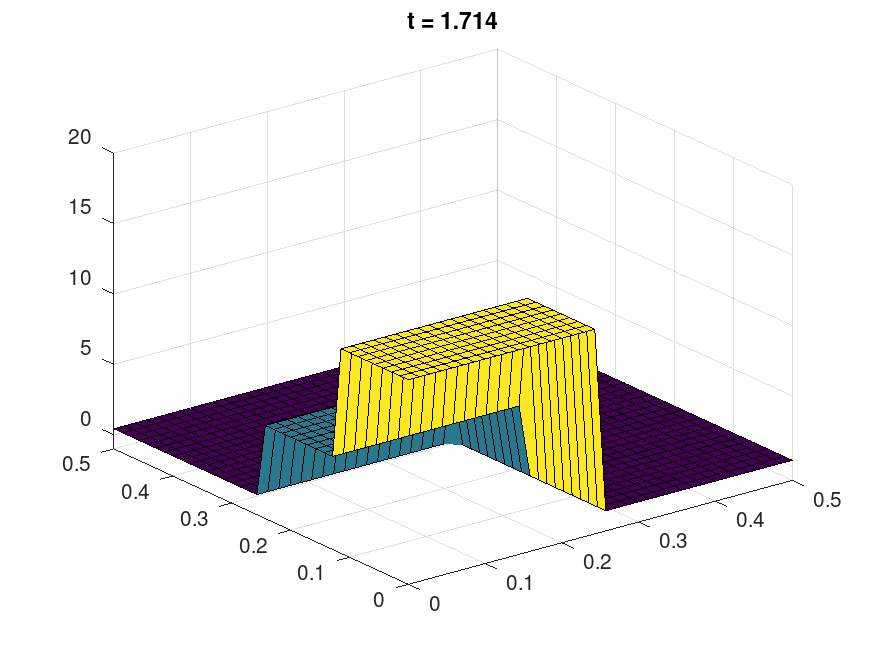} &
	\includegraphics[width=5cm]{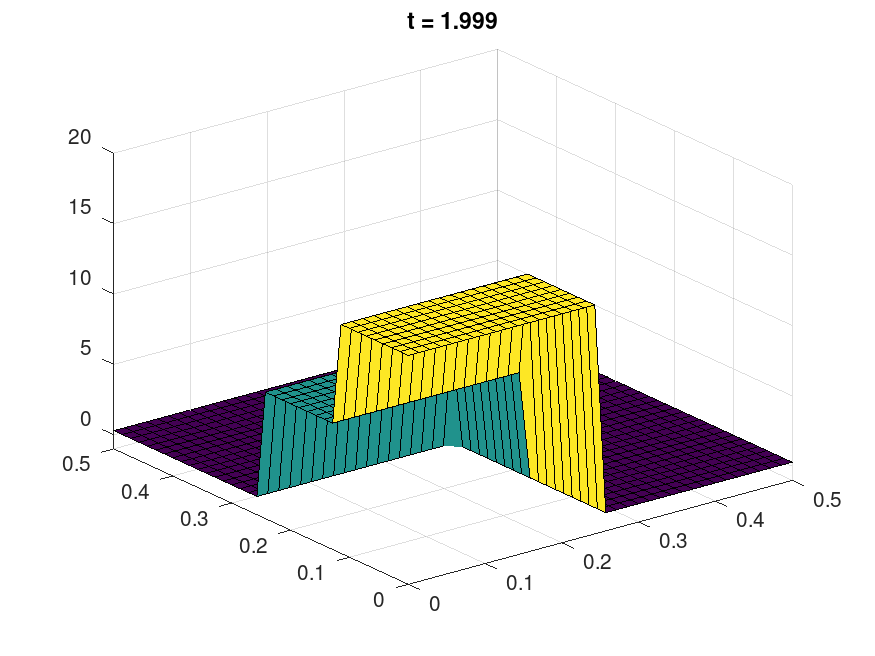} & \includegraphics[width=5cm]{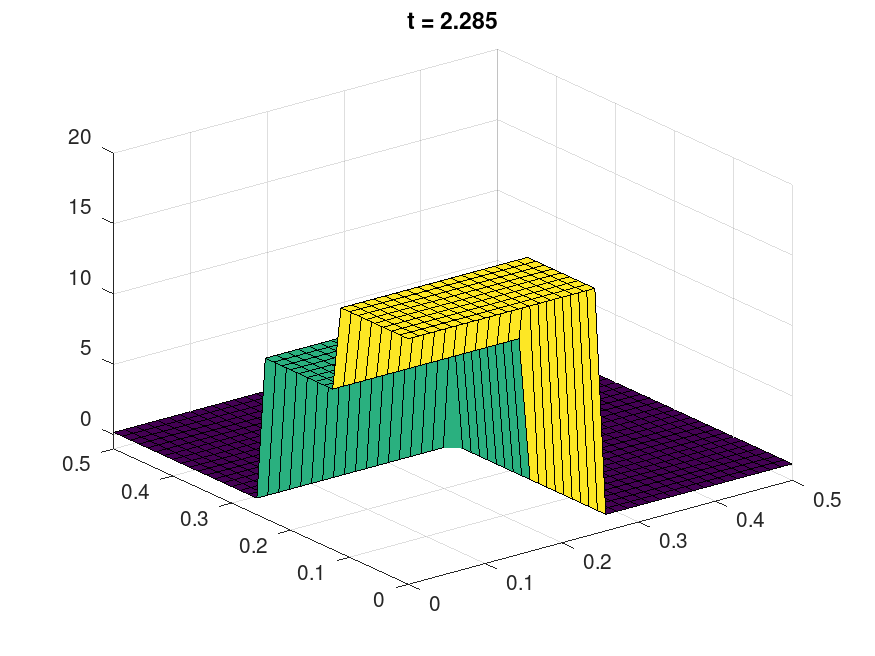} \\
	\hline
	\includegraphics[width=5cm]{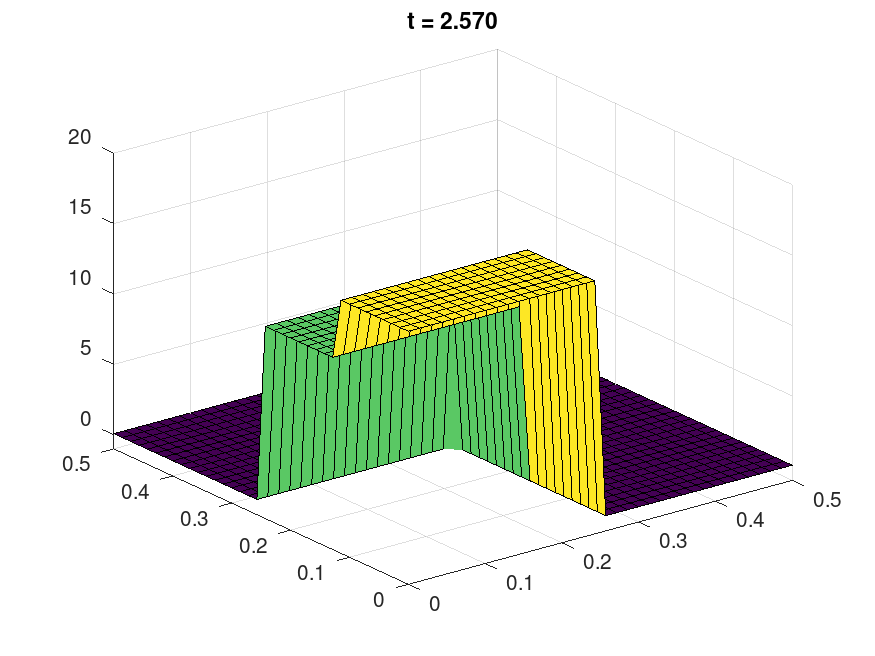} &
	\includegraphics[width=5cm]{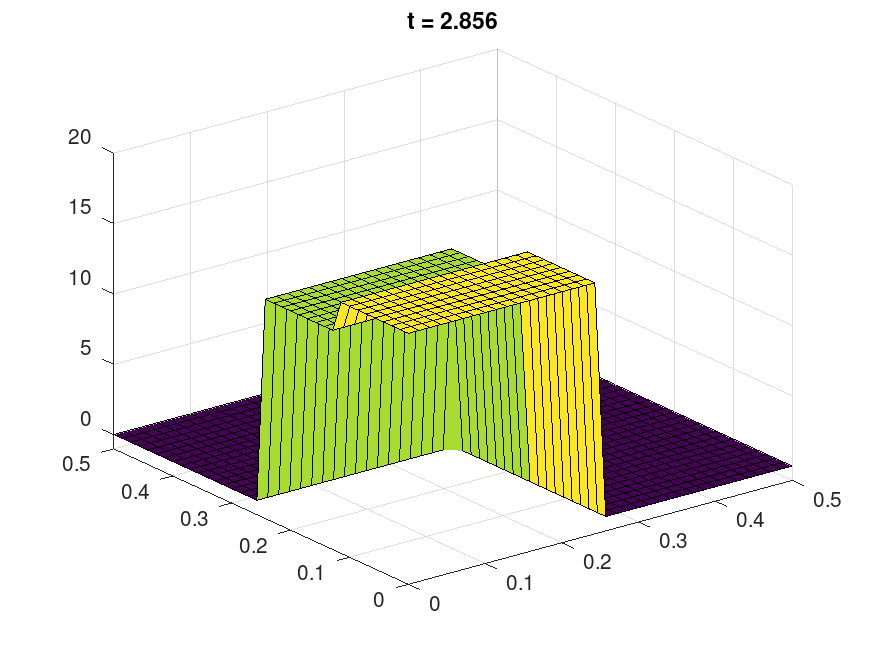} & \includegraphics[width=5cm]{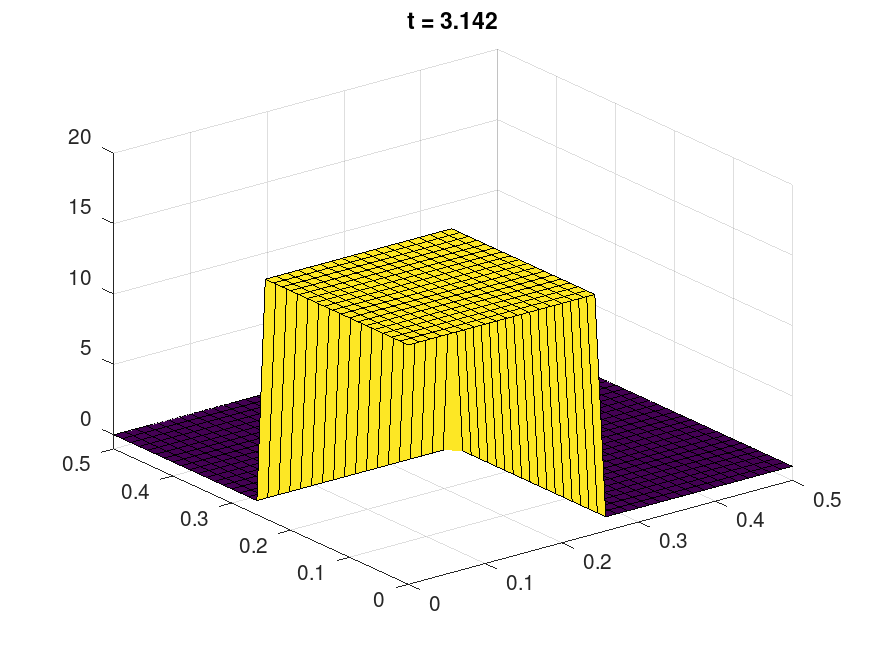} \\
	\hline
\end{tabular}
\caption{$g_{0 1}$}\label{fig:figure631}
\end{figure}
\begin{figure}
\begin{tabular}{|c|c|c|}
	\hline
	\includegraphics[width=5cm]{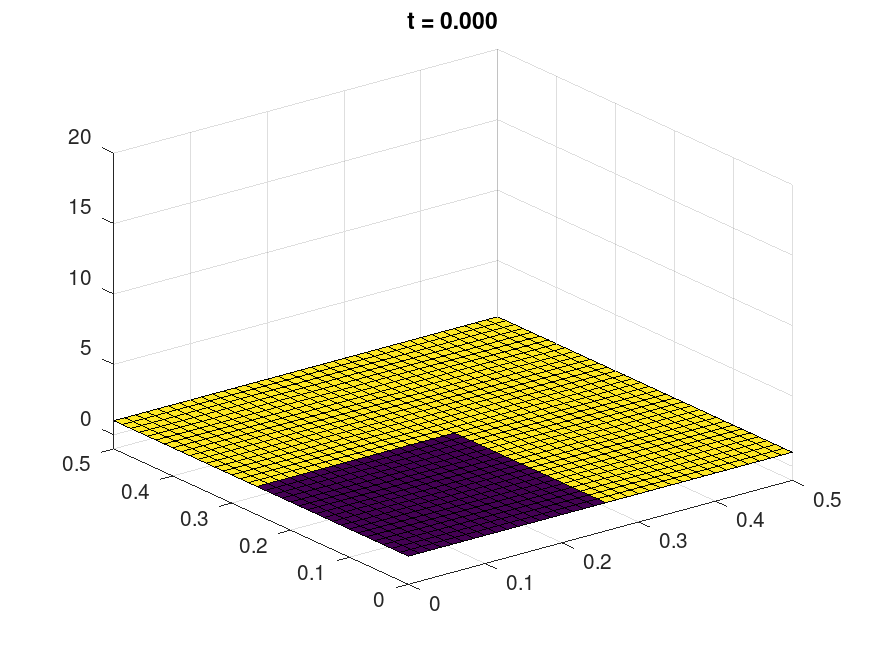} & \includegraphics[width=5cm]{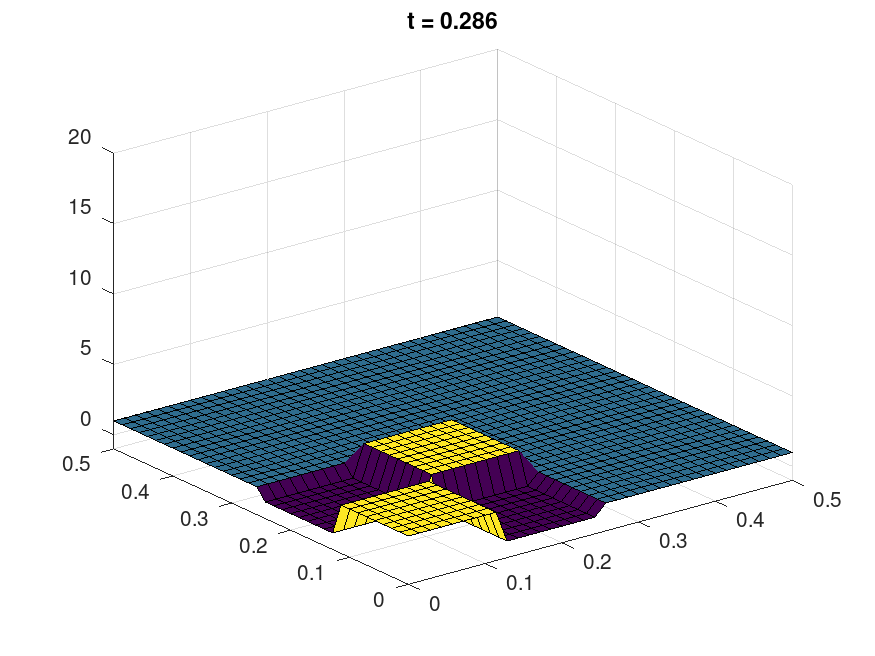} &  \includegraphics[width=5cm]{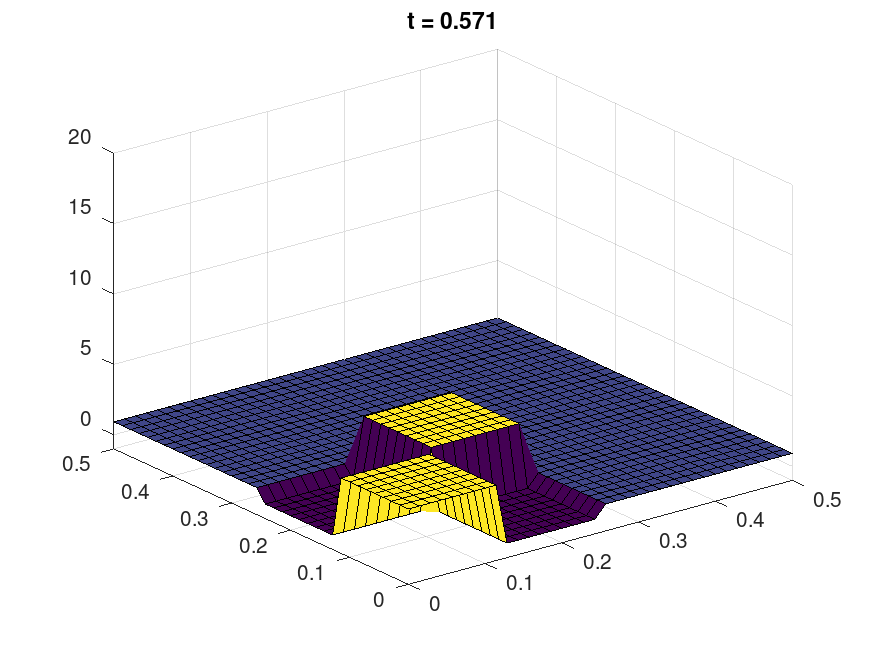}\\
	\hline
	\includegraphics[width=5cm]{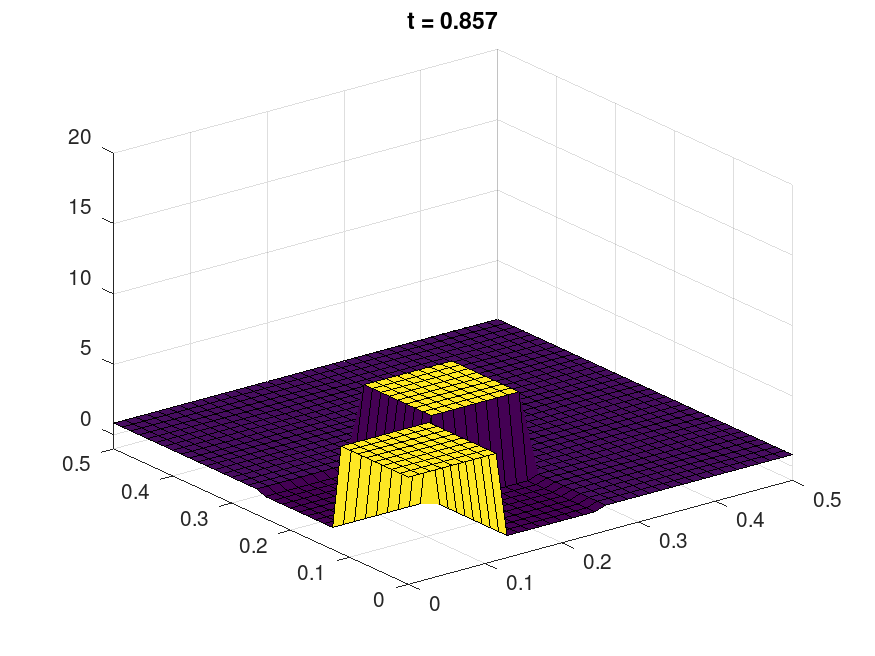} &
	\includegraphics[width=5cm]{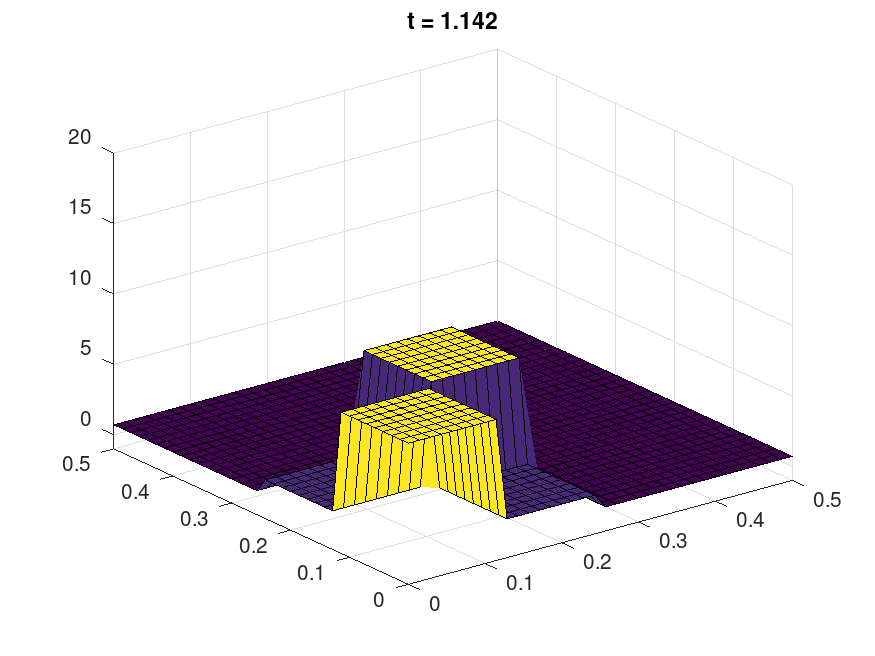} & \includegraphics[width=5cm]{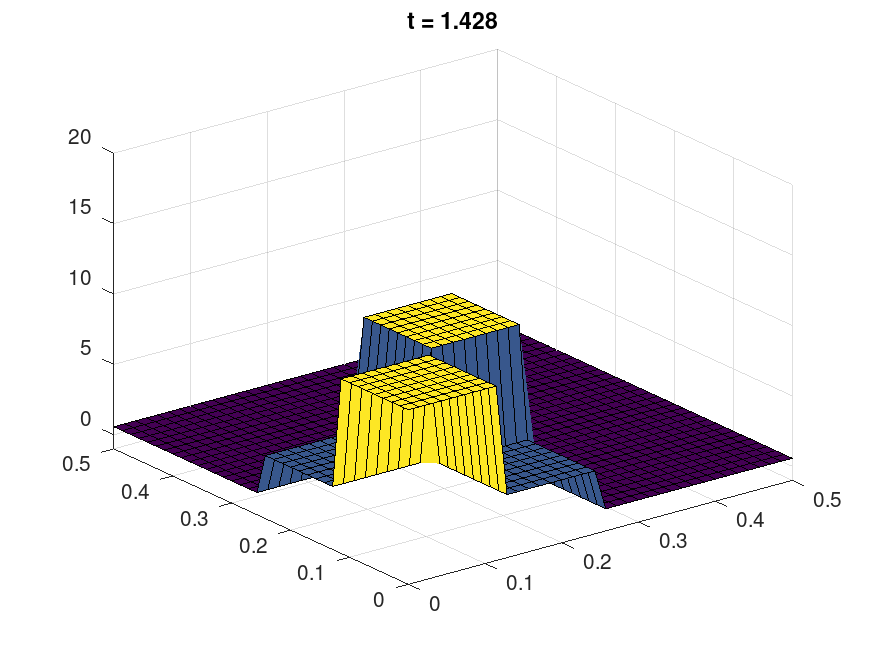}  \\
	\hline
	\includegraphics[width=5cm]{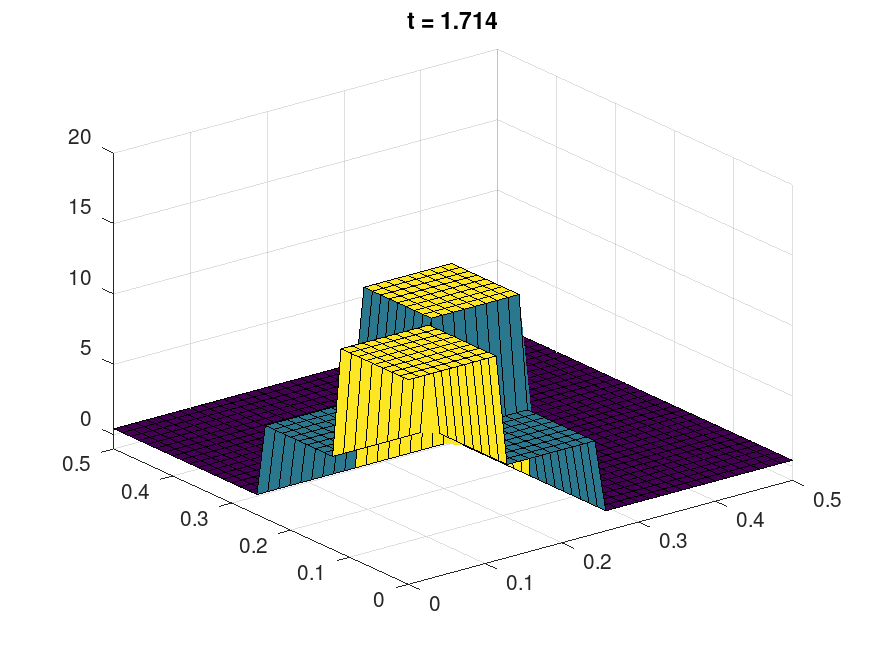} &
	\includegraphics[width=5cm]{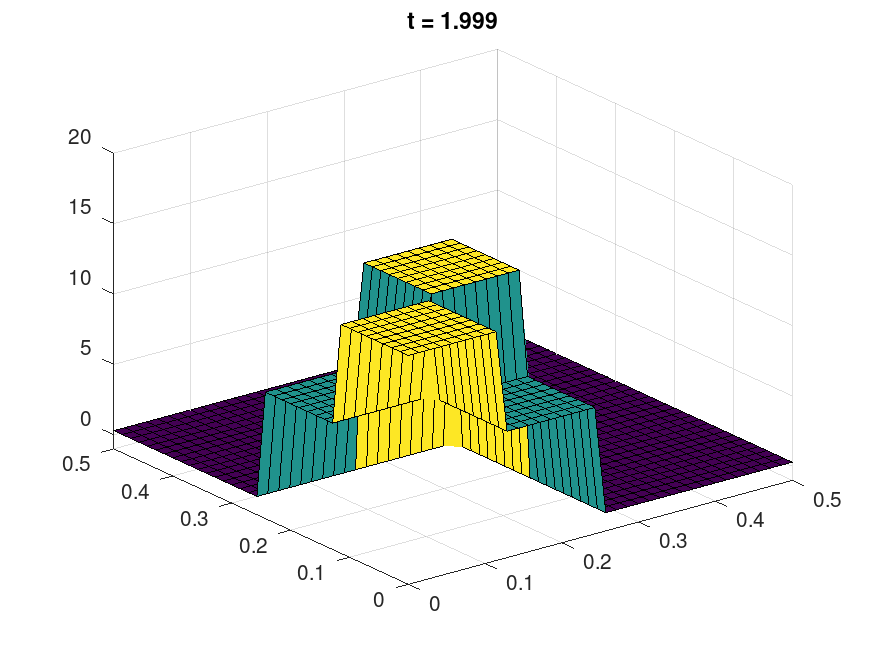} & \includegraphics[width=5cm]{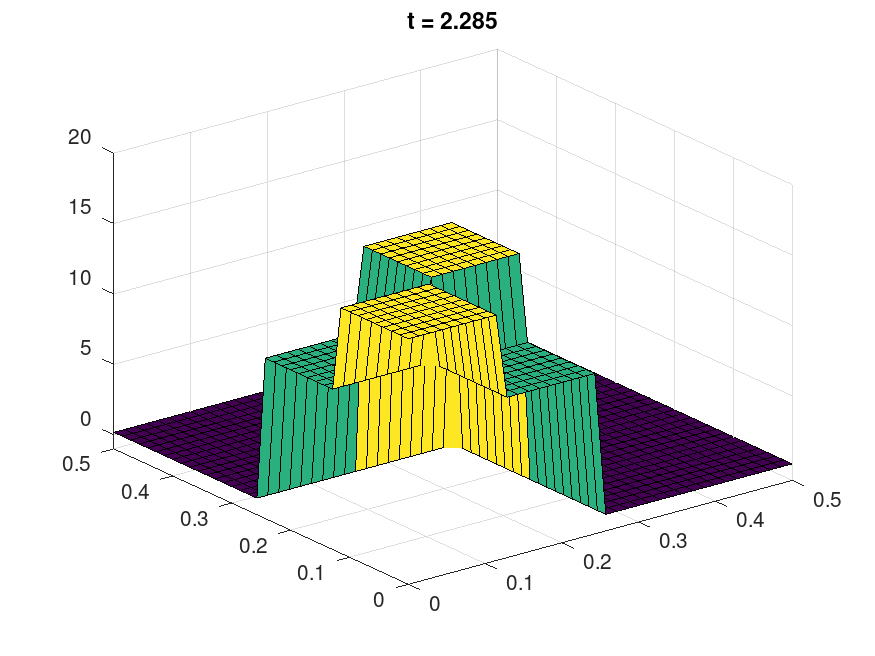} \\
	\hline
	\includegraphics[width=5cm]{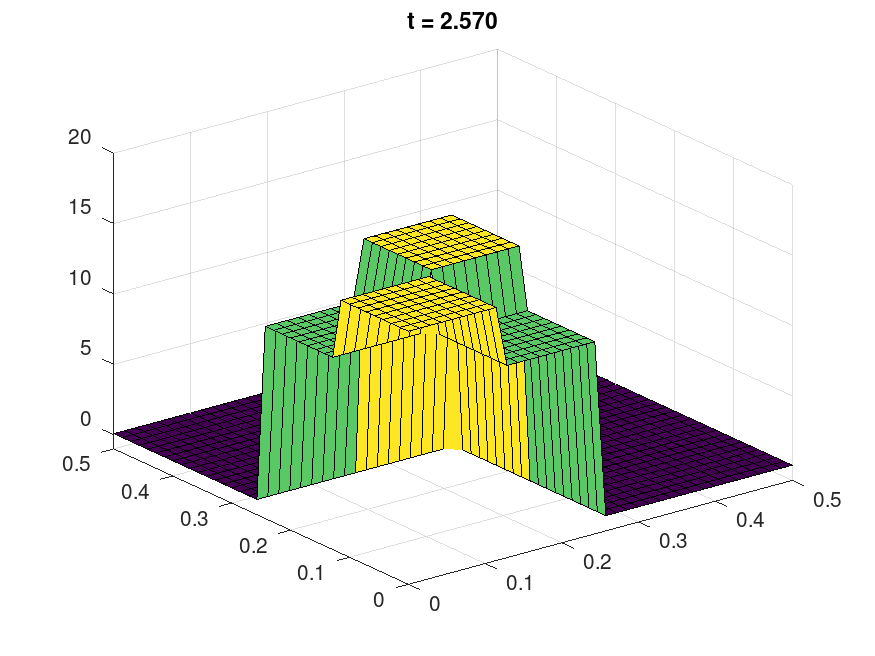} &
	\includegraphics[width=5cm]{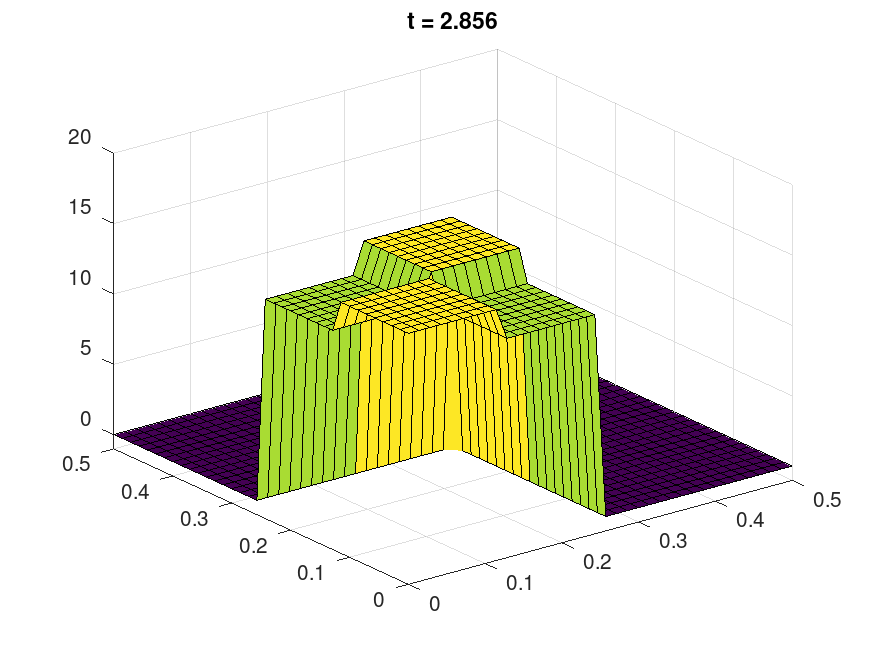} & \includegraphics[width=5cm]{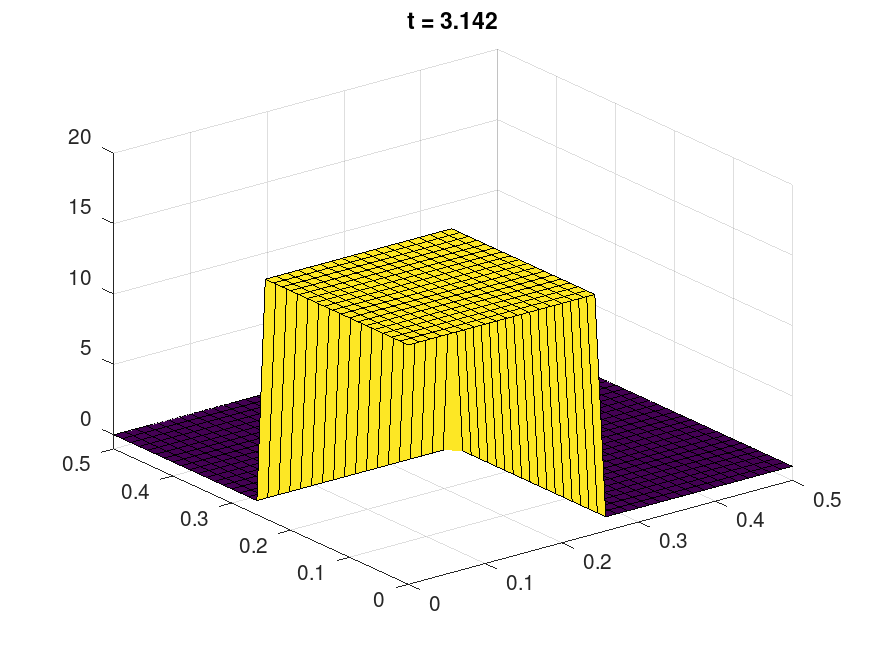} \\
	\hline
\end{tabular}
\caption{$g_{0 2}$}\label{fig:figure632}
\end{figure} 
\begin{figure}
\begin{tabular}{|c|c|c|}
	\hline
	\includegraphics[width=5cm]{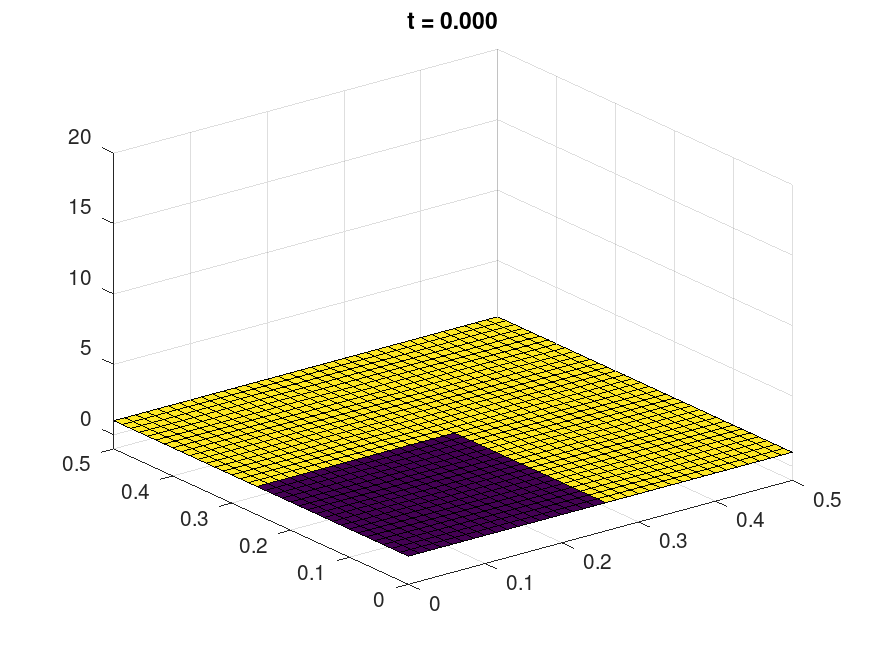} & \includegraphics[width=5cm]{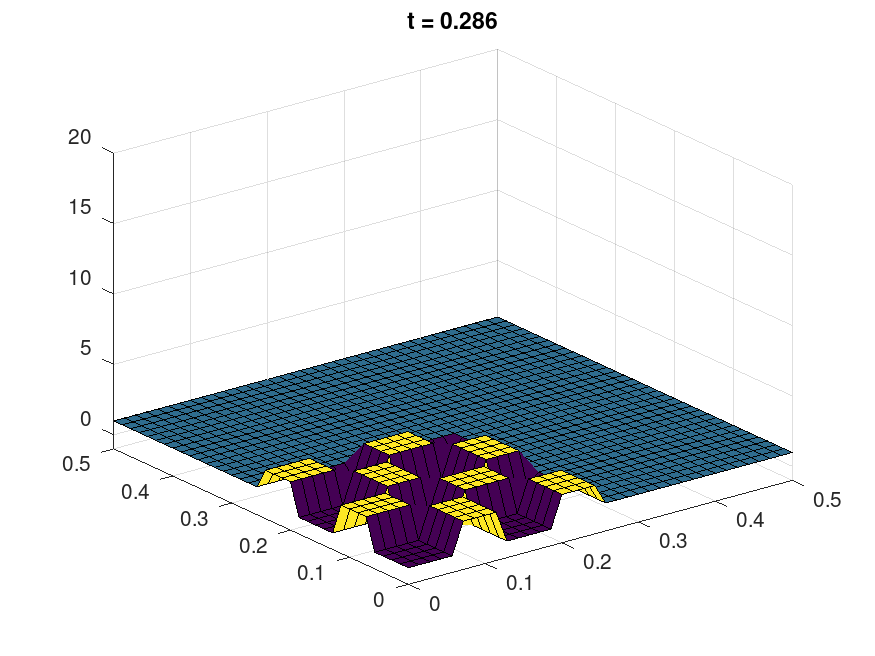} &  \includegraphics[width=5cm]{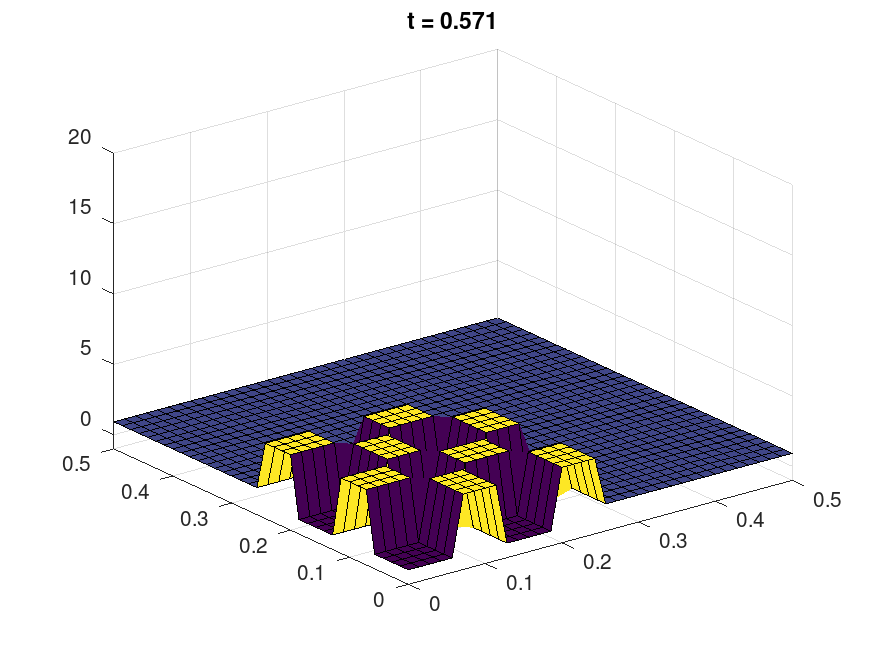}\\
	\hline
	\includegraphics[width=5cm]{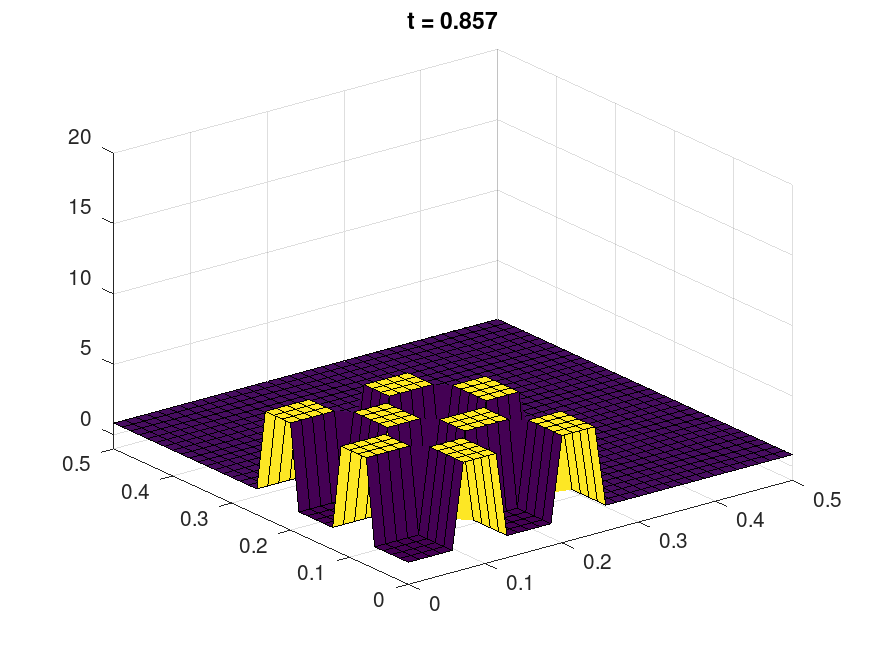} &
	\includegraphics[width=5cm]{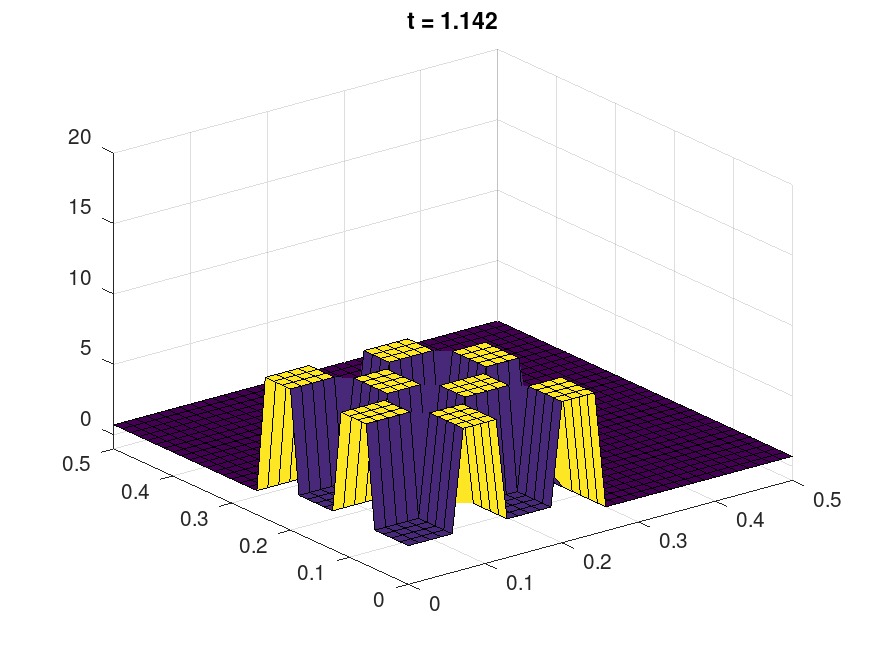} & \includegraphics[width=5cm]{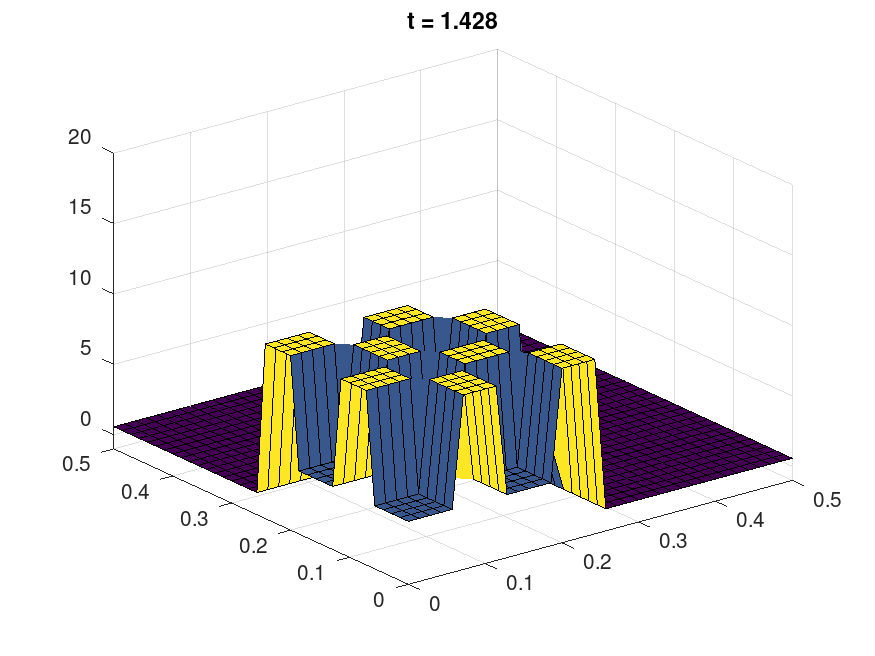}  \\
	\hline
	\includegraphics[width=5cm]{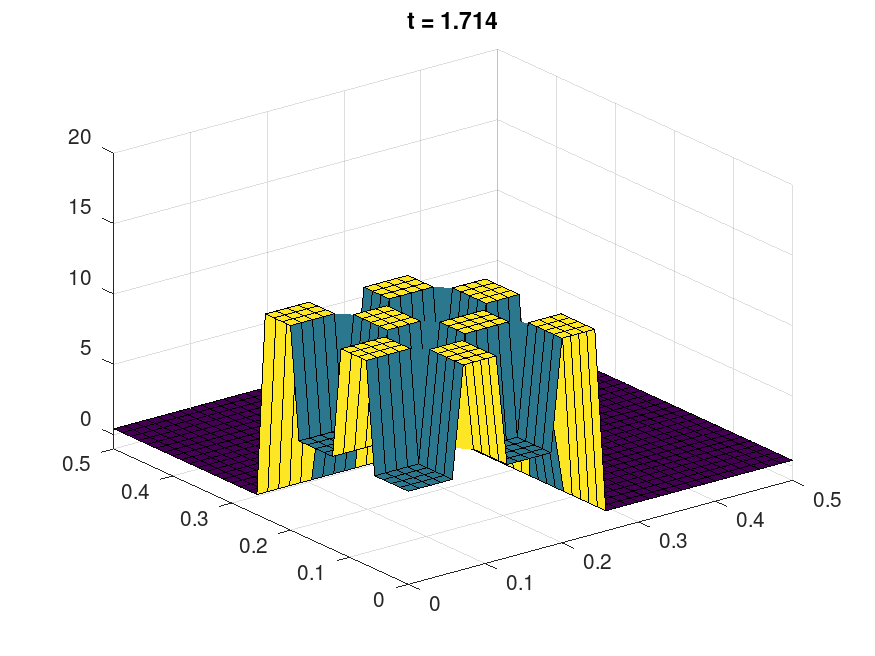} &
	\includegraphics[width=5cm]{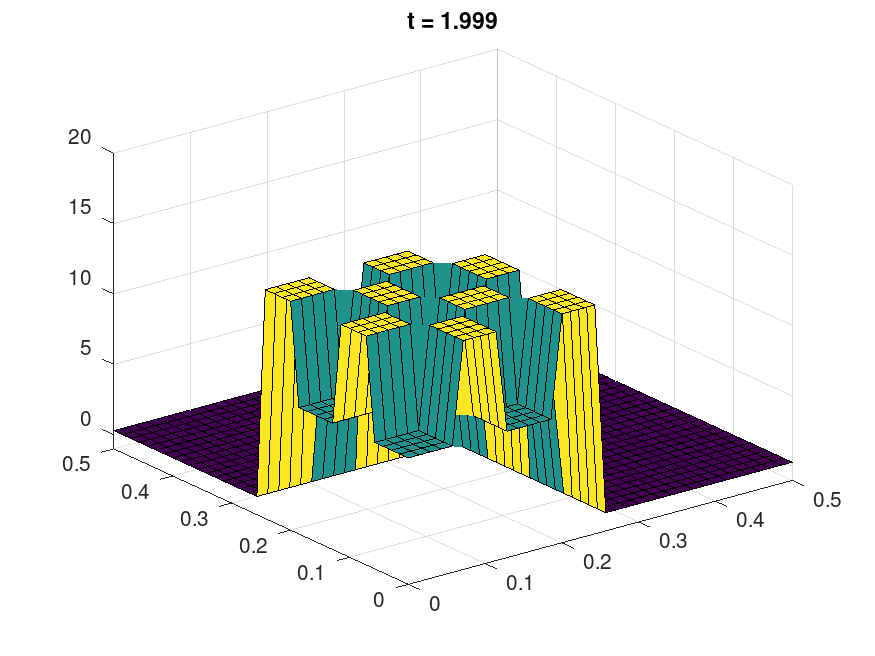} & \includegraphics[width=5cm]{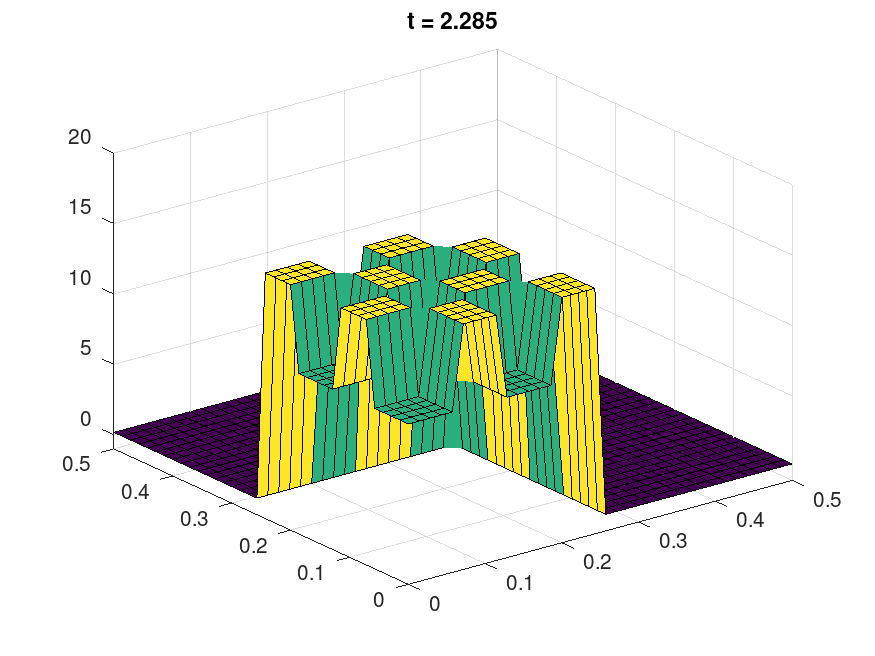} \\
	\hline
	\includegraphics[width=5cm]{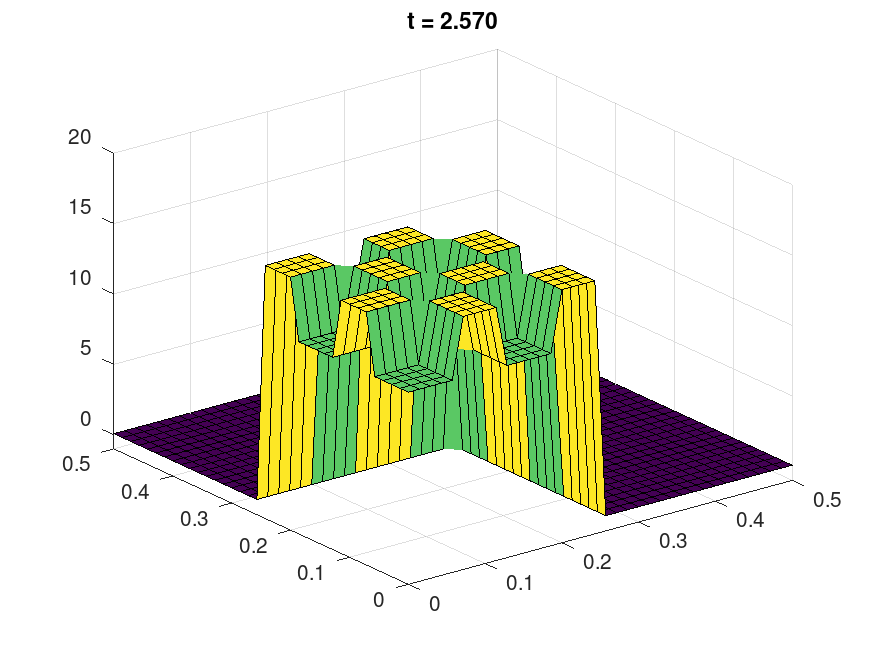} &
	\includegraphics[width=5cm]{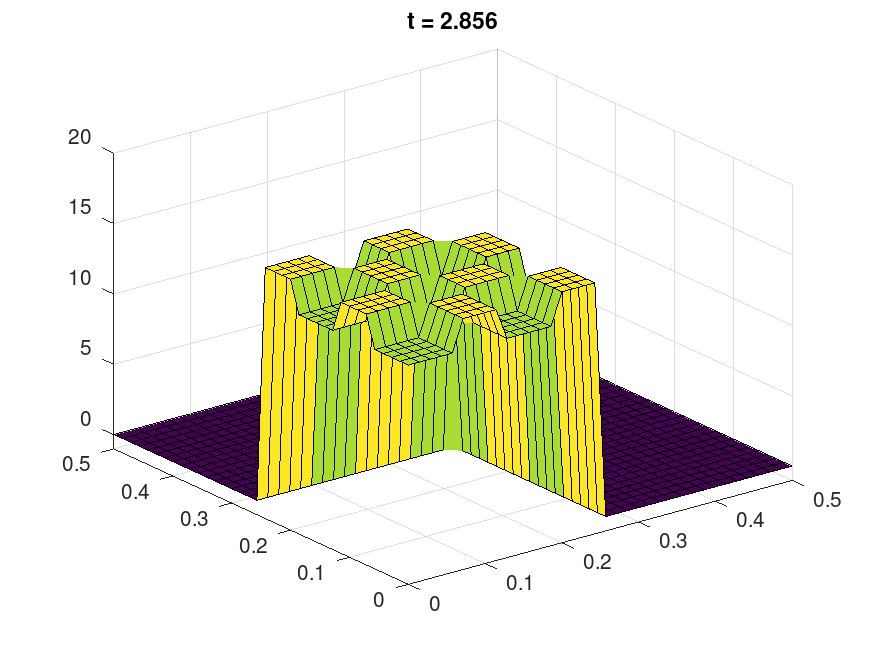} & \includegraphics[width=5cm]{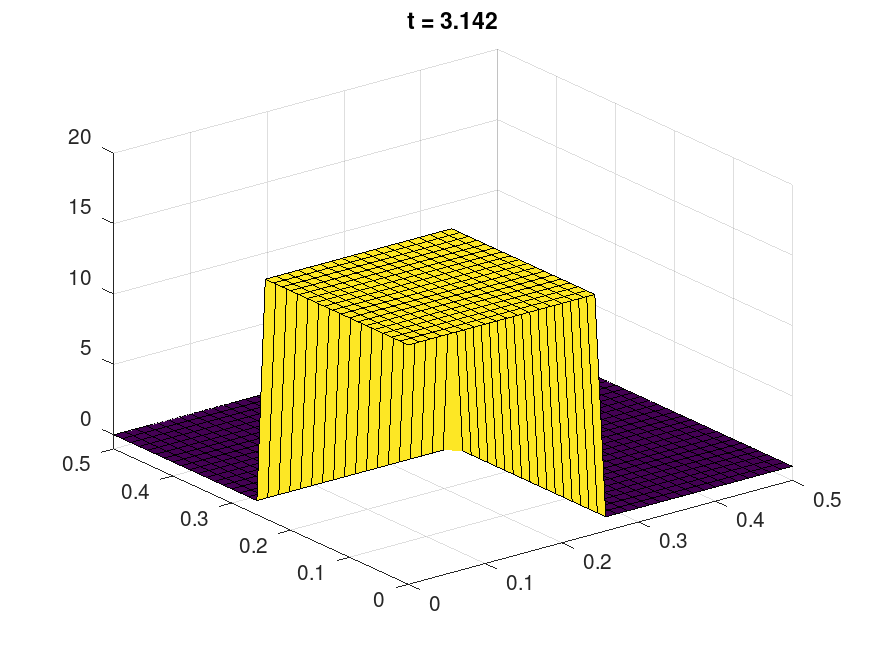} \\
	\hline
\end{tabular}
\caption{$g_{0 3}$}\label{fig:figure633}
\end{figure}

We may also have a global view of this Fisher-Riemann transport by looking at the dynamics of mean and variance. See Figure~\ref{fig:figure64}, page~\pageref{fig:figure64}.
\begin{figure}
\begin{tabular}{c c c c}
	\includegraphics[width=6cm]{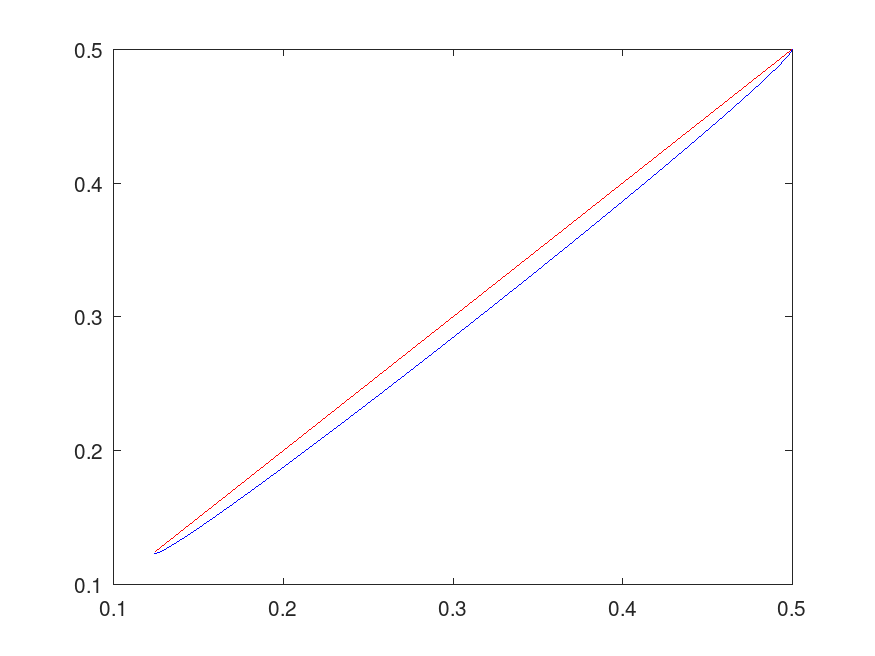} &
 	\includegraphics[width=6cm]{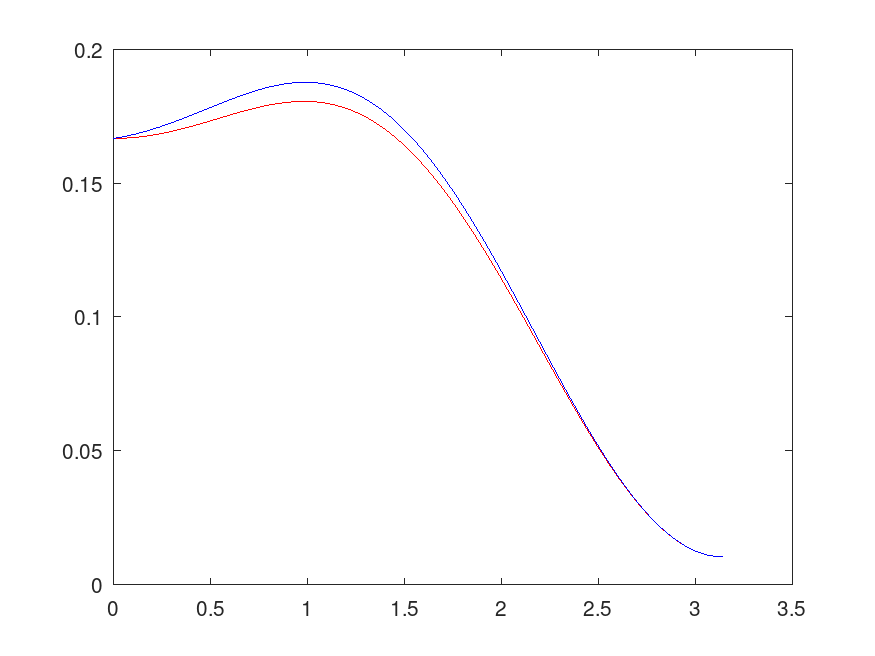}
 \end{tabular}
	\caption{Mean (left) and variance (right) of the trajectories in the 2-dimensional case. The blue curve is obtained with initial velocity $g_{0 1}$ while the red curve is obtained for both $g_{0 2}$ and $g_{0 3}$. }\label{fig:figure64}
\end{figure}

\newpage


\providecommand{\bysame}{\leavevmode\hbox to3em{\hrulefill}\thinspace}
\providecommand{\MR}{\relax\ifhmode\unskip\space\fi MR }
\providecommand{\MRhref}[2]{%
	\href{http://www.ams.org/mathscinet-getitem?mr=#1}{#2}
}
\providecommand{\href}[2]{#2}

\section*{Declarations}

\subsection*{Acknowledgements}
This work was supported by Consejo Nacional de Investigaciones Cient\'ificas y T\'ecnicas - CONICET and Universidad Nacional del Litoral - UNL in Argentina.

\subsection*{Funding}
Consejo Nacional de Investigaciones Cient\'ificas y T\'ecnicas, grant PIP-2021-2023-11220200101940CO.

\subsection*{Conflicts of interest/Competing interests}
The authors have no conflicts of interest to declare that are relevant to the content of this article.

\subsection*{Availability of data and material}
Not applicable.



\bigskip

%

\noindent{\footnotesize
	\noindent\textit{Affiliation 1.\,}
	\textsc{Instituto de Matem\'{a}tica Aplicada del Litoral ``Dra. Eleonor Harboure'', UNL, CONICET.}

	\smallskip
	\noindent\textit{Address.\,}\textmd{CCT CONICET Santa Fe, Predio ``Dr. Alberto Cassano'', Colectora Ruta Nac.~168 km 0, Paraje El Pozo, S3007ABA Santa Fe, Argentina.}

	\smallskip
	\noindent \textit{E-mail address.\, }\verb|haimar@santafe-conicet.gov.ar|; \verb|ivanagomez@santafe-conicet.gov.ar| 
}

\bigskip

\noindent{\footnotesize
	\noindent\textit{Affiliation 2.\,}
	\textsc{Facultad de Ingenier\'ia Qu\'imica, UNL, CONICET.}

	\smallskip
	\noindent\textit{Address.\,}\textmd{Santiago del Estero 2829, 3000 Santa Fe, Argentina.}

	\smallskip
	\noindent \textit{E-mail address.\, }\verb|anibalchicco@gmail.com|
}
\end{document}